\documentclass[12pt, reqno]{amsart}
\usepackage{mathrsfs}
\usepackage{amsfonts,amsmath,dsfont,amssymb,amsthm,stmaryrd,bbm,color,comment}
\usepackage[hidelinks]{hyperref}
\usepackage{appendix}
\usepackage[pdftex]{graphicx}
\usepackage{bbm}
\usepackage{array}
\newcolumntype{L}{>{\centering\arraybackslash}m{2.5cm}}
\usepackage{multirow}
\usepackage{xcolor}
\usepackage[percent]{overpic}
\usepackage{algorithm}
\usepackage{algpseudocode}
\usepackage{mathtools}

\newtheorem{thm}{Theorem}[section]
\newtheorem{lem}[thm]{Lemma}
\newtheorem{prop}[thm]{Proposition}
\newtheorem{coro}[thm]{Corollary}

\theoremstyle{definition}
\newtheorem{defn}[thm]{Definition}

\theoremstyle{remark}
\newtheorem{rmk}[thm]{Remark}
\numberwithin{equation}{section}

\newcommand{\Var}{\mathrm{Var}}

\newcommand{\R}{\mathbb{R}}

\newcommand{\Z}{\mathbb{Z}}

\newcommand{\E}{\mathbb{E}}

\def\P{{\mathbb P}}

\begin{document}
\setcounter{page}{1}

\color{black}{
\centerline{}
\centerline{}
\title[CLT in R\'enyi divergence for lattice random variables]{Central limit theorem in R\'enyi divergence for lattice random variables}
\author[Zhen Fu, and Jiange Li]{Zhen Fu, and Jiange Li} 
\address{(Z.F, J.L) Institute for Advanced Study in Mathematics, Harbin Institute of Technology, China}
\email{zhenfu@stu.hit.edu.cn, jiange.li@hit.edu.cn}

\begin{abstract}

Let $\{X_i\}_{i=1}^n$ be independent, identically distributed lattice random variables  with zero mean, unit variance, and maximal span $h>0$. Let $S_n=(X_1+\cdots+X_n)/\sqrt n$. Let $Z_n$ denote the quantization of a standard Gaussian random variable on the support lattice of $S_n$. We show that the R\'enyi divergence $D_\alpha(S_n\|Z_n)$ of order $1<\alpha<\infty$ tends to 0 as $n\to\infty$ if and only if $D_\alpha(S_{n_0}\| Z_{n_0})<\infty$ for some integer $n_0$ and a sub-Gaussian condition holds. Furthermore, we derive an Edgeworth-type asymptotic expansion of the divergence to arbitrary order. We also show that the R\'enyi divergence $D_\infty(S_n\|Z_n)$ of infinite order tends to 0 as $n\to\infty$ if and only if a sub-Gaussian condition, a curvature condition, and a uniform separation condition hold. 

These results constitute lattice counterparts of Rényi entropic central limit theorems for continuous random variables obtained by Bobkov, Chistyakov and G\"{o}tze (\emph{Ann. Probab.} \textbf{47} (2019), 270--323) and Bobkov and G\"{o}tze (\emph{Ann. Probab.} \textbf{53} (2025), 453--477). Uniform separation is an additional assumption beyond the sub-Gaussian and curvature conditions used in the continuous setting. The failure of uniform separation may cause an exponentially tilted version of $X_1$ to concentrate on a sublattice of max span $kh$ for some integer $k>1$, which yields the lower bound $D_\infty(S_n\|Z_n)\ge\log k$, and hence a failure of the entropic central limit theorem.


\end{abstract} 

\maketitle
\tableofcontents

\begingroup
\let\raggedright\centering
\part{Introduction and preliminary tools}
\endgroup


\section{Introduction}

\subsection{Background}

Let $X_1,  \cdots, X_n$ be independent and identically distributed (i.i.d.) random variables with zero mean and unit variance. The classical central limit theorem (CLT for short) asserts that the normalized sum
\begin{equation*}
S_n=\frac{X_1+\cdots+X_n}{\sqrt{n}}
\end{equation*}
converges weakly to a standard Gaussian random variable $Z$. An information-theoretic proof of the CLT, even for non-identically distributed summands, was first given by Linnik \cite{Lin59}. Subsequently, Barron \cite{Bar86} established the entropic CLT in the i.i.d. setting: the Kullback–Leibler (KL for short) divergence $D(S_n\| Z)$ tends to 0 as $n\to\infty$ provided that $D(S_{n_0}\| Z)<\infty$ for some integer $n_0$. By Pinsker's inequality, Barron's entropic CLT implies the classical version, thereby offering a stronger mode of convergence. The KL divergence is intimately linked to other fundamental metrics. For instance, Talagrand's $W_2$ inequality relates it to the quadratic Wasserstein distance, while the logarithmic Sobolev inequality for the standard Gaussian measure connects it to the relative Fisher information. Further developments of the entropic CLT include the monotonicity of $D(S_n\| Z)$ \cite{ABBN04-a}, quantitative convergence under spectral gap conditions \cite{ABBN04-b, JB04}, Edgeworth-type expansion of $D(S_n\| Z)$ \cite{BCG13}, and extensions to non-identically distributed random variables \cite{Joh00, BCG14-a}.

Let $p_n(x)$ be the density of $S_n$.  For $\alpha>0$, $\alpha\neq1$, the R\'enyi divergence $D_\alpha(S_n\| Z)$ of order $\alpha$ of $S_n$ from $Z$ is defined by
$$
D_\alpha(S_n\|Z) =\frac{1}{\alpha-1}\log\int_{\mathbb R} \frac{p_n(x)^\alpha}{\varphi(x)^{\alpha-1}}dx,
$$
where $\varphi(x)=\frac{1}{\sqrt{2\pi}}e^{-x^2/2}$ is the standard Gaussian density. Its limit at $\alpha=1$ recovers the KL divergence (and the subindex 1 is typically omitted). The infinite order R\'enyi divergence $D_\infty(S_n\| Z)$ is given by the limit
$$
D_\infty(S_n\| Z)=\lim_{\alpha \to \infty}D_\alpha(S_n\|Z)= \log \left(\operatorname*{ess\,sup}_{x\in \mathbb{R}}\frac{p_{n}(x)}{\varphi(x)}\right).
$$
The mapping $\alpha\mapsto D_\alpha(S_n\|Z)$ is nondecreasing. Therefore, the R\'enyi divergence of order $\alpha>1$ provides a scale of convergence notations that are stronger than the KL divergence. For $\alpha>1$, the weight $\varphi(x)^{1-\alpha}$ grows exponentially at a Gaussian rate, so the convergence of $D_\alpha(S_n\|Z)$ is sensitive not only to the bulk of the distribution but also to moderate and large deviations of $p_n(x)$. Consequently, the convergence of $D(S_n\| Z)$ requires only general moment conditions, whereas the convergence of $D_\alpha(S_n\| Z)$ for $\alpha>1$ necessitates sub-Gaussianity and hence finiteness of all moments. For $\alpha=\infty$, the situation is even more delicate. Unlike the finite-order case, where an integral averaging effect is retained, $D_{\infty}(S_n\| Z)$ is determined by the essential supremum of $p_{n}(x)/\varphi(x)$, and therefore, requires uniform upper control of the probability ratio. Hence, local peaks that are negligible in finite-order R\'enyi integrals may still prevent convergence at $\alpha=\infty$, and this leads to stronger tail and local regularity conditions. These distinctions underly the CLT in Rényi divergence. The finite-order case was established by Bobkov, Chistyakov and G\"{o}tze \cite{BCG19}, whereas the infinite-order case was subsequently studied in \cite{BG25-a, BG25-b}. For a comprehensive overview, we refer the reader to the recent survey \cite{BCG25}.

The present paper investigates the R\'enyi entropic CLT for lattice random variables. A random variable $X$ is said to be a \textit{lattice random variable} if there exists $a\in \mathbb R$ and $h>0$ such that $X$ is supported on the lattice $a+h\mathbb Z:=\{a+kh: k \in \mathbb Z\}$; the largest such $h>0$ is called the \textit{maximal span} of $X$. Let $X_1, \cdots, X_n$ be i.i.d. lattice random variables with zero mean, unit variance and max span $h>0$. The normalized sum $S_n$ is a lattice random variable with maximal span $\delta_n=h/\sqrt n$ supported on 
$$
\mathcal{L}_n=\{x_{n, k}\}_{k\in\mathbb Z}, \quad x_{n,k}=\frac{na+kh}{\sqrt n}.
$$
We write $p_{n, k}:=\mathbb P(S_n=x_{n, k})$. In this discrete setting, the standard Gaussian distribution itself cannot serve directly as the reference measure, since the divergence of $S_n$ from the continuous Gaussian distribution is infinite. Instead, we compare $S_n$ with the quantized Gaussian random variable $Z_n$ supported on $\mathcal{L}_n$, defined by
\begin{equation}\label{eq:pmf-zn}
\mathbb P(Z_n=x_{n,k})=\frac{q_{n, k}}{\sum_{j\in\Z}q_{n, j}}, \quad q_{n, k}=\delta_n\varphi(x_{n,k}).
\end{equation}
The R\'enyi divergence $D_{\alpha}(S_n\|Z_n)$ of order $\alpha$ of $S_n$ from $Z_n$ is defined by
\begin{equation}\label{eq:R-div}
D_\alpha(S_n\|Z_n)=\frac{1}{\alpha-1}\log \sum_{k \in \mathbb{Z}} \frac{\mathbb P(S_n=x_{n,k})^\alpha}{\mathbb P(Z_n=x_{n,k})^{\alpha-1}}.
\end{equation}
In particular, the infinite order R\'enyi divergence $D_{\infty}(S_n\|Z_n)$ is defined by 
\begin{equation}\label{eq:Inifinte R-div}D_{\infty}(S_n\|Z_n)=\log \sup_{k\in\mathbb{Z}} \frac{\mathbb P(S_n=x_{n,k})}{\mathbb P(Z_n=x_{n,k})}.
\end{equation}

To the best of our knowledge, results on the convergence of $D_\alpha(S_n\|Z_n)$ for lattice random variables are scarce. Takano~\cite{Tak87} established the KL divergence $D(S_n\|Z_n)$ tends to 0 with an almost $O(1/\sqrt n)$ rate under suitable moment assumptions. More recently, Gavalakis and Kontoyiannis~\cite{GK24} reproved the convergence $D(S_n\|Z_n)\to 0$ by transforming the lattice problem into the continuous setting through the addition of independent uniform random variables, thereby invoking Barron's entropic CLT.


\subsection{Main results}
The Laplace transform and cumulant generating function of $X$ are, respectively, given by 
$$
L(t)=\E e^{tX}, \quad \quad K(t)=\log L(t).
$$  
Write 
$$
A(t)=\frac{t^2}{2}-K(t),\quad t\in\R.
$$ 
We will show in Lemma \ref{lem:entire-mgf} that, if $L(t)<\infty$ for all $t\in \mathbb{R}$, then $L$ extends to an entire function on $\mathbb{C}$.

The following result gives a necessary and sufficient condition for the convergence of normalized sums $S_n$ to quantized Gaussian $Z_n$ in R\'enyi divergence of order $\alpha>1$: Finiteness of the divergence at some convolution level together with a strict sub-Gaussian bound on the Laplace transform. This constitutes the lattice counterpart of the R\'enyi entropic CLT due to Bobkov, Chistyakov and G\"{o}tze \cite{BCG19}.  
\begin{thm}\label{thm:main}
Let $\alpha >1$ and $\beta=\alpha/(\alpha-1)$. Let $X_{1},\cdots,X_{n}$ be independent copies of a lattice random variable $X$ with zero mean, unit variance, and maximal span $h>0$. Then $$
D_{\alpha}(S_{n}||Z_{n})\to 0 \quad \text{as}~n\to\infty
$$  
if and only if the following conditions are fulfilled:
\begin{flalign*}
 \textit{(1)}~ D_{\alpha}(S_{n_0}||Z_{n_0}) < \infty ~\text{for some}~ n_0\in\mathbb Z;&&
\end{flalign*}
\begin{flalign*}
\textit{(2)}~ L(t)<e^{\beta t^2/2}~\text{for all}~t \in \mathbb R,~t\neq0.&&
\end{flalign*}
Moreover, under the above two conditions, for any given integer $s\geq 3$, we have the following expansion
\begin{align*}
D_{\alpha}(S_{n}||Z_{n})=\frac{1}{\alpha-1} \sum_{j=1}^{\lfloor s/2-1 \rfloor}b_{j}n^{-j}+o\big(n^{-\frac{s-2}{2}}\big),
\end{align*}
where the coefficients $b_j$'s represent certain polynomials of the cumulants of $X$. 

\end{thm}

The following result gives a necessary and sufficient condition for the convergence of normalized sums $S_n$ to quantized Gaussian $Z_n$ in R\'enyi divergence of infinite order: A sub-Gaussian condition (SG), a curvature condition (C), and a uniform separation condition (US). This constitutes the lattice counterpart of the R\'enyi entropic CLT due to Bobkov and G\"{o}tze \cite{BG25-a}.

\begin{thm}\label{thm:main1}
Let $X_1,\ldots,X_n$ be independent copies of a lattice random variable $X$ with zero mean, unit variance, and maximal span $h>0$. Then 
$$
D_\infty(S_n\|Z_n)\to0 \quad \text{as} \quad n\to \infty
$$
if and only if the following conditions are fulfilled:
\begin{flalign}
\textit{(1)}~ L(t)\le e^{t^2/2} ~\text{for all}~ t\in\mathbb{R};&&
\tag{SG}\label{eq:SG}
\end{flalign}
\begin{flalign}
\textit{(2)}~ \lim_{A(t)\to0}(A''(t))_+=0;&&
\tag{C}\label{eq:C}
\end{flalign}
\textit{(3)}~ for every $\eta\in(0,\pi/h)$, there exist $\delta>0$ and $\rho<1$ such that
\begin{equation}
\sup_{A(t)\le\delta}\sup_{\eta\le |u|\le\pi/h}
\left|\frac{L(t+iu)}{L(t)}\right|\le\rho.
\tag{US}\label{eq:U}
\end{equation}
\end{thm}

\begin{thm}
\label{thm:obstruction}
There exists a symmetric lattice random variable $X$ with mean zero, variance one
and maximal span one such that the following properties hold:
\begin{flalign*}
\textrm{(1)}& ~L(t)<e^{t^2/2}~ \text{for all}~t\ne0, &&
\\[5pt]
\textrm{(2)}& ~\lim_{A(t)\to 0}|A''(t)|=0.&&
\end{flalign*}
Nevertheless, for every $\eta\in(0,\pi)$ and every $\delta>0$,
\begin{equation*}
\sup_{A(t)\le\delta}\sup_{\eta\le |u|\le\pi/h}\frac{|L(t+iu)|}{L(t)}=1.
\end{equation*}
Let $X_1,\cdots, X_n$ be independent copies of $X$, and set $S_n=(X_1+\cdots+X_n)/\sqrt n$. Let $Z_n$ denote the corresponding quantized Gaussian random variable. We have 
\begin{equation}\label{eq:ce-divergence}
\liminf_{n\to\infty}D_\infty(S_n\Vert Z_n)\ge\log2.
\end{equation}
\end{thm}

\begin{rmk}\label{rem:Th1.2}
(1)~Condition \eqref{eq:SG} in Theorem \ref{thm:main1} inherits from Theorem \ref{thm:main} by taking $\alpha\to\infty$. It implies that every $\Gamma_n:=\sup_{k\in \mathbb{Z}}p_{n,k}/q_{n,k}$ is finite. For $k\in\Z$ and $t\in\R$, positivity and independence give
$$
p_{n,k}e^{t\sqrt n\,x_{n,k}}\leq \mathbb{E} e^{t\sqrt{n}S_{n}}= L(t)^n\le e^{nt^2/2}.
$$
Setting $t=x_{n,k}/\sqrt n$, we obtain
\begin{equation}\label{eq:Gamma-finite}
p_{n,k}\le e^{-x_{n,k}^2/2},\qquad \Gamma_n\le\sqrt{2\pi n}/h<\infty.
\end{equation}
Therefore, the divergence $D_{\infty}(S_{n}\|Z_{n})=\log \Gamma_{n}+\log \sum_{j \in\mathbb{Z}}q_{n,j}$ is finite for every $n$. Thus no separate assumption that $D_\infty(S_{n_0}\|Z_{n_0})<\infty$ for some $n_0$ is needed in Theorem \ref{thm:main1}. This is somewhat different from the continuous setting.

(2) The function $A(t)$ may be periodic and have infinitely many zeros. Condition \eqref{eq:C} exhibits kind of continuity of $A''(t)$ with respect to $A(t)$. Under \eqref{eq:SG}, condition \eqref{eq:C} is equivalent to the following two conditions
\begin{equation}
A(t)=0\quad\Longrightarrow\quad A''(t)=0,
\tag{C1}\label{eq:C1}
\end{equation}
\begin{equation}
A(t_j)\to0, ~~ |t_j|\to\infty	~\Longrightarrow\quad \limsup_{j\to\infty}A''(t_j)\le0.
\tag{C2}\label{eq:C2}
\end{equation}
This equivalence will be established in Remark \ref{rem:curvature}.

(3) Condition \eqref{eq:U} is an additional assumption beyond the sub-Gaussian and curvature conditions used in the continuous setting. As we will see, $L(t+iu)/L(t)$ is the characteristic function of an exponentially tilted random variable. Condition \eqref{eq:U} is crucial for establishing upper bound of the $n$-th moment of $L(t+iu)/L(t)$ uniform with respect to $t$ such that $A(t)<\delta$ (inequality \eqref{eq:second-term} in Proposition \ref{prop:uniform-integral-estimate}). Under exponential tilting, a lattice distribution may become asymptotically concentrated on a proper sub-lattice, producing Fourier peaks at nonzero frequencies with modulus close to one; the example in Section \ref{sec:failure US} indeed reflects this phenomenon. In that case, the $n$-th moment of $L(t+iu)/L(t)$ for certain $t$ such that $A(t)<\delta$ would be bounded away from zero.
The maximal span alone does not exclude such concentration uniformly, so condition \eqref{eq:U} is needed to rule it out. In the continuous case, condition \eqref{eq:U} is guaranteed by the sub-Gaussian condition  (see Corollary 2 in \cite{BCG23}). The distribution constructed in Section \ref{sec:example} satisfies conditions \eqref{eq:SG} and \eqref{eq:C}, yet fails \eqref{eq:U}. This example therefore shows that \eqref{eq:U} cannot be replaced by the two analytic conditions \eqref{eq:SG} and \eqref{eq:C}; it is an intrinsic feature of the lattice infinite-order R\'enyi CLT, rather than a technical assumption.  
\end{rmk}

\begin{rmk}
R\'enyi divergence is closely connected to relative Tsallis entropy (see Section \ref{Renyi divergence}). Moreover, they are of the same order when they are small. Therefore, the above results also hold for relative Tsallis entropy.
\end{rmk}


\subsection{Proof strategies}
\subsubsection{Proof strategy of Theorem \ref{thm:main}} We first explain the sufficiency part. As we will see later, the normalizing denominator $\sum_{j\in \mathbb Z}q_{n,j}$ in equation \eqref{eq:pmf-zn} tends to one as $n\to \infty$. Hence, by definition \eqref{eq:R-div}, the convergence of $D_{\alpha}(S_{n}||Z_{n})\to 0$ as $n\to\infty$ is equivalent to 
$$
\sum_{k\in\mathbb Z}\left(\frac{p_{n, k}}{q_{n, k}}\right)^\alpha q_{n, k}\to 1 \quad \text{as}~n\to\infty.
$$ 
We split the summation into a central part and a tail part. The central region consists of $x_{n, k}=O(\sqrt {\log n})$, and under general moment conditions, the classical local limit theorem (Proposition \ref{prop:llt}) gives the uniform approximation (with respect to $k$)
$$
\frac{p_{n, k}}{q_{n, k}}=1+o(1),
$$
where the remainder term has polynomial decay. This shows that the central part is $1+o(1)$ (see Proposition \ref{prop:bulk-expansion}).

The main difficulty is the tail part, since the local limit theorem no longer provides effective information about $p_{n, k}$. However, the finiteness of $D_\alpha(S_{n_0}\| Z_{n_0})$ for some $n_0$ yields the sub-Gaussianity of $X$ (see Proposition \ref{prop:sub-Gaussian-a}). This property, combined with tools from complex and Fourier analysis, allows us to establish a nonuniform pointwise bound of the form (see Proposition \ref{prop:moderate-large deviation}):
\begin{equation}\label{eq:tail-deviation}
p_{n,k}\leq C\exp\left(-\frac{x_{n,k}^{2}}{2\beta}\right)\psi\left(\frac{x_{n,k}}{\beta\sqrt{n}}\right)^{n-O(1)},\quad k\in\mathbb{Z}.
\end{equation}
Here, $\psi(t)=\mathbb E e^{tX} e^{-\beta t^2/2}$, which satisfies that $0<\psi(t)<1$ for $t\in\mathbb R,~t\ne 0$. The factor 
$$
\psi\left(\frac{x_{n,k}}{\beta\sqrt{n}}\right)^{n-O(1)}
$$ 
exhibits polynomial decay for $x_{n, k}$ of order $\sqrt{\log n}$, and exponential decay for $x_{n, k}$ of order $\sqrt n$. This compensates for the growing Gaussian weight $q_{n,k}^{1-\alpha}$, and consequently, we can show that the tail concentration is $o(1)$.

For the necessity part, the finiteness condition $D_\alpha(S_{n_0}\| Z_{n_0})<\infty$ follows immediately from the convergence of $D_{\alpha}(S_{n}\|Z_{n})\to 0$ as $n\to\infty$. As noted earlier, the finiteness of R\'enyi divergence implies the sub-Gaussianity of $X$, i.e.,  
$$
\E e^{tX}\leq e^{\beta t^2/2} \quad \text{for all}~t\in\mathbb R.
$$
It remains to exclude equality at a nonzero point. If equality held at some $t_0\neq0$, one could form exponential tilts of the law of $S_n$ and of the
quantized Gaussian $Z_n$. R\'enyi convergence implies that the corresponding pair of tilted measures become close in total variation. On the other hand, the
two quantized Gaussian measures involved have centers separated by order $\sqrt n$, so their total variation distance tends to its maximal value. The
triangle inequality then gives a contradiction. This proves the strict Laplace-transform condition.

\subsubsection{Proof strategy of Theorem \ref{thm:main1}} By definition \eqref{eq:Inifinte R-div}, the convergence of $D_{\infty}(S_{n}||Z_{n})\to 0$ as $n\to\infty$ is equivalent to 
$$
\Gamma_n=\sup_{k\in\mathbb Z}\frac{p_{n, k}}{q_{n, k}}\to 1 \quad \text{as}~n\to\infty.
$$ 
The proof of sufficiency starts from the shifted Fourier inversion formula \eqref{eq:inversion 1}, from which we derive the following probability ratio (see Lemma \ref{lem:p/q-ratio})
$$
\frac{p_{n,k}}{q_{n,k}}=\sqrt{\frac{n}{2\pi}}\,e^{-nA(x_{n,k}/\sqrt n)}I_n\left(\frac{x_{n,k}}{\sqrt n}\right), \quad
I_n(t)=\frac{h}{2\pi}\int_{-\pi/h}^{\pi/h}e^{-intu}\left[\frac{L(t+iu)}{L(t)}\right]^n\,du.
$$
 This representation separates the exponential decay governed by $A(t)$ from the local behavior of the complex Laplace transform $I_n(t)$. 
 When $A(t)$ is bounded away from zero, the bound $|I_n(t)|\le 1$, together with the factor $e^{-nA(t)}$, makes the probability ratio uniformly negligible as $n\to\infty$. Therefore, the main task is to control the integral $I_n(t)$ when $A(t)$ is small. To this end, we first rewrite 
$$
I_n(t)=\frac{h}{2\pi}\int_{-\pi/h}^{\pi/h}\left[e^{-iuA'(u)}\psi_t(u)\right]^n\,du,
$$
where $\psi_t$ is the characteristic function of an exponentially tilted random variable of mean zero and variance $K''(t)$. Therefore, it is reasonable to approximate $I_n(t)$ by the Gaussian integral 
$$
\int_{\mathbb R}e^{-inuA'(t)}e^{-nK''(t)u^2/2}\,du.
$$
(Note $e^{-K''(t)u^2/2}$ is the characteristic function  of a centered Gaussian random variable of variance $K''(t)$.) 
Under conditions \eqref{eq:SG}, \eqref{eq:C}, and \eqref{eq:U}, we prove in Proposition \ref{prop:uniform-integral-estimate} that
\begin{equation}\label{eq:p/q}
\frac{p_{n,k}}{q_{n,k}}	=\frac{\exp\left[-nA(x_{n,k}/\sqrt n)-\dfrac{nA'(x_{n,k}/\sqrt n)^2}{2K''(x_{n,k}/\sqrt n)}	\right]}{\sqrt{K''(x_{n,k}/\sqrt n)}}+O\left(\frac{1}{\sqrt n}\right),
\end{equation}
uniformly over $k$ such that $A(x_{n,k}/\sqrt n)$ is sufficiently small. Since the exponential factor in the leading term is at most one, condition \eqref{eq:C} ensures that the prefactor $1/\sqrt{K''(t)}$ is arbitrarily close to one from above when $A(t)$ is sufficiently small. Therefore, the ratio $p_{n, k}/q_{n, k}$ is uniformly bounded above by a quantity tending to one as $n\to\infty$. Together with the exponential decay obtained when $A(t)$ is bounded away from zero, this yields the desired upper bound for $\Gamma_n$. The lower bound for $\Gamma_n$ follows from the simple fact $1=\sum_{k\in \mathbb{Z}}p_{n,k}\le \Gamma_n \sum_{k\in \mathbb{Z}}q_{n,k}$ and $\sum_{k\in \mathbb{Z}}q_{n,k}\to1$. Thus, we have $\Gamma_n\to1$ as $n\to\infty$.

For the necessity part, we use a separate argument rather than reversing the preceding asymptotic calculation. Since $p_{n,k}\le \Gamma_n q_{n,k}$ and $\Gamma_n\to1$, summing the inequality against the exponential weight $e^{t\sqrt n\,x_{n,k}}$ and passing to the limit gives $L(t)\le e^{t^2/2}$, hence \eqref{eq:SG}. To prove the necessity of \eqref{eq:C} and \eqref{eq:U}, we test $\Gamma_{n}\geq p_{n,k}/ q_{n,k}$ against
$$
f(x)=1+\operatorname{Re}\big(e^{i(\sqrt nux-\theta)}\big),
$$
where $\theta$ is chosen to recover $|L(t+iu)|^n$. Lemma \ref{lem:test} then gives
\begin{equation} \label{eq:low bound}
\Gamma_n\ge\frac{e^{-nA(t)}\bigl(1+|L(t+iu)/L(t)|^n\bigr)}{G_n(t,0)+|G_n(t,u)|},
\quad 
G_n(t,u)=\delta_n\sum_{k\in\mathbb Z}
\phi(x_{n,k}-\sqrt n\,t)e^{i\sqrt nux_{n,k}}.
\end{equation} 
We then prove the necessity of \eqref{eq:C} and \eqref{eq:U} separately by contradiction, using different frequency scales. 
\begin{itemize}
\item If \eqref{eq:C} fails, we choose $t=t_n$ with $A(t_n)\le n^{-2}$ and $K''(t_n)\le1-\varepsilon$ for some fixed $\varepsilon>0$. Taking $u_n=u/\sqrt n$ for some fixed $u\neq 0$, the Gaussian approximation of $L(t_n+iu_n)/L(t_n)$, together with the corresponding Gaussian estimate for $G_n(t_n,u_n)$, makes the lower bound \eqref{eq:low bound} strictly larger than one in the limit, contradicting $\Gamma_n\to1$. Hence \eqref{eq:C} is necessary. 
\item If \eqref{eq:U} fails, we instead choose $t_j,u_j$ with $A(t_j)\to0$, $|L(t_j+iu_j)|/L(t_j)\to1$, and $|u_j|$ bounded away from zero in the fundamental interval. For each fixed $n$, letting $j\to\infty$ makes the numerator in \eqref{eq:low bound} tend to $2$, while the nonzero-frequency Gaussian term is exponentially small in $n$. As $n\to\infty$, it gives
$$
\liminf_{n\to\infty}D_\infty(S_n\|Z_n)\ge\log2.
$$
Thus \eqref{eq:U} is necessary.
\end{itemize}

\subsection{Comparison with related works}
A standard route to the entropic CLT is to prove a suitable Fisher information inequality and then to integrate it using de Bruijn's identity along the heat semigroup. This strategy has been employed in a number of works, including Barron \cite{Bar86}, Johnson \cite{Joh00}, Arstein, Ball, Barthe and Naor \cite{ABBN04-a, ABBN04-b}, Johnson and Barron \cite{JB04}. In the discrete setting, however, this method faces fundamental difficulties due to the absence of both Fisher information and de Bruijn's identity. An alternative approach, developed mainly by Bobkov, Chistyakov and G\"{o}tze \cite{BCG13, BCG14-a, BCG14-b, BCG19, BG25-a, BG25-b}, proceeds by rewriting the strong distance as a functional of the ratio between the law of normalized sum and the reference measure. The estimate is then split into a central part and a tail part: the central part is controlled via a Edgeworth-type local limit theorem, while the tail is handled by a distance-specific tail method. This approach applies to both continuous and lattice distributions and is particularly useful for obtaining refined convergence rates and asymptotic expansions.

\subsubsection{Finite-Order R\'enyi CLT} Our proof of the sufficiency part of Theorem \ref{thm:main} follows the same general strategy as \cite{BCG19}, namely the central–tail decomposition outlined in the previous subsection.

\begin{enumerate}
\item In the central region, both our paper and \cite{BCG19} employ an Edgeworth-type local limit theorem together with Taylor expansion. In addition, we use the Poisson summation formula to derive a series of estimates for quantized Gaussian moments, relating these sums to the corresponding Gaussian integrals with exponentially small errors. Consequently, the lattice structure contributes exponentially small errors and does not affect the polynomial-order expansion of the main term.

\item The key tool in the tail analysis is the pointwise bound \eqref{eq:tail-deviation}, which is analogous to Proposition 13.1 of \cite{BCG19}. Both proofs rely on a shifted Fourier inversion formula. In the continuous case, Fourier inversion is over $\mathbb R$, and the contour is shifted using decay on the vertical sides. In the discrete setting, by contrast, we employ lattice Fourier analysis and exploit the periodicity of the Fourier inversion integrand induced by the lattice structure to shift the contour over the fundamental interval $[-\pi/\delta_n, \pi/\delta_n]$ (Lemmas \ref{lem:phi-phi_n-extension}, \ref{lem:periodicity of F_n} and \ref{lem:+iy}). This periodicity causes the two vertical contour integrals to cancel exactly. The argument relies on (enhanced) sub-Gaussianity, which is facilitated by introducing a discrete Gaussian smoothing operator, analogous to Weierstrass transform used in \cite{BCG19}.
\end{enumerate}

Our proof of the necessity direction differs significantly from that of \cite{BCG19}. Both proofs first reduce the problem to showing that equality in the sub-Gaussian condition cannot occur at any nonzero point. In \cite{BCG19}, equality is excluded using high-power Laplace-transform estimates and local analyticity. Our proof, instead, uses the sequential stability of H\"{o}lder's inequality together with total-variation estimates for three probability measures, leading to a contradiction with the triangle inequality. This argument is conceptually different, arguably simpler, and potentially portable back to the continuous setting (since the argument is essentially independent of the lattice structure).

\subsubsection{Infinite-Order R\'enyi CLT} Theorem \ref{thm:main1} is closely related to its continuous counterpart in \cite{BG25-a}. The problem in both settings is to control the supremum ratio between the law of normalized sum and the corresponding Gaussian reference measure. The lattice setting, however, introduces a genuinely new obstruction, which leads to a different form of the final characterization (i.e., the uniform separation condition \eqref{eq:U}). 

In both settings, the ratio between the law of normalized sum and the corresponding Gaussian reference measure consists of two parts: an exponential term that governs the asymptotic behavior when $A(t)$ is bounded away from zero, and an approximated Gaussian term that encodes the local behavior when $A(t)$ is small (see Lemma \ref{lem:p/q-ratio} of the current paper and Proposition 3.3 of \cite{BG25-a}). The main task is to establish a Gaussian approximation for the second term. 

The key estimate in \cite{BG25-a} is a uniform local limit theorem (Lemma 6.1), which is applied to exponentially tilted random variables to prove Lemma 7.1 (counterpart of our Proposition \ref{prop:uniform-integral-estimate}). A crucial technical ingredient for establishing such an estimate is the $L^p$ bound for characteristic functions in the high frequency region. The validity of such an $L^p$ bound crucially relies on that the modulus of the characteristic function of an absolutely continuous random variable is strictly less than one in any compact set not containing the origin (Proposition 4, \cite{BCG23}). However, the exponential tilting of a lattice distribution may become asymptotically concentrated on a proper sub-lattice, producing Fourier peaks at nonzero frequencies with modulus close to one. Hence, we need to impose the uniform separation condition \eqref{eq:U} to rule this out. Our argument for the low frequency region is technically different, but similar in spirit, in comparison with \cite{BG25-a}. Also, in the lattice setting, condition \eqref{eq:SG} itself ensures that $D_\infty(S_n\| Z)<\infty$ for every $n$, so no separate convolution-level finiteness assumption is required.

Our proof of the necessity differs more substantially from that of \cite{BG25-a}. Bobkov and G\"otze use the local asymptotic formula from the sufficiency part to show that, under convergence, $A(t)\to0$ forces the tilted variance $K''(t)$ to approach one, yielding the curvature condition. Our necessity proof is independent of the asymptotic argument used for sufficiency. Instead, we employ a nonnegative exponential--cosine test to derive a lower bound for the supremum probability ratio. By choosing frequencies on two different scales, the same test detects failures of \eqref{eq:C} and \eqref{eq:U}, respectively.

The uniform separation phenomenon has no counterpart in the continuous setting. In the lattice setting, failure of \eqref{eq:U} yields the quantitative lower bound 
$$
\liminf_{n\to\infty}D_\infty(S_n\|Z_n)\ge \log 2.
$$ 
The example in Section \ref{sec:example} shows that the sub-Gaussian and curvature conditions alone do not suffice. Thus, the essential new feature of the lattice case is the uniform separation obstruction encoded by \eqref{eq:U}.


\subsection{Notations}
The following notations will be used throughout the paper. Let $X$ be a lattice random variable taking values in $a+h\mathbb Z$, with zero mean, unit variance, and maximal span $h>0$. Let $X_1, \cdots, X_n$ be independent copies of $X$. We define 
\begin{itemize}
\item The normalized sum: 
$$
S_{n}=\frac{X_{1}+\cdots+X_{n}}{\sqrt{n}}.
$$
\item The support of $S_n$: 
$$
\mathcal{L}_n=\{x_{n, k}\}_{k\in\mathbb Z},\quad \quad~x_{n,k}=\frac{na+kh}{\sqrt n}.
$$
\item The probability mass function of $S_n$: 
$$
p_{n,k}=\mathbb P(S_{n}=x_{n,k}).
$$
\item The maximal span of $S_{n}$: 
$$
\delta_n=\frac{h}{\sqrt{n}}.
$$
\item The Laplace transform and cumulant generating function of $X$ are, respectively given by 
$$
L(t)=\E e^{tX}, \quad \quad K(t)=\log L(t),\qquad t\in\R.
$$  
\item The Gaussian deficit function: 
$$
A(t)=\frac{t^2}{2}-K(t),\quad t\in\R.
$$ 
\item The characteristic functions of $X$ and $S_n$: 
$$
\phi(t)=\mathbb Ee^{itX},~\quad\phi_{n}(t)=\mathbb Ee^{itS_{n}}=\phi\left(\frac{t}{\sqrt{n}}\right)^{n},\quad t \in \mathbb{R}.
$$
\item The cumulants $\{\gamma_{j}\}_{j=1}^\infty$ of $X$, defined by   
$$
\gamma_{j}=i^{-j}(\text{log}\phi)^{(j)}(0).
$$
\item The standard Gaussian density: 
$$
\varphi(x)=\frac{1}{\sqrt{2\pi}}e^{-x^2/2}.
$$
\item A sequence $\{q_{n,k}\}_{k\in\mathbb Z}$ given by 
$$
q_{n,k}=\delta_{n}\varphi(x_{n,k}).
$$
\item A quantized Gaussian random variable $Z_n$ taking values in $\mathcal{L}_{n}$ with given by
$$
\mathbb P(Z_{n}=x_{n,k})=\frac{q_{n,k}}{\sum_{j\in\mathbb Z}q_{n, j}}.
$$
\item The R\'enyi divergence $D_{\alpha}(S_n\|Z_n)$:
$$
D_\alpha(S_n\|Z_n)=\frac{1}{\alpha-1}\log \sum_{k \in \mathbb{Z}} \frac{\mathbb P(S_n=x_{n,k})^\alpha}{\mathbb P(Z_n=x_{n,k})^{\alpha-1}}.
$$
\item The infinite-order R\'enyi divergence $D_{\infty}(S_n\|Z_n)$:
$$
D_{\infty}(S_n\|Z_n)=\log \left[\sup_{k\in \mathbb{Z}}\frac{\mathbb P(S_n=x_{n,k})}{\mathbb P(Z_n=x_{n,k})}\right].
$$
\item $\alpha$-average lattice-Gaussian ratio:
$$
A_{n,\alpha}=\sum_{k \in \mathbb{Z}}\frac{p_{n, k}^\alpha}{q_{n, k}^{\alpha-1}}.
$$
\item The supremum of lattice-Gaussian ratio:
$$
\Gamma_{n}=\sup_{k \in \mathbb{Z}}\frac{p_{n,k}}{q_{n,k}}.
$$
\item Constants appearing in $O(\cdot)$ and $o(\cdot)$ are independent of $n$, but may depend on other given parameters.
\end{itemize}




\section{Preliminaries} \label{sec:Preliminaries}
This section establishes the preliminary tools that will be employed throughout the rest of the paper.

\subsection{R\'enyi divergence} \label{Renyi divergence}

\begin{defn} 
Let $P$ and $Q$ be two probability distributions on a countable set $\mathcal{X}$, with probability mass functions $\{p(x)\}_{x\in\mathcal{X}}$ and $\{q(x)\}_{x\in\mathcal{X}}$, respectively. For $\alpha>0$ and $\alpha\neq 1$, the R\'enyi divergence $D_{\alpha}(P\|Q)$ of order $\alpha$ of $P$ from $Q$ is defined by
$$
D_{\alpha}(P\|Q)=\frac{1}{\alpha-1}\log \sum_{x\in\mathcal{X}} \frac{p(x)^{\alpha}}{q(x)^{\alpha-1}}.
$$
By taking a limit, one can define the R\'enyi divergence for special orders $\alpha=0, 1, \infty$ by
\begin{align*}
D_{0}(P\|Q) &=-\log Q(\{x\in\mathcal{X}: p(x)>0\}),\\
D_1(P\|Q)&=\sum_{x\in\mathcal{X}} p(x)\log\frac{p(x)}{q(x)},\\
D_\infty(P\|Q) &= \log\sup_{x\in\mathcal{X}}\frac{p(x)}{q(x)},
\end{align*}
where $D_1(P\|Q)$ is the classical Kullback--Leibler divergence (i.e., relative entropy) and the subindex 1 is usually omitted. 
\end{defn}

We briefly review some general properties of R\'enyi divergence; further details can be found in \cite{EH14}. Jensen's inequality implies both the positivity $D_{\alpha}(P\|Q)\geq0$ and the monotonicity of $D_{\alpha}(P\|Q)$ in $\alpha$. Consequently, the definition of R\'enyi divergence extends to limiting orders by taking appropriate limits. For $\alpha>0$, the equality $D_{\alpha}(P\|Q)=0$ holds if and only if $P=Q$; for $\alpha=0$, the equality holds if and only if $Q\ll P$.  The R\'enyi divergence of order $0<\alpha<1$ exhibits several distinctive properties. In particular, all such divergences in this range are equivalent in the sense that for any $0<\alpha<\beta<1$, 
$$
\frac{\alpha}{1-\alpha}\frac{1-\beta}{1-\alpha}D_\beta(P\|Q)\le D_\alpha(P\|Q)\le D_\beta(P\|Q).
$$
R\'enyi divergence is not symmetric with respect to its arguments $P$ and $Q$, and hence does not define a metric on the space of probability distributions. Nevertheless, it provides a useful family of information-theoretic measures for quantifying the discrepancy between distributions. Moreover, it is closely related to several important statistical distances, including the total variation distance, the Hellinger distance, and the $\chi^2$ distance; for a comprehensive treatment, see \cite{BCG19}.

Furthermore, the R\'enyi divergence $D_{\alpha}(P\|Q)$ is intimately connected to the relative Tsallis entropy $T_\alpha(P\|Q)$, which is defined by
$$
T_\alpha(P\|Q)=\frac{1}{\alpha-1}\left[\sum_{x\in\mathcal{X}} \frac{p(x)^\alpha}{q(x)^{\alpha-1}}-1\right].
$$
By taking a limit, we define the relative Tsallis entropy of orders $\alpha=0, 1$ by 
\begin{align*}
T_{0}(P\|Q) &=1-Q(\{x\in\mathcal{X}: p(x)>0\}),\\
T_1(P\|Q)&=\sum_{x\in\mathcal{X}} p(x)\log\frac{p(x)}{q(x)}.
\end{align*}
By convention, the relative Tsallis entropy of order $\alpha=\infty$ is defined by
$$
T_\infty(P\|Q)=\operatorname*{ess\,sup}_{x\in \mathcal{X}}\frac{p(x)-\varphi(x)}{\varphi(x)}.
$$
A direct calculation yields
$$
T_\alpha(P\|Q)=\frac{1}{\alpha-1}\left[e^{(\alpha-1)D_{\alpha}(P\|Q)}-1\right].
$$
Therefore, we have $D_{\alpha}(P\|Q)\le T_{\alpha}(P\|Q)$, and moreover, the two quantities are of the same order when they are small.

\subsection{Fourier transform}

In this subsection, we discuss the Fourier transform on the lattice $a+h\mathbb Z$, together with the associated Plancherel theorem and Hausdorff--Young inequality. We begin by recalling the Fourier transform on $h\mathbb Z=\{hk: k \in \mathbb{Z}\}$, whose dual group is isomorphic to the fundamental interval $[-\pi/h, \pi/h)$.

\begin{defn} [Fourier transform on $h\mathbb{Z}$]
The  Fourier transform of  $f\in \ell^1(h\mathbb Z)$ is defined by
\begin{equation*}
\widehat f(t)=\sum_{k\in\mathbb Z} f(kh)e^{itkh},\quad t\in \left[-\frac{\pi}{h}, \frac{\pi}{h}\right). 
\end{equation*}
Here, $\ell^p(h\mathbb{Z})
=\left\{
f:h\mathbb{Z}\to\mathbb{C}:\sum_{k\in\mathbb{Z}} |f(kh)|^p<\infty\right\}$ for $1\leq p <\infty$.
The Fourier inversion formula is given by
\begin{equation*}
f(kh)=\frac{h}{2\pi}	\int_{-\pi/h}^{\pi/h}\widehat f(t)e^{-itkh} \,dt.
\end{equation*}
\end{defn}

The following Plancherel Theorem and Hausdorff-Young inequality are special cases of Theorems 4.26 and 4.28, respectively, in \cite{Fol16}.

\begin{thm}[The Plancherel Theorem on $h\mathbb{Z}$]\label{thm:Plancherel-a}
For $f \in \ell^{1}(h\mathbb Z)$, we have
\begin{equation*}
\frac{h}{2\pi}\int_{-\pi/h}^{\pi/h}|\widehat f(t)|^2\,dt=\sum_{k\in\mathbb Z}|f(kh)|^2.
\end{equation*}
\end{thm}

\begin{thm}[The Hausdorff--Young Inequality on $h\mathbb{Z}$]\label{thm:HYI-a}
Suppose $1\le p\le 2$ and $1/p+1/q=1$. If $f\in \ell^p(h\mathbb Z)$, then we have
\begin{equation*}
\widehat f\in L^q\left(\left[-\frac{\pi}{h},\frac{\pi}{h}\right),\frac{h}{2\pi}\,dt\right),
\end{equation*}
and moreover
\begin{equation*}
\left(\frac{h}{2\pi}\int_{-\pi/h}^{\pi/h}|\widehat f(t)|^q\,dt\right)^{1/q}\le\left(
\sum_{k\in\mathbb Z}|f(kh)|^p\right)^{1/p}.
\end{equation*}
\end{thm}

Now we transfer the above definitions and results to the lattice $a+h\mathbb Z$. For $1\leq p<\infty$, we define
\begin{equation*}
\ell^p(a+h\mathbb{Z})=
\left\{f:a+h\mathbb{Z}\to\mathbb{C}:\sum_{k\in\mathbb{Z}} |f(a+kh)|^p<\infty\right\}.
\end{equation*}

\begin{defn} [Fourier transform on $a+h\mathbb{Z}$]
For $f\in \ell^1(a+h\mathbb{Z})$, its Fourier transform is defined by
\begin{equation*} 
\widehat f(t)=\sum_{k\in\mathbb{Z}} f(a+kh)e^{it(a+kh)},\quad t\in\left[-\frac{\pi}{h},\frac{\pi}{h}\right).
\end{equation*}
The Fourier inversion formula is given by
\begin{equation} \label{eq:lattice inverse Fourier transform}
f(a+kh)=
\frac{h}{2\pi}\int_{-\pi/h}^{\pi/h}\widehat f(t)e^{-it(a+kh)}\,dt.
\end{equation}
\end{defn}

\begin{thm}[The Plancherel Theorem on $a+h\mathbb{Z}$]\label{thm:Plancherel-b}
If $f\in \ell^1(a+h\mathbb{Z})$, then
\begin{equation*}
\frac{h}{2\pi}\int_{-\pi/h}^{\pi/h}|\widehat f(t)|^2\,dt=
\sum_{k\in\mathbb{Z}} |f(a+kh)|^2.
\end{equation*}
\end{thm}

\begin{thm}[The Hausdorff--Young Inequality on $a+h\mathbb{Z}$]\label{thm:HYI-b}
Suppose $1\le p\le 2$ and $1/p+1/q=1.$ If $f\in \ell^p(a+h\mathbb{Z})$, then 
$$
\widehat f\in L^q\left(\left[-\frac{\pi}{h},\frac{\pi}{h}\right),\frac{h}{2\pi}\,dt\right),
$$ 
and moreover
\begin{equation*}
\left(\frac{h}{2\pi}\int_{-\pi/h}^{\pi/h}|\widehat f(t)|^q\,dt\right)^{1/q}\leq\left(\sum_{k\in\mathbb{Z}} |f(a+kh)|^p\right)^{1/p}.
\end{equation*}
\end{thm}

\begin{proof}[Proof of Theorem \ref{thm:Plancherel-b} and \ref{thm:HYI-b}]
We define $g:h\mathbb{Z}\to\mathbb{C}$ by $g(kh)=f(a+kh)$. Then $g\in \ell^{p}(h\mathbb{Z})$ and note $\widehat f(t)=e^{ita}\widehat g(t)$. Then we can obtain Theorems \ref{thm:Plancherel-b} and \ref{thm:HYI-b} by applying Theorems \ref{thm:Plancherel-a} and \ref{thm:HYI-a} to $g$.
\end{proof}


\subsection{Stability of H\"older's inequality}

Let $(\Omega,\mu)$ be a measure space. For $1\le p<\infty$, we denote by $L^{p}(\Omega,\mu)$ the space of measurable functions $f: \Omega\to \mathbb{R}$ such that
$$
\|f\|_{p}:=\left(\int_{\Omega}|f(x)|^{p}\,\mu(dx)\right)^{1/p}<\infty.
$$
Let $q=p/(p-1)$ be the H\"older conjugate of $p$. The classical H\"older inequality states that for any $f\in L^p(\Omega,\mu)$ and $g\in L^q(\Omega,\mu)$, it holds that
\begin{equation*} 
\|fg\|_1\leq \|f\|_p\|g\|_q .
\end{equation*}
Moreover, if $\|f\|_{p}\|g\|_{q} >0$, equality holds if and only if 
\begin{equation*}
\frac{|f|^p}{\|f\|_p^p}=\frac{|g|^q}{\|g\|_q^q} \quad \text{a.e.}
\end{equation*}

A quantitative stability result of H\"older's inequality was established in \cite{Ald08}. 

\begin{thm}[\cite{Ald08}, Theorem 2.2] \label{thm:holder-stabi}
Let $1<p<\infty$. For $f\in L^p(\Omega,\mu), ~g\in L^q(\Omega,\mu)$ such that $\|f\|_p\|g\|_q>0$, it holds that
$$
\min\left\{\frac{1}{p},\frac{1}{q}\right\}\left\|\frac{|f|^{p/2}}{\|f\|_p^{p/2}}-\frac{|g|^{q/2}}{\|g\|_q^{q/2}}\right\|_{2}^{2}\leq 1-\frac{\|fg\|_1}{\|f\|_p\|g\|_q}\leq \max\left\{\frac{1}{p},\frac{1}{q}\right\}\left\|\frac{|f|^{p/2}}{\|f\|_p^{p/2}}-\frac{|g|^{q/2}}{\|g\|_q^{q/2}}\right\|_{2}^{2}.
$$
\end{thm}

As a consequence, we obtain the following sequential stability of H\"older's inequality, which will play a crucial role in our proof of the necessity direction of the main result. 

\begin{coro} \label{coro:holder-stabi-seq}
Let $1<p<\infty$. Let $\{f_{n}\}_{n=1}^\infty\subset L^{p}(\Omega,\mu)$ and $\{g_{n}\}_{n=1}^\infty\subset L^{q}(\Omega,\mu)$ such that $\|f_{n}\|_p\|g_{n}\|_q>0$. Suppose that 
$$
\lim_{n\to\infty}\frac{\|f_{n}g_{n}\|_{1}}{\|f_{n}\|_{p}\|g_{n}\|_{q}}=1,
$$
then 
$$
\lim_{n\to\infty}\left\|\frac{|f_{n}|^{p}}{\|f_{n}\|_p^{p}}-\frac{|g_{n}|^{q}}{\|g_{n}\|_q^{q}}\right\|_1=0.
$$
\end{coro}

\begin{proof} 
By Theorem \ref{thm:holder-stabi}, we have 
$$
\left\|\frac{|f_{n}|^{p/2}}{\|f_{n}\|_p^{p/2}}-\frac{|g_{n}|^{q/2}}{\|g_{n}\|_q^{q/2}}\right\|_2\to 0 \quad \text{as}~n\to \infty.
$$
By Cauchy-Schwarz inequality, 
\begin{align*}
\left\|\frac{|f_{n}|^{p}}{\|f_{n}\|_p^{p}}-\frac{|g_{n}|^{q}}{\|g_{n}\|_q^{q}}\right\|_1
&=\left\|\left(\frac{|f_{n}|^{p/2}}{\|f_{n}\|_p^{p/2}}-\frac{|g_{n}|^{q/2}}{\|g_{n}\|_q^{q/2}}\right)\left(\frac{|f_{n}|^{p/2}}{\|f_{n}\|_p^{p/2}}+\frac{|g_{n}|^{q/2}}{\|g_{n}\|_q^{q/2}}\right)\right\|_{1}\\
&\leq \left\|\frac{|f_{n}|^{p/2}}{\|f_{n}\|_p^{p/2}}-\frac{|g_{n}|^{q/2}}{\|g_{n}\|_q^{q/2}}\right\|_{2} \cdot\left\|\frac{|f_{n}|^{p/2}}{\|f_{n}\|_p^{p/2}}+\frac{|g_{n}|^{q/2}}{\|g_{n}\|_q^{q/2}}\right\|_{2}\\
&\leq 2\left\|\frac{|f_{n}|^{p/2}}{\|f_{n}\|_p^{p/2}}-\frac{|g_{n}|^{q/2}}{\|g_{n}\|_q^{q/2}}\right\|_{2}\to 0 \quad \text{as}~n\to \infty.
\end{align*}
The last inequality follows from the fact that 
$$
\left\|\frac{|f_{n}|^{p/2}}{\|f_{n}\|_p^{p/2}}\right\|_2=\left\|\frac{|g_{n}|^{q/2}}{\|g_{n}\|_q^{q/2}}\right\|_2=1.
$$
\end{proof}


\section{Quantized Gaussian moments}\label{sec:quantized gaussian moments}

In this section, we use the Poisson summation formula to derive estimates for quantized Gaussian moments, relating these sums to the corresponding Gaussian integrals with exponentially small errors. These results will facilitate estimates in the subsequent sections. 

A function $f:\mathbb{R}\to \mathbb{C}$ is called a \textit{Schwartz function} if $f$ and  all its derivatives are rapidly decreasing, in the sense that 
$$
\sup_{x\in \mathbb{R}}|x|^{t}|f^{(k)}(x)|<\infty, \quad \forall ~t\geq 0,~k\in\mathbb{Z}_{\ge 0}.
$$
Theorem 3.1 of Chapter 5 in \cite{SS03} provides a Poisson summation formula on $\{a+\mathbb{Z}\}$. A simple change of variable yields the following extension. Therefore we omit the proof.

\begin{prop}[Poisson summation formula] 
Let $f:\mathbb{R}\to \mathbb{C}$ be a Schwartz function with the Fourier transform $\widehat f: \mathbb{R}\to \mathbb{C}$ defined by
$$
\widehat f(t)=\int_{\mathbb{R}} f(x)e^{itx}\,dx .
$$
For $a\in\mathbb{R}$ and $h>0$, we have the following Poisson summation formula
\begin{equation}\label{eq:poisson sum}
\sum_{k\in\mathbb{Z}} f(a+kh)=\frac{1}{h}\sum_{k\in\mathbb{Z}}\widehat f\left(\frac{2\pi k}{h}\right)e^{-\frac{2\pi ika}{h}}.
\end{equation}
\end{prop}

Recall $x_{n, k}=(na+kh)/\sqrt{n}$ and $q_{n,k}=\delta_{n}\varphi(x_{n,k})$, where $\delta_n=h/\sqrt n$ and $\varphi(x)=\frac{1}{\sqrt{2\pi}}e^{-x^2/2}$. Now we establish estimates of quantized Gaussian moments.
 
\begin{lem} 
For all $\ell \geq 1$, we have
\begin{equation} \label{eq:gaussian-moment} 
\delta_{n}\sum_{k\in\mathbb{Z}}\varphi^{\ell}(x_{n,k})
=\int_{\mathbb{R}} \varphi^{\ell}(x)\,dx+O\big(e^{-2\pi^2 n/\ell h^2}\big).
\end{equation}	
In particular, we obtain
\begin{equation} \label{eq:normalizing-const}
\sum_{k\in\mathbb{Z}} q_{n,k}=1+O\big(e^{-2\pi^{2}n/h^{2}}\big).
\end{equation}
\end{lem}

\begin{proof}
It is clear that $\varphi^{\ell}(x)$ is a Schwartz function and its Fourier transform is 
$$
\widehat{\varphi^{\ell}}(t)=\int_{\mathbb R} \varphi^{\ell}(x)e^{itx}\,dx= \frac{\ell^{-1/2}}{(\sqrt{2\pi})^{\ell-1}}e^{-\frac{t^2}{2\ell}}.
$$
Apply identity \eqref{eq:poisson sum} to $\varphi^{\ell}(x)$, with $a$ and $h$ replaced by $\sqrt{n}a$ and $\delta_{n}$, respectively, to obtain
\begin{align*}
\delta_{n}\sum_{k\in\mathbb{Z}}\varphi^{\ell}
(x_{n,k})
&=\sum_{k\in \mathbb{Z}}\widehat{\varphi^\ell}\left(\frac{2\pi k}{\delta_{n}}\right)e^{-\frac{2\pi i k\sqrt{n}a}{\delta_{n}}}\\
&=\int_{\mathbb{R}} \varphi^{\ell}(x)\,dx+\frac{\ell^{-1/2}}{(\sqrt{2\pi})^{\ell-1}}\sum_{k\neq 0} e^{-\frac{2\pi^{2}k^{2}}{\ell\delta_{n}^{2}}}e^{-\frac{2\pi i k\sqrt{n}a}{\delta_{n}}}.
\end{align*}
Then
\begin{align*}
\left|\delta_{n}\sum_{k\in\mathbb{Z}}\varphi^{\ell}(x_{n,k})-\int_{\mathbb{R}} \varphi^{\ell}(x)\,dx\right| 
&\leq \frac{2\ell^{-1/2}}{(\sqrt{2\pi})^{\ell-1}}\sum_{k=1}^{\infty} e^{-\frac{2\pi^{2}k^{2}}{\ell\delta_{n}^{2}}}\leq \frac{2\ell^{-1/2}}{(\sqrt{2\pi})^{\ell-1}}\sum_{k=1}^{\infty} e^{-\frac{2\pi^{2}k}{\ell\delta_{n}^{2}}}\\
&\le \frac{2\ell^{-1/2}}{(\sqrt{2\pi})^{\ell-1}}\cdot\frac{e^{-2\pi^{2}/\ell \delta_n^2}}{1-e^{-2\pi^2/\ell \delta_n^2}}\\
&= \frac{2\ell^{-1/2}}{(\sqrt{2\pi})^{\ell-1}}\cdot \frac{e^{-2\pi^{2}n/\ell h^{2}}}{1-e^{-2\pi^{2}n/\ell h^{2}}}.
\end{align*} 
Thus we get 
$$\delta_{n}\sum_{k\in \mathbb{Z}}\varphi^{\ell}(x_{n,k})=\int_{\mathbb{R}} \varphi^{\ell}(x)\,dx+O\big(e^{-2\pi^2 n/\ell h^2}\big).
$$
This concludes the proof.
\end{proof}

\begin{rmk} 
One can check that the same argument yields identity \eqref{eq:gaussian-moment} for translations of $\varphi(x)$, that is for any $b \in \mathbb{R}$, we have
\begin{equation}\label{eq:translation}
\delta_{n}\sum_{k\in\mathbb{Z}}\varphi^{\ell}(x_{n,k}-b)=\int_{\mathbb{R}} \varphi^{\ell}(x)\,dx+O\big(e^{-2\pi^2 n/\ell h^2}\big).
\end{equation}
\end{rmk}

\begin{lem}	\label{lem:uniform-gaussian-summation}	For $n\geq 1$, $b\in \mathbb{R}$ and $u \in \mathbb{R}$ satisfying $|u|\leq \pi/h$, we have 
\begin{equation}\label{eq:gaussian moment remainder}
\left|\delta_{n}\sum_{k \in \mathbb{Z}} \varphi(x_{n,k}-b)e^{i\sqrt{n}u(x_{n,k}-b)}-e^{-\frac{1}{2}nu^2}\right|\leq C_{h}e^{-n\pi^2/(2h^2)},
\end{equation}
where $C_{h}=2/(1-e^{-n\pi^2/h^2})$. 
\end{lem}

\begin{proof}
Denote $f(x)=\varphi(x-b)e^{i\sqrt{n}u(x-b)}, |u|\leq \pi/h$. Then $f$ is a Schwartz function. And its Fourier transform is 
$$
\widehat{f}(t)=\int_{\mathbb{R}}f(x)e^{itx}\,dx =\frac{e^{itb}}{\sqrt{2\pi}}\int_{\mathbb{R}}	e^{-\frac12 x^2+i(t+\sqrt{n}u)x}\,dx =e^{itb}e^{-\frac12(t+\sqrt{n}u)^2}.	
$$
Apply identity \eqref{eq:poisson sum} to $f$, with $a$ and $h$ replaced by $\sqrt{n}a$ and $\delta_{n}$, respectively, to obtain
\begin{align*}
\delta_n\sum_{k\in\mathbb{Z}}\varphi(x_{n,k}-b)e^{i\sqrt{n}u(x_{n,k}-b)}
&=\sum_{k\in\mathbb{Z}}\exp\left[-\frac{1}{2}\left(\frac{2\pi k}{\delta_n}+\sqrt{n}u\right)^2\right]e^{-\frac{2\pi ik (\sqrt{n}a-b)}{\delta_n}} \\
&=\sum_{k\in\mathbb{Z}}\exp\left[-\frac {n}{2}\left(\frac{2\pi k}{h}+u\right)^2\right]e^{-\frac{2\pi ik (\sqrt{n}a-b)}{\delta_n}} \\
&=e^{-\frac12nu^2}+\sum_{k\neq0}\exp\left[-\frac {n}{2}\left(\frac{2\pi k}{h}+u\right)^2\right]e^{-\frac{2\pi ik (\sqrt{n}a-b)}{\delta_n}}.
\end{align*}
Since $|u|\leq \pi/h$, for every $k\neq0$,
$$
\left|\frac{2\pi k}{h}+u\right|\geq\frac{2\pi|k|}{h}-|u| \geq\frac{(2|k|-1)\pi}{h}.
$$
Then we can get
\begin{align*}
\left|\delta_n\sum_{k\in\mathbb{Z}}\varphi(x_{n,k}-b)e^{i\sqrt{n}u (x_{n,k}-b)}-e^{-\frac12nu^2}\right|
&\leq \sum_{k\neq0}\exp\left[-\frac{n\pi^2}{2h^2}(2|k|-1)^2\right] \\
&=2\sum_{k=1}^{\infty}\exp\left[-\frac{n\pi^2}{2h^2}(2k-1)^2\right] \\
&\leq2\sum_{k=1}^{\infty}\exp\left[-\frac{n\pi^2}{2h^2}(2k-1)\right] \\
&=\frac{2e^{-n\pi^2/(2h^2)}}{1-e^{-n\pi^2/h^2}}.
\end{align*}
This completes the proof.
\end{proof}

\begin{lem} \label{lem:gaussian-moment-a} 
For any polynomial $f(x)$ of degree at most $d\geq 0$, we have
\begin{equation*} 
\sum_{k\in\mathbb{Z}}f(x_{n,k})q_{n,k}=\int_{\mathbb{R}}f(x) \varphi(x)~dx+O(e^{-\pi^2n/h}).
\end{equation*}
\end{lem}

\begin{proof}
Denote $f(x)=\sum_{r=0}^{d} c_{r}x^{r}$ and $g(x)=f(x)\varphi(x)$, then $g(x)$ is still a Schwartz function. Apply identity \eqref{eq:poisson sum} to $g(x)$, with $a$ and $h$ replaced by $\sqrt{n}a$ and $\delta_{n}$, respectively, and obtain 
$$
\delta_{n}\sum_{k\in\mathbb{Z}} g(x_{n,k})
=\sum_{k\in\mathbb{Z}}
\widehat g\left(\frac{2\pi k}{\delta_{n}}\right)e^{\frac{-2\pi i k \sqrt{n}a}{\delta_{n}}}.
$$
Hence
\begin{align}\label{eq:PI-1}
\sum_{k\in\mathbb{Z}}f(x_{n,k})q_{n,k}
&=\delta_{n} \sum_{k\in \mathbb{Z}} f(x_{n,k})\varphi(x_{n,k})
= \sum_{k\in \mathbb{Z}}\widehat g\left(\frac{2\pi k}{\delta_{n}}\right)e^{\frac{-2\pi i k \sqrt{n}a}{\delta_{n}}} \notag\\
&=\int_{\mathbb{R}} f(x)\varphi(x)\,dx+\sum_{k\neq 0}
\widehat g\left(\frac{2\pi k}{\delta_n}\right)e^{\frac{-2\pi i k \sqrt{n}a}{\delta_{n}}}.
\end{align}
Since $\widehat{\varphi}^{(r)}(t)=(-1)^{r}H_{r}(t)e^{-t^{2}/2}$ where $H_{r}(t)$ is the Chebyshev-Hermit polynomial of order $r$. Then 
$$
\widehat{x^{r}\varphi}(t)=\int_{\mathbb{R}} x^{r}\varphi(x)e^{itx}\,dx=(-i)^{r}\frac{d^{r}}{dt^{r}}\int_{\mathbb{R}} \varphi(x)e^{itx}\,dx=(-i)^{r}\widehat{\varphi}^{(r)}(t)=i^{r}H_{r}(t)e^{-t^{2}/2}.
$$
We can obtain 
\begin{equation}\label{eq:g-fourier}
\widehat g(t)=\widehat{f\varphi}(t)=\sum_{r=0}^{d}c_{r}\widehat{x^{r}\varphi}(t)=\sum_{r=0}^{d}i^{r}c_{r}H_{r}(t)e^{-t^{2}/2}.
\end{equation}
Denote $H(t)=\sum_{r=0}^{d}i^{r}c_{r}H_{r}(t)$ whose degree is at most $d$. Therefore there exist constant $C >0$ such that $|H(t)| \leq C(1+|t|^{d})$ for all $t\in\mathbb R$. Combine this with \eqref{eq:PI-1} and \eqref{eq:g-fourier} to obtain
\begin{align*}
\left|\sum_{k\in \mathbb{Z}} f(x_{n,k})q_{n,k}-\int_{\mathbb{R}} f(x)\varphi(x)\,dx\right|
&\leq \sum_{k\neq 0}\left|\widehat g\left(\frac{2\pi k}{\delta_n}\right)\right|\\
&\leq \sum_{k\neq 0}  C\left(1+\left|2\pi k/\delta_n\right|^{d}\right)e^{-\frac{2\pi^{2}k^{2}}{\delta_{n}^{2}}}\\
&=\sum_{k\neq 0}  C\left(1+\left|2\pi k/\delta_n\right|^{d}\right)e^{-\frac{\pi^{2}k^{2}}{\delta_{n}^{2}}}e^{-\frac{\pi^{2}k^{2}}{\delta_{n}^{2}}}\\
&\leq \widetilde C \sum_{k=1}^{\infty} e^{-\frac{\pi^{2}k}{\delta_{n}^{2}}}=\frac{\widetilde{C}e^{-\pi^2/\delta_n^2}}{1-e^{-\pi^2/\delta_n^2}}\\
&=\frac{\widetilde Ce^{-\pi^{2}n/h^{2}}}{1-e^{-\pi^{2}n/h}},
\end{align*}
where $\widetilde C =2C\sup_{x}(1+|2\pi x|^{d})e^{-\pi^{2}x^{2}} $. We get 
$$
\sum_{k\in\mathbb{Z}}f(x_{n,k})q_{n,k}=\int_{\mathbb{R}}f(x) \varphi(x)~dx+O\big(e^{-\pi^{2}n/h^{2}}\big).
$$
\end{proof}

\begin{lem} \label{lem:gaussian-moment-tail} 
For any integer $\ell\geq 0$, and large enough $M_{n}$, it holds that
$$
\left| \sum_{|x_{n,k}|\geq M_{n}}|x_{n,k}|^{\ell}q_{n,k}-\int_{|x|\geq M_{n}} |x|^{\ell}\varphi(x)\,dx\right|\leq 2\delta_{n}M_{n}^{\ell}\varphi(M_{n}).
$$
\end{lem}

\begin{proof}
We first consider the case $x_{n,k}\geq M_{n}$ for sufficiently large $M_n$ and choose $k_0$ such that
$$
x_{n,k_0-1}<M_n\leq x_{n,k_0}.
$$
Denote $g(x)=|x|^\ell \varphi(x).$ For $x\geq M_n $, $g(x)=x^\ell \varphi(x),$ since $M_n$ is large enough, we have
$$
g'(x)=\ell x^{\ell-1}\varphi(x)-x^{\ell+1}\varphi(x) =x^{\ell-1}\varphi(x)(\ell-x^2)<0.
$$
Thus $g(x)$ is decreasing on $[M_n,\infty)$. We have
$$
\int_{M_n}^{\infty} g(x)\,dx
=\int_{M_n}^{x_{n,k_0}} g(x)\,dx
+\sum_{k=k_0}^{\infty}\int_{x_{n,k}}^{x_{n,k+1}} g(x)\,dx.
$$
 For any $k\geq k_{0}$, by the the monotonicity of $g$, we obtain 
$$
g(x_{n,k+1})\delta_n\leq \int_{x_{n,k}}^{x_{n,k+1}} g(x)\,dx\leq g(x_{n,k})\delta_n.
$$
Thus 
$$
\int_{M_n}^{\infty} g(x)\,dx\leq g(M_{n})\delta_{n}+\sum_{k=k_0}^{\infty}g(x_{n,k})\delta_n,
$$
and
\begin{align*}
\int_{M_n}^{\infty} g(x)\,dx &\geq \int_{M_n}^{x_{n,k_0}} g(x)\,dx +\sum_{k=k_0}^{\infty}g(x_{n,k})\delta_n-g(x_{n,k_{0}})\delta_{n}\\
&\geq -g(M_{n})\delta_n+\sum_{k=k_0}^{\infty}g(x_{n,k})\delta_n. 
\end{align*}
Hence
$$
\left|\sum_{x_{n,k}\geq M_{n}} g(x_{n,k})\delta_n-\int_{M_n}^{\infty} g(x)\,dx\right|\leq g(M_n)\delta_n=\delta_{n}M_{n}^{\ell}\varphi(M_{n}) . 
$$
Applying a similar argument for $x_{n,k} \leq -M_{n}$, we obtain 
\begin{align*}
\left|\sum_{x_{n,k}\leq -M_{n}} g(x_{n,k})\delta_n-\int_{-\infty}^{-M_n} g(x)\,dx\right|\leq\delta_{n}M_{n}^{\ell}\varphi(M_{n}).
\end{align*}
\end{proof}

\begingroup
\let\raggedright\centering
\part{CLT in R\'enyi divergence of order $1<\alpha<\infty$}
\endgroup


\section{sub-Gaussianity} \label{sec:sub-Gaussianity}

Recall that $X_{1},\cdots, X_{n}$ are independent copies of a lattice random variable $X$ with zero mean, unit variance, and maximal span $h>0$, and $S_{n}=(X_{1}+\cdots+X_{n})/\sqrt{n}$ is their normalized sum. 
The objective  of this section is to show that finite  R\'enyi divergence $D_\alpha(S_n\|Z_n)$ implies that $X$ and $S_n$ are sub-Gaussian. Recall that
$$
D_\alpha(S_n\|Z_n)=\frac{1}{\alpha-1}\log \sum_{k \in \mathbb{Z}} \frac{p_{n, k}^\alpha}{q_{n, k}^{\alpha-1}}+\log \left(\sum_{k\in\mathbb Z}q_{n, k}\right).
$$
By virtue of the identity in  \eqref{eq:normalizing-const}, the condition $D_\alpha(S_n\|Z_n)<\infty$ is equivalent to the summability condition
$$
A_{n,\alpha}=\sum_{k \in \mathbb{Z}}\frac{p_{n, k}^\alpha}{q_{n, k}^{\alpha-1}}<\infty.
$$

\begin{prop} \label{prop:sub-Gaussian-a} 
Let $\alpha >1$ and $\beta =\alpha/(\alpha-1)$. Suppose $D_{\alpha}(S_{n}||Z_{n}) < \infty$ for some $n$. For all $c<1/(2\beta)$, we have
$$
\mathbb{E} e^{cX^{2}}< \infty \quad \text{and} \quad \mathbb{E} e^{cS_{n}^{2}}< \infty.
$$
\end{prop}

\begin{proof}
We denote by $\{p_k\}_{k\in\mathbb Z}$ the distribution of $X$, that is, $p_k=\mathbb P\{X=x_k\}$ where $x_k=a+kh$. Clearly, we have $p_k^n\le p_{n, nk}$ for all $k\in\mathbb Z$ and therefore
\begin{equation}\label{eq:1st factor}
\sum_{k\in\mathbb Z}p_k^{n\alpha}q_{n,nk}^{1-\alpha}\le \sum_{k\in\mathbb Z}p_{n, nk}^{\alpha}q_{n,nk}^{1-\alpha}\le A_{n, \alpha}<\infty.
\end{equation}
Apply H\"older's inequality with $\gamma=n\alpha$ and $\gamma'=\gamma/(\gamma-1)$ to obtain
\begin{align*}
\mathbb E e^{cX^2} &=\sum_{k\in\mathbb Z}p_k e^{cx_k^2}=\sum_{k\in\mathbb Z}p_k e^{x_k^2/2\beta}\cdot e^{(c-1/2\beta)x_k^2}\\
&\leq \left[\sum_{k\in\mathbb Z}p_k^\gamma e^{\gamma x_k^2/2\beta}\right]^{1/\gamma}\left[\sum_{k\in\mathbb Z}e^{-\gamma'(1/2\beta-c)x_k^2}\right]^{1/\gamma'}\\
&=\left[\sum_{k\in\mathbb Z}p_k^{n\alpha}e^{n(\alpha-1) (a+kh)^2/2}\right]^{1/n\alpha}\left[\sum_{k\in\mathbb Z}e^{-\gamma'(1/2\beta-c)(a+kh)^2}\right]^{1/\gamma'}\\
&=\left[C_n\sum_{k\in\mathbb Z}p_k^{n\alpha}q_{n, nk}^{1-\alpha}\right]^{1/n\alpha}\left[\sum_{k\in\mathbb Z}e^{-\gamma'(1/2\beta-c)(a+kh)^2}\right]^{1/\gamma'}<\infty,
\end{align*}
where $C_{n}=(\sqrt{2\pi n}/h)^{1-\alpha}$. The first factor is finite due to \eqref{eq:1st factor}. The second factor is also finite since $c<1/(2\beta)$. As a consequence, we have
\begin{align*}
\mathbb E e^{cS_n^2} &= \mathbb E e^{c(X_1+\cdots+X_n)^2/n}\le \mathbb E e^{c(X_1^2+\cdots+X_n^2)}=\big(\mathbb E e^{cX^2}\big)^n<\infty.
\end{align*}
\end{proof}

\begin{prop} \label{prop:sub-Gaussian-b}
Let $\alpha >1$ and $\beta =\alpha/(\alpha-1)$. Suppose $D_{\alpha}(S_{n}||Z_{n}) < \infty$ for some $n$. There exists a constant $C>0$ depending on $h$ such that for all $t\in \mathbb R$ we have
$$
\mathbb{E} e^{tS_{n}}\leq C A_{n,\alpha}^{1/\alpha}~e^{\beta t^{2}/2} \quad \text{and} \quad \mathbb{E} e^{tX}\leq (C A_{n,\alpha}^{1/\alpha})^{1/n}~e^{\beta t^{2}/2}.
$$
\end{prop}

\begin{proof}
By H\"older's inequality, we have 
\begin{align*}
\mathbb E e^{tS_n}&=\sum_{k\in \mathbb{Z}} p_{n,k} e^{t x_{n,k}} =\sum_{k\in \mathbb{Z}}\frac{p_{n,k}}{q_{n,k}^{1/\beta}}\cdot q_{n,k}^{1/\beta}e^{t x_{n,k}} \\
&\leq\left[\sum_{k\in \mathbb{Z}}\frac{p_{n,k}^{\alpha}}{q_{n,k}^{\alpha-1}}\right]^{1/\alpha}\left[\sum_{k\in \mathbb{Z}} q_{n,k} e^{t\beta x_{n,k}}\right]^{1/\beta}\\
&=A_{n,\alpha}^{1/\alpha}\left[\frac{\delta_{n}}{\sqrt{2\pi}}\sum_{k\in \mathbb{Z}}e^{-(x_{n,k}-t\beta)^2/2}\right]^{1/\beta}e^{\beta t^2/2}.
\end{align*}
By \eqref{eq:translation} with $\ell=1$ and $b=t\beta$, we have
$$
\frac{\delta_{n}}{\sqrt{2\pi}}\sum_{k\in \mathbb{Z}}e^{-(x_{n,k}-t\beta)^2/2}= 1+O\big(e^{-2\pi^2 n/h^{2}}\big).
$$
Put these together to obtain the first inequality. The second one follows from the identity
$$
\mathbb{E} e^{tX}=\big(\mathbb{E} e^{\sqrt n tS_n})^{1/n}.
$$
\end{proof}

Assuming $D_{\alpha}(S_{n}||Z_{n})<\infty$, we can actually sharpen Proposition \ref{prop:sub-Gaussian-b} for large values of $t$.

\begin{prop} \label{prop:sub-Gaussian-large t}
Let $\alpha >1$ and $\beta =\alpha/(\alpha-1)$. Suppose $D_{\alpha}(S_{n}||Z_{n}) < \infty$ for some $n$. Then we have 
$$
\lim_{|t|\to\infty } \mathbb E e^{tS_{n}}e^{-\beta t^{2}/2}=0 \quad \text{and} \quad \lim_{|t|\to\infty } \mathbb E e^{tX}e^{-\beta t^{2}/2}=0.
$$
\end{prop}

\begin{proof}
We first consider the case of $t>0$ and decompose
\begin{align}\label{eq:mgf-sn-large t}
\mathbb E e^{tS_n}=\sum_{k \in \mathbb Z} p_{n,k}e^{t x_{n,k}} &=\sum_{x_{n,k}\le 0} p_{n,k}e^{t x_{n,k}}+\sum_{0<x_{n,k}\le \beta t/2} p_{n,k}e^{t x_{n,k}}+\sum_{x_{n,k}>\beta t/2} p_{n,k}e^{t x_{n,k}} \notag\\
&:=I_1(t)+I_2(t)+I_3(t).
\end{align}
Since $t>0$, we have
\begin{equation}\label{eq:I_1-limit}
I_1(t)=\sum_{x_{n,k}\le 0} p_{n,k}e^{t x_{n,k}} \leq \sum_{x_{n,k}\le 0} p_{n,k}\leq 1.
\end{equation}	
For $I_2(t)$, by H\"older's inequality, we 
\begin{align}
I_2(t) 
&=\sum_{0<x_{n,k}\le \beta t/2}\frac{p_{n,k}}{q_{n,k}^{1/\beta}}\cdot q_{n,k}^{1/\beta} e^{t x_{n,k}} \notag\\
&\leq\left[\sum_{0<x_{n,k}\le \beta t/2}\frac{p_{n,k}^{\alpha}}{q_{n,k}^{\alpha-1}}\right]^{1/\alpha}\left[\sum_{0<x_{n,k}\le \beta t/2}q_{n,k} e^{\beta t x_{n,k}}\right]^{1/\beta} \notag\\
&\leq A_{n,\alpha}^{1/\alpha}\left[\frac{\delta_n}{\sqrt{2\pi}}\sum_{0<x_{n,k}\le \beta t/2}
e^{-x_{n,k}^2/2+\beta t x_{n,k}}\right]^{1/\beta} \notag\\
&\leq A_{n,\alpha}^{1/\alpha}\left[\frac{\delta_n}{\sqrt{2\pi}}\sum_{0<x_{n,k}\le \beta t/2}
e^{3\beta t^2/8}\right]^{1/\beta}\notag\\
&\le A_{n,\alpha}^{1/\alpha} \left[\frac{\delta_n}{\sqrt{2\pi}}\left(\frac{\beta t}{2\delta_{n}}+1\right)e^{3\beta^{2} t^2/8}\right]^{1/\beta}\notag\\
&\leq A_{n,\alpha}^{1/\alpha}\left(\beta t/2+\delta_{n}\right)^{1/\beta}e^{3\beta t^2/8}\label{eq:I_2-limit}. 
\end{align}
For $I_3(t)$, again by H\"older's inequality,
\begin{align*}
I_3(t)=\sum_{x_{n,k}>\beta t/2}\frac{p_{n,k}}{q_{n,k}^{1/\beta}}\cdot q_{n,k}^{1/\beta}e^{t x_{n,k}} \leq\left[\sum_{x_{n,k}>\beta t/2}\frac{p_{n,k}^{\alpha}}{q_{n,k}^{\alpha-1}}\right]^{1/\alpha}\left[\sum_{x_{n,k}>\beta t/2}q_{n,k}e^{\beta t x_{n,k}}\right]^{1/\beta}.
\end{align*}
Since $A_{n,\alpha}<\infty$, we have
$$
\eta^{\alpha}(t):=\sum_{x_{n,k}>\beta t/2}p_{n,k}^{\alpha}q_{n,k}^{1-\alpha}\to 0 \quad \text{as}~t\to\infty.
$$
By \eqref{eq:translation} with $\ell=1$ and $b=t\beta$, there exists a constant $C>0$ depending on $h$ such that  
\begin{align*}
\sum_{x_{n,k}>\beta t/2}q_{n,k}e^{\beta t x_{n,k}} \leq \sum_{k \in \mathbb{Z}} q_{n,k}e^{\beta t x_{n,k}}
=\frac{\delta_{n}}{\sqrt{2\pi}}\sum_{k\in \mathbb{Z}}e^{-(x_{n,k}-t\beta)^2/2}e^{\beta^2 t^2/2}
\le C^{\beta} e^{\beta^2 t^2/2}.  
\end{align*}
Put these together and obtain
\begin{equation}\label{eq:I_3-limit}
I_3(t)\le C \eta(t)e^{\beta t^2/2}.
\end{equation}
Combine \eqref{eq:mgf-sn-large t}, \eqref{eq:I_1-limit}, \eqref{eq:I_2-limit}, \eqref{eq:I_3-limit} to obtain
$$
\mathbb E e^{tS_n}e^{-\beta t^2/2}\leq e^{-\beta t^2/2}+A_{n,\alpha}^{1/\alpha}\left(\beta t/2+\delta_{n}\right)^{1/\beta}e^{-\beta t^2/8}+C\eta(t)\to 0 \quad \text{as}~t\to\infty.
$$
By the same argument, we can obtain
$$
\lim_{t\to-\infty}\mathbb E e^{tS_n}e^{-\beta t^2/2}=0.
$$
The other part of the statement follows from the identity
$$
\mathbb E e^{tX}e^{-\beta t^2/2}=\big[\mathbb E e^{\sqrt n tS_{n_0}}e^{-\beta nt^2/2}\big]^{1/n}.
$$
\end{proof}

\begin{rmk} \label{remark-Proposition4.1-4.3}
Let $Y$ be a lattice-valued random variable and let $Z$ be the corresponding quantized Gaussian random variable. Assume that $D_\alpha(Y\|Z)<\infty$ for some $\alpha>1$. Then the statements of Propositions \ref{prop:sub-Gaussian-a}, \ref{prop:sub-Gaussian-b} and \ref{prop:sub-Gaussian-large t} hold with $S_n$ and $Z_n$ replaced by $Y$ and $Z$, respectively. 
\end{rmk}


\section{Enhanced sub-Gaussianity}\label{sec:enhanced sub-Gaussianity}

Let $n_0$ be an integer. For a lattice random variable $Y \in \mathcal{L}_{n_{0}}$ satisfying $D_{\alpha}(Y||Z_{n_{0}})<\infty$, we have established that $\mathbb{E} e^{cY^{2}}<\infty$ for every  $c<1/(2\beta)$. In general, however, this conclusion fails at the critical threshold $c=1/(2\beta)$. The goal of this section is to demonstrate that the normalized sum of sufficiently many independent copies of  $Y$ does exhibit the desired sub-Gaussian behavior even in this critical case.

For $1\leq \alpha <\infty$, we denote by $\ell^{\alpha}(\mathcal{L}_{n})$ the family of functions $f:\mathcal{L}_{n}\to \mathbb{R}$ with finite $\alpha$-norm $\|f\|_{\alpha}=\left(\sum_{k\in \mathbb{Z}}|f(x_{n,k})|^{\alpha}\right)^{1/\alpha}$. In analogy with the heat semigroup, we define the discrete Gaussian smoothing operator $T_{n,t}$ on $ \ell^{\alpha}(\mathcal{L}_{n})$ as follows. For any $t>0$, $f\in \ell^{\alpha}(\mathcal{L}_{n})$ and $y\in \mathcal{L}_{n}$, we define 
$$
T_{n,t}f(y)=\frac{1}{C_{n,t}}\sum_{x\in \mathcal{L}_{n}}e^{-(y-x)^2/(2t)}f(x),
$$
where $C_{n,t}=\sum_{k \in \mathbb{Z}} e^{-k^{2}\delta_{n}^{2}/(2t)}$ is the normalizing constant.

\begin{prop}\label{prop:Gaussian smooth}
Let $\alpha>1$ and $\beta=\alpha/(\alpha-1)$. For any $t>0$ and $f\in \ell^\alpha(\mathcal{L}_n)$, we have
\begin{enumerate}
\item $\|T_{n,t}f\|_{\alpha} \leq  \|f\|_{\alpha}$;
\item $\|T_{n,t}f\|_{\infty} \leq \big(C_{n,t/\beta}^{1/\beta}/C_{n, t}\big) \|f\|_{\alpha}$;
\item For $\alpha\leq\gamma <\infty$, $\|T_{n,t}f\|_{\gamma} \leq \big(C_{n,t/\beta}^{1/\beta}/C_{n,t}\big)^{1-\alpha/\gamma}\|f \|_{\alpha}.$
\end{enumerate}   
\end{prop}

\begin{proof}
(1) 
For $\alpha >1$, we apply Jensen's inequality to obtain  
$$
|T_{n,t}f(y)|^\alpha=\left|\sum_{x\in \mathcal{L}_{n}}\frac{e^{-(y-x)^2/(2t)}}{C_{n,t}}f(x)\right|^\alpha\leq 
\sum_{x\in \mathcal{L}_{n}}\frac{e^{-(y-x)^2/(2t)}}{C_{n,t}}|f(x)|^\alpha.
$$
Then 
\begin{align*}
 \|T_{n,t}f\|_\alpha^\alpha=\sum_{y\in \mathcal{L}_{n}}|T_{n,t}f(y)|^\alpha &\leq\sum_{y\in\mathcal{L}_{n}}\sum_{x\in \mathcal{L}_{n}}\frac{e^{-(y-x)^2/(2t)}}{C_t}|f(x)|^\alpha\\
 &=\sum_{x\in\mathcal{L}_{n}}\sum_{y\in \mathcal{L}_{n}}\frac{e^{-(y-x)^2/(2t)}}{C_{n,t}}|f(x)|^\alpha=\|f\|_\alpha^\alpha.   
\end{align*}

(2)	For any $y\in \mathcal{L}_{n}$, by H\"older's inequality with $\alpha>1$ and $\beta=\alpha/(\alpha-1)$, we have 
\begin{align*}
|T_{n,t} f(y)| &=\left|\sum_{x\in \mathcal{L}_{n}}\frac{e^{-(y-x)^2/(2t)}}{C_{n,t}}f(x) \right|\\
&\leq\left[\sum_{x\in \mathcal{L}_{n}}\frac{e^{-\beta(y-x)^2/(2t)}}{C_{n,t}^{\beta}}\right]^{1/\beta}\left[\sum_{x\in\mathcal{L}_{n}}|f(x)|^{\alpha}
\right]^{1/\alpha}=\frac{C_{n,t/\beta}^{1/\beta}}{C_{n,t}}\cdot\|f\|_{\alpha}.
\end{align*}

(3) For $1<\alpha\leq \gamma<\infty$, apply properties (1) and (2) to obtain
\begin{align*}
\sum_{y\in \mathcal{L}_{n}} |T_{n,t}f(y)|^\gamma 
&=\sum_{y\in \mathcal{L}_{n}}
|T_{n,t}f(y)|^{\gamma-\alpha}|T_{n,t}f(y)|^\alpha\\
&\leq\|T_{n,t}f\|_\infty^{\gamma-\alpha}\cdot\|T_{n,t}f\|_\alpha^\alpha\\
&\le \left(\frac{C_{n,t/\beta}^{1/\beta}}{C_{n,t}}\right)^{\gamma-\alpha}\|f\|_\alpha^{\gamma-\alpha}\cdot \|f\|_\alpha^\alpha\\
&= \left(\frac{C_{n,t/\beta}^{1/\beta}}{C_{n,t}}\right)^{\gamma-\alpha}\|f\|_{\alpha}^{\gamma}.
\end{align*}
\end{proof}

\begin{rmk}
For any $1<\alpha<\gamma$, we have the embedding $\ell^\alpha(\mathcal{L}_n)\subset\ell^\gamma(\mathcal{L}_n)$. Combined with statement (1) of Proposition \ref{prop:Gaussian smooth}, this implies that the discrete Gaussian smoothing operator $T_{n, t}: \ell^\alpha(\mathcal{L}_n)\to\ell^\gamma(\mathcal{L}_n)$ is a contraction for any $t>0$. This is different from its continuous counterpart,  Weierstrass transform (see Section 5 of \cite{BCG19}).
\end{rmk}

\begin{prop}\label{prop:sub-Gaussian-cricical} 
Let $n_0$ be a positive integer. Let $Y$ be a lattice random variable in $\mathcal{L}_{n_0}$ such that $D_{\alpha}(Y||Z_{n_0})<\infty$. Let $Y_{1},\cdots,Y_{n}$ be independent copies of $Y$ and set $U_{n}=(Y_{1}+\cdots+Y_{n})/\sqrt{n}$. For $n \geq \alpha$, we have
$$
\mathbb{E}\exp\left(\frac{U_{n}^{2}}{2\beta}\right)\leq \left[\frac{h+\sqrt{2\beta}}{Q_{n_0}}\cdot e^{D_{\alpha}(Y||Z_{n_0})}\right]^{^{n/\beta}},
$$
where $Q_{n_0}=\sum_{k\in\mathbb Z}q_{n_0, k}$.
\end{prop}

\begin{proof}
Let $P$ be the probability mass function of $Y$. For any $y,y_{1},\cdots,y_{n} \in \mathcal{L}_{n_{0}}$, denote 
$$
\widehat y=\frac{y_1+\cdots+y_n}{\sqrt{n}}\quad and \quad 
f(y)=	\frac{P(y)}{q_{n_0}^{1/\beta}(y)},
$$
where $q_{n_{0}}(y)=\delta_{n_{0}}\varphi(y)$. One can check that $\|f\|_{\alpha}^{\alpha }=\big[e^{D_{\alpha}(Y||Z_{n_0})}/Q_{n_0}\big]^{(\alpha-1)}$. Since $D_{\alpha}(Y||Z_{n_{0}})<\infty$, we have $f \in \ell^{\alpha}(\mathcal{L}_{n_{0}})$. We may therefore apply Proposition \ref{prop:Gaussian smooth} to $T_{n_{0},t}f$ and use the resulting properties to derive an upper bound for $\mathbb E e^{U_n^2/2\beta}$.

\begin{align*}
\mathbb{E}\exp\left(\frac{U_{n}^{2}}{2\beta}\right)
&=\sum_{y_1\in \mathcal L_{n_0}}\cdots\sum_{y_n\in \mathcal L_{n_0}}e^{\widehat{y}^2/(2\beta)}P(y_1)\cdots P(y_n) \\
&=\left(\frac{\delta_{n_0}}{\sqrt{2\pi}}\right)^{n/\beta}\sum_{y_1\in \mathcal L_{n_0}}
\cdots\sum_{y_n\in \mathcal L_{n_0}}\exp\left[\frac{(\sum_{j=1}^ny_j)^2-n\sum_{j=1}^ny_j^2}{2n\beta}\right]\prod_{i=1}^nf(y_i)\\
&=\left(\frac{\delta_{n_0}}{\sqrt{2\pi}}\right)^{n/\beta}\sum_{y_1\in \mathcal L_{n_0}}
\cdots\sum_{y_n\in \mathcal L_{n_0}}\exp\left[-\frac{1}{4\beta n}\sum\limits_{m=1}^n\sum\limits_{j=1}^n (y_m-y_j)^2\right]\prod_{i=1}^nf(y_i)\\
&=\left(\frac{\delta_{n_0}}{\sqrt{2\pi}}\right)^{n/\beta}\sum_{y_1\in \mathcal L_{n_0}}
\cdots\sum_{y_n\in \mathcal L_{n_0}}\prod_{m=1}^n\left[\exp\left(-
\frac{1}{4\beta}\sum_{j=1}^n (y_m-y_j)^2\right)\prod_{i=1}^nf(y_i)\right]^{1/n},
\end{align*}
where the second last equality follows from $\sum_{m=1}^n\sum_{j=1}^n (y_m-y_j)^2=2n(y_1^2+\cdots+y_n^2)-2(y_1+\cdots+y_n)^2$. By H\"older inequality, we have 
\begin{align*}
\mathbb{E}\exp\left(\frac{U_{n}^{2}}{2\beta}\right)
&\le\left(\frac{\delta_{n_0}}{\sqrt{2\pi}}\right)^{n/\beta}\prod_{m=1}^n\left[\sum_{y_1\in \mathcal L_{n_0}}\cdots\sum_{y_n\in \mathcal L_{n_0}}\exp\left(-\frac{1}{4\beta}\sum_{j=1}^n (y_m-y_j)^2\right)\prod_{i=1}^nf(y_i)\right]^{1/n} \\
&=\left(\frac{\delta_{n_0}}{\sqrt{2\pi}}\right)^{n/\beta}\sum_{y_1\in \mathcal L_{n_0}}\cdots\sum_{y_n\in \mathcal L_{n_0}}\exp\left(-\frac{1}{4\beta}\sum_{j=1}^n (y_1-y_j)^2\right)\prod_{i=1}^nf(y_i) \\
&=\left(\frac{\delta_{n_0}}{\sqrt{2\pi}}\right)^{n/\beta}\sum_{y_1\in \mathcal L_{n_0}}\cdots\sum_{y_n\in \mathcal L_{n_0}}\prod_{j=1}^{n}\left[\exp\left(	-\frac{(y_1-y_j)^2}{4\beta}\right)f(y_j)\right] \\
&=\left(\frac{\delta_{n_0}}{\sqrt{2\pi}}\right)^{n/\beta}\sum_{y_1\in \mathcal L_{n_0}}
\left[\sum_{y\in \mathcal L_{n_0}}\exp\left(-\frac{(y_1-y)^2}{4\beta}\right)f(y)\right]^{n-1}f(y_1) \\
&=\left(\frac{\delta_{n_0}}{\sqrt{2\pi}}\right)^{n/\beta}C_{n_{0},2\beta}^{n-1}\sum_{y_1\in \mathcal L_{n_0}}
\left[T_{n_{0},2\beta}f(y_1)\right]^{n-1}f(y_1)\\
&\leq\left(\frac{\delta_{n_0}}{\sqrt{2\pi}}\right)^{n/\beta}C_{n_{0},2\beta}^{n-1}\|T_{n_{0},2\beta}f\|_{\beta(n-1)}^{n-1}\|f\|_{\alpha}.
\end{align*}
Now we derive an estimate for $C_{n_{0},t}$.
$$
C_{n_{0},t}=\sum_{k\in \mathbb Z}e^{- \delta_{n_{0}}^2 k^2/(2t)}
=1+2\sum_{k=1}^{\infty}e^{-\delta_{n_{0}}^2 k^2/(2t)} 
\le 1+2\int_0^\infty e^{- \delta_{n_{0}}^2 s^2/(2t)}\,ds 
=1+\sqrt{2\pi t}/\delta_{n_{0}},
$$
where the inequality follows from that $e^{-\beta \delta_{n_{0}}^2 s^2/(2t)}$ is decreasing on $(0,\infty)$. 
For $n\geq \alpha$, $\beta(n-1)=\alpha(n-1)/ (\alpha-1)\geq \alpha$. By (3) in Propsition \ref{prop:Gaussian smooth}, replacing $\gamma$ and $t$ by $\beta(n-1)$ and $2\beta$, respectively, we obtain 
\begin{align*}
C_{n_{0},2\beta} \|T_{n_{0},2\beta}f\|_{\beta(n-1)}&\leq C_{n_{0},2\beta}\big(C_{n_{0},2}^{1/\beta}/C_{n_{0},2\beta
}\big)^{1-\alpha/(\beta(n-1))}\|f\|_{\alpha}\\
&=C_{n_{0},2}^{1/\beta-\alpha/(\beta^{2}(n-1))}C_{n_{0},2\beta}^{\alpha/(\beta(n-1))}\|f\|_{\alpha}\\
&\leq \big(1+2\sqrt{\pi\beta}/\delta_{n_{0}}\big)^{n/(\beta(n-1))}\|f\|_{\alpha}.
\end{align*}
Furthermore,
\begin{align*}
\mathbb{E}\exp\left(\frac{U_{n}^{2}}{2\beta}\right)&\leq \left(\frac{\delta_{n_0}}{\sqrt{2\pi}}\right)^{n/\beta}\left(1+\frac{2\sqrt{\pi\beta}}{\delta_{n_{0}}}\right)^{n/\beta} \|f\|_{\alpha}^{n} \\
&\leq \big(h+\sqrt{2\beta}\big)^{n/\beta}\left(e^{D_{\alpha}(Y||Z_{n_0})}/Q_{n_0}\right)^{n/\beta}.
\end{align*} 
\end{proof}

\begin{coro}\label{coro:integrability}
Let $\alpha>1$ and $\beta=\alpha/(\alpha-1)$. Suppose $D_{\alpha}(S_{n_{0}}||Z_{n_0})<\infty$ for some $n_0$. Then the function $\psi(t) =\mathbb{E}e^{tX}e^{-\beta t^{2}/2}$ is integrable with power $\ell n_{0}$ for any integer $\ell\geq \alpha$, and moreover 
$$
\int_{\mathbb R}\psi(t)^{\ell n_{0}} ~\text{d}t \leq \sqrt{\frac{2\pi}{\beta \ell n_{0}}}\left[\frac{h+\sqrt{2\beta}}{Q_{n_0}}\cdot e^{D_\alpha(S_{n_0}\|Z_{n_0})}\right]^{\ell/\beta},
$$
where $Q_{n_0}=\sum_{k\in\mathbb Z}q_{n_0, k}$.
\end{coro}

\begin{proof} 
Set $n=\ell n_0$. It is clear that  
$$
\psi(t)^n=\mathbb E e^{t\sqrt n S_n}e^{-n\beta t^2/2}.
$$
By Fubini's theorem, we have
\begin{align*}
\int_{\mathbb R} \psi(t)^n\,dt
&=\mathbb E\left[\int_{\mathbb R}\exp\left(-\frac{n\beta t^2}{2}+t\sqrt n S_n\right)\,dt\right]
=\sqrt{\frac{2\pi}{\beta n}}\mathbb E
\exp\left(\frac{S_{n}^2}{2\beta}\right).
\end{align*} 
Let $Y_1,\ldots,Y_\ell$ be independent copies of $S_{n_0}$. Then we have $S_n\overset{d}{=}U_\ell=(Y_1+\cdots+Y_\ell)/\sqrt \ell$. 
By assumption, we have $D_{\alpha}(Y_i||Z_{n_{0}}) < \infty$. Therefore, we apply Proposition \ref{prop:sub-Gaussian-cricical} to obtain
$$
\int_{\mathbb R} \psi(t)^{\ell n_{0}}\,dt=\sqrt{\frac{2\pi}{\beta \ell n_{0}}}\mathbb E
\exp\left(\frac{U_\ell^2}{2\beta}\right)
\leq \sqrt{\frac{2\pi}{\beta \ell n_{0}}}\left[\frac{h+\sqrt{2\beta}}{Q_{n_0}}\cdot e^{D_\alpha(S_{n_{0}}\|Z_{n_0})}\right]^{\ell/\beta}.
$$
\end{proof}


\section{Tail estimates} \label{sec:tail estimates}

Given that $D_{\alpha}(S_{n_{0}}||Z_{n_{0}}) <\infty$ for some $n_{0}$, Proposition \ref{prop:sub-Gaussian-cricical} yields a Gaussian decay of $S_n$ of the form $p_{n, k}\le Ce^{-x_{n, k}^2/2\beta}$. We now establish refined pointwise bounds on $p_{n, k}$ that are effective in the moderate and large deviation regimes. This estimate complements the local limit theorem, which is typically restricted to moderately sized deviations. Let $\psi(t)=\mathbb{E}e^{tX}e^{-\beta t^{2}/2}$ and recall that $A_{n,\alpha}=\sum_{k \in \mathbb{Z}}p_{n,k}^{\alpha}/q_{n,k}^{\alpha-1}$.

\begin{prop}\label{prop:moderate-large deviation} 
Let $\alpha>1$ and $\beta=\alpha/(\alpha-1)$. Suppose $D_{\alpha}(S_{n_{0}}||Z_{n_{0}}) <\infty$ for some $n_{0}$. For $n\geq n_{\beta}:=\max\,\{2,\beta\}n_{0}$, we have
\begin{equation} 
p_{n,k}\leq C\exp\left(-\frac{x_{n,k}^{2}}{2\beta}\right)\psi\left(\frac{x_{n,k}}{\beta\sqrt{n}}\right)^{n-n_{\beta}},\quad k\in\mathbb{Z}.
\end{equation}
Here, $C>0$ is a constant depending on $n_{0},\alpha$ and $h$, which can be chosen explicitly as 
$$
\begin{aligned}
C&=
\begin{cases}
\delta_{n_{0}}A_{n_0,\alpha}^{1/(\alpha-1)}/\sqrt{2\pi},& ~1<\alpha \leq 2\\
\delta_{n_{0}}A_{n_0,\alpha}^{2/\alpha}\left(\sqrt{\frac{\beta(\alpha-2)}{2\alpha}}+ \frac{2\delta_{n_{0}}}{\sqrt{2\pi}}\right)^{\frac{\alpha-2}{\alpha}}\Big/\sqrt{2\pi},& ~\alpha > 2.
\end{cases}
\end{aligned}
$$
\end{prop}

\begin{lem}\label{lem:phi-phi_n-extension}
Let $\phi$ and $\phi_n$ be the characteristic functions of $X$ and $S_{n}$, respectively. That is,
\begin{equation*}
\phi(t)=\mathbb{E}e^{itX}\quad\text{and}\quad \phi_{n}(t)=\mathbb{E}e^{itS_{n}}, \quad t\in \mathbb{R}.
\end{equation*}
Suppose $D_{\alpha}(S_{n_0}||Z_{n_0}) <\infty$ for some $n_{0}$. Then $\phi$ and $\phi_n$ can be extended to the complex plane as an entire function. 
\end{lem}

\begin{proof}
Let $0<c<1/(2\beta)$. By Proposition \ref{prop:sub-Gaussian-b}, we have for any $z=u+iv\in \mathbb C$ that
\begin{equation*}
|\mathbb E e^{izX}|\le \mathbb E|e^{izX}|=\mathbb E e^{-vX}<\infty. 
\end{equation*}
Therefore, $\phi$ can be extended to the complex plane $\mathbb C$. Furthermore, for any given $z\in \mathbb C$, we have
$$
e^{izX}=\sum_{m=0}^\infty\frac{(izX)^m}{m!}\quad \text{and} \quad e^{|zX|}=\sum_{m=0}^\infty\frac{|zX|^m}{m!},
$$
and clearly
$$
\left|\sum_{m=0}^n\frac{(izX)^m}{m!}\right|\le \sum_{m=0}^n\frac{|zX|^m}{m!}\quad \text{for all}~n\in\mathbb{Z}_{\ge 0}.
$$
We apply Proposition \ref{prop:sub-Gaussian-a} to obtain
$$
\mathbb E e^{|zX|}\le \mathbb E e^{|z|^2/(4c)+cX^2}=e^{|z|^2/(4c)}\mathbb E e^{cX^2}<\infty.
$$
By Lebesgue's dominated convergence theorem, we have
$$
\mathbb{E}e^{izX}=\sum_{m=0}^{\infty}\frac{i^{m}\mathbb{E}X^{m}}{m!}z^{m}, \quad z \in \mathbb{C}.
$$
Thus $\phi$ admits an entire extension to the whole complex plane. Similarly, one can show that $\phi_{n}$ also admits such an entire extension. 
\end{proof}

One can check the following simple property, which will be repeatedly used in the following proofs. For any $z=u+iv\in \mathbb{C}$, we have
\begin{align}
\phi_{n}(z) &=\phi(z/\sqrt n)^{n}\label{eq:phi-phi_n},\\
|\phi(z)|&\le\mathbb Ee^{-vX}=\phi(iv)<\infty\label{eq:Fourier bound}.
\end{align}

\begin{lem}\label{lem:periodicity of F_n}
For each $n \in \mathbb{Z}_{\geq 1}$ and $k \in \mathbb{Z}$, the function $F_{n,k}(z)=e^{-izx_{n,k}}\phi_n(z), ~z\in \mathbb{C}$ is entire and $2\pi/\delta_n$-periodic, i.e.,
$$
F_{n,k}\left(z+\frac{2\pi}{\delta_{n}}\right)=F_{n,k}(z).
$$
\end{lem}
	
\begin{proof}
For each $n \in \mathbb{Z}_{\geq 1}$ and $k \in \mathbb{Z}$, both $e^{izx_{n,k}}$ and $\phi_{n}(z)$ are entire functions of $z$, and therefore so is their product $F_{n,k}(z)$. For $x_{n,k}=(na+kh)/\sqrt n$ and any $z\in \mathbb{C}$, we have 
\begin{align}\label{eq:e^itx}
\exp\left(-i\left(z+\frac{2\pi}{\delta_{n}}\right)x_{n,k}\right)
= e^{-izx_{n,k}}\exp\left(\frac{-2\pi ina}{h}\right). 
\end{align}
Since $X\in a+h\mathbb Z$ a.s., we may write $X=a+hY$, where $Y$ is a random variable in $\mathbb{Z}$. Therefore
$$
\exp\left(\frac{2\pi i X}{h}\right)=\exp\left(\frac{2\pi i (a+hY)}{h}\right)=\exp\left(\frac{2\pi ia}{h}\right).
$$
Note $\phi_{n}(z)=\phi(z/\sqrt{n})^{n}$. Then we get
\begin{align}
\phi_n\left(z+\frac{2\pi}{\delta_{n}}\right)
&=\left[\phi\left(\frac{z+2\pi/\delta_{n}}{\sqrt n}\right)\right]^n \notag\\
	&=\left[\mathbb E\exp\left(\frac{izX}{\sqrt n}\right)\exp\left(\frac{2\pi iX}{h}\right)
	\right]^n \notag\\
&=\left[\exp\left(\frac{2\pi ia}{h}\right)\mathbb E\exp\left(\frac{izX}{\sqrt n}\right)\right]^n \notag\\
&=\exp\left(\frac{2\pi ina}{h}\right)\phi_n(z). \label{eq:periodic-phi} 
\end{align}
Combine \eqref{eq:e^itx} and \eqref{eq:periodic-phi} to obtain
\begin{align*}
F_{n,k}\left(z+\frac{2\pi}{\delta_{n}}\right)&=\exp\left(-i\left(z+\frac{2\pi}{\delta_{n}}\right)x_{n,k}\right)\phi_n\left(z+\frac{2\pi}{\delta_{n}}\right) 
=e^{-izx_{n,k}}\phi_n(z) 
=F_{n,k}(z).
\end{align*}
This concludes the proof.	
\end{proof}

\begin{lem} \label{lem:+iy}
Let $n \in \mathbb{Z}_{\geq 1}$ and $k\in \mathbb{Z}$. For every fixed $y \in \mathbb{R}$, it holds that
\begin{equation}\label{eq:+iy}
\int_{-\pi/\delta_{n}}^{\pi/\delta_{n}} F_{n,k}(t)\,dt
=\int_{-\pi/\delta_{n}}^{\pi/\delta_{n}} F_{n,k}(t+iy)\,dt.
\end{equation}
\end{lem}

\begin{proof}
For $y=0$, it is obvious. For $y>0$, we choose the rectangle
$$
R_{y}=\{t+is:-\pi/\delta_{n}\leq t\leq \pi/\delta_{n}, ~0\leq s\leq y\}.
$$
We orient the boundary $\partial R_{y}$ counterclockwise and decompose it as $\partial R_{y}=\tau_{1}\cup\tau_{2}\cup\tau_{3}\cup\tau_{4}$, where 
\begin{align*}
\tau_{1} &=\{-\pi/\delta_{n}+2\pi r/\delta_{n}: 0\leq r \leq 1\}, \\
\tau_{2} &=\{\pi/\delta_{n}+iy r:  0\leq r \leq 1\}, \\
\tau_{3} &=\{\pi/\delta_{n}-2\pi r/\delta_{n}+iy: 0\leq r \leq 1\}, \\
\tau_{4} &=\{-\pi/\delta_{n}+iy(1-r): 0\leq r \leq 1\} . 
\end{align*}
Since $F_{n,k}$ is entire, Cauchy's integral theorem gives
\begin{align*}
0=\oint_{\partial R_{y}} F_{n,k}(z)\,dz
&~=\int_{-\pi/\delta_{n}}^{\pi/\delta_{n}} F_{n,k}(t)\,dt+\int_0^y F_{n,k}(\pi/\delta_{n}+is)i\,ds \\ &~~~+\int_{\pi/\delta_{n}}^{-\pi/\delta_{n}} F_{n,k}(t+iy)\,dt+\int_y^0 F_{n,k}(-\pi/\delta_{n}+is)i\,ds.
\end{align*}
By the periodicity of \(F_{n,k}\) stated in Lemma \ref{lem:periodicity of F_n},  we have
$$
F_{n,k}(\pi/\delta_{n}+is)=F_{n,k}(-\pi/\delta_{n}+is).
$$
Hence the second and fourth terms cancel each other. Then we can get 
$$
\int_{-\pi/\delta_{n}}^{\pi/\delta_{n}} F_{n,k}(t)\,dt=\int_{-\pi/\delta_{n}}^{\pi/\delta_{n}} F_{n,k}(t+iy)\,dt.
$$
The $y<0$ case of \eqref{eq:+iy} can be proved in a similar manner.
\end{proof}
    
\begin{proof}[Proof of Proposition \ref{prop:moderate-large deviation}]
Recall $p_{n, k}=\mathbb P(S_n=x_{n, k})$, where $x_{n,k}=(na+kh)/\sqrt{n}$. Using the Fourier inversion formula \eqref{eq:lattice inverse Fourier transform}, we have
\begin{align}
p_{n,k}&=\frac{\delta_n}{2\pi}\int_{-\pi/\delta_n}^{\pi/\delta_n}e^{-itx_{n,k}}\phi_n(t)\,dt\notag\\ 
&=\frac{\delta_n}{2\pi}\int_{-\pi/\delta_n}^{\pi/\delta_n}e^{-i(t+iy)x_{n,k}}\phi_n(t+iy)\,dt \notag\\
&=e^{yx_{n,k}}\cdot\frac{\delta_n}{2\pi}\int_{-\pi/\delta_n}^{\pi/\delta_n}e^{-itx_{n,k}} \phi\left((t+iy)/\sqrt n\right)^n dt\label{eq:inversion 1}.
\end{align}
The second identity follows from Lemma \ref{lem:+iy} and it holds for all $y\in\mathbb R$. The last identity follows from \eqref{eq:phi-phi_n}.

\textbf{Case 1: $\alpha>2, n\geq 2n_0$}. Combine \eqref{eq:inversion 1}, \eqref{eq:Fourier bound} and \eqref{eq:phi-phi_n} with $t:=s\sqrt{n/n_0}$ to obtain
\begin{align}
p_{n,k}&\leq e^{yx_{n,k}}
\phi\left(\frac{iy}{\sqrt n}\right)^{n-2n_0}\cdot \frac{\delta_{n}}{2\pi}
\int_{-\pi/\delta_n}^{\pi/\delta_n}\left|\phi\left(\frac{t}{\sqrt{n}}+\frac{iy}{\sqrt{n}}\right)\right|^{2n_0}\, dt \notag\\
&=e^{yx_{n,k}}
\phi\left(\frac{iy}{\sqrt n}\right)^{n-2n_0}\cdot \frac{\delta_{n_{0}}}{2\pi}
\int_{-\pi/\delta_{n_{0}}}^{\pi/\delta_{n_{0}}}\left|\phi_{n_{0}}\left(s+iy\sqrt{\frac{n_0}{n}}\right)\right|^{2}\, ds. \label{eq:p_{n,k} bond 1}
\end{align}
By the definition in Lemma \ref{lem:phi-phi_n-extension}, we have
\begin{align*}
\phi_{n_{0}}\left(s+iy\sqrt{\frac{n_0}{n}}\right)
&=\sum_{j\in\mathbb{Z}}\exp\left(i\left(s+iy\sqrt{\frac{n_0}{n}}\right)x_{n_0,j}\right)\,p_{n_0,j}\\
&=\sum_{j\in \mathbb{Z}} e^{isx_{n_0,j}}\exp\left(-yx_{n_0,j}\sqrt{\frac{n_0}{n}}\right)\,p_{n_0,j}.
\end{align*}
Note that the last quantity is the Fourier transform of the sequence
$$
\left\{\exp\left(-yx_{n_0,j}\sqrt{\frac{n_0}{n}}\right)\,p_{n_0,j}\right\}_{j\in\mathbb Z}.
$$
By Plancherel's Theorem \ref{thm:Plancherel-b}, we have 
\begin{equation}\label{eq:q_n,k-Plancherel}
 \frac{\delta_{n_0}}{2\pi}\int_{-\pi/\delta_{n_0}}^{\pi/\delta_{n_0}}\left|\phi_{n_0}\left(s+iy\sqrt{\frac{n_0}{n}}\right)\right|^2ds=\sum_{j\in \mathbb{Z}}\exp\left(-2yx_{n_0,j}\sqrt{\frac{n_0}{n}}\right)\,p_{n_0,j}^2.   
\end{equation}
Recall $\alpha>2$ and $\beta=\alpha/(\alpha-1)$. Apply the H\"older inequality to the right-hand side of \eqref{eq:q_n,k-Plancherel} to obtain
\begin{align}
&\sum_{j\in \mathbb{Z}}\exp\left(-2yx_{n_0,j}\sqrt{\frac{n_0}{n}}\right)
\,p_{n_0,j}^2\notag\\
=&\sum_{j\in \mathbb{Z}}
\exp\left
(-2yx_{n_0,j}\sqrt{\frac{n_0}{n}}\right)q_{n_0,j}^{2/\beta}\cdot\frac{p_{n_0,j}^2}{q_{n_0,j}^{2/\beta}} \notag\\
\leq&\left[\sum_{j\in \mathbb{Z}}\exp\left(-\frac{2\alpha yx_{n_0,j}}{\alpha-2}\sqrt{\frac{n_0}{n}}\right)q_{n_0,j}^{\frac{2\alpha}{\beta(\alpha-2)}}\right]^{(\alpha-2)/\alpha}
\left[\sum_{j\in \mathbb{Z}}\frac{p_{n_0,j}^{\alpha}}{q_{n_0,j}^{\alpha-1}}\right]^{2/\alpha}. \label{eq:Plancherel expansion}
\end{align}
 We now establish an estimate for the first term of \eqref{eq:Plancherel expansion}.
\begin{align}
&\sum_{j\in \mathbb{Z}}\exp\left(-\frac{2\alpha yx_{n_0,j}}{\alpha-2}\sqrt{\frac{n_0}{n}}\right)q_{n_0,j}^{\frac{2\alpha}{\beta(\alpha-2)}}\notag\\
=&\left(\frac{\delta_{n_0}}{\sqrt{2\pi}}\right)^{\frac{2\alpha}{\beta(\alpha-2)}}\sum_{j\in \mathbb{Z}}\exp\left(-\frac{\alpha x_{n_0,j}^2}{\beta(\alpha-2)}-\frac{2\alpha yx_{n_0,j}}{\alpha-2}\sqrt{\frac{n_0}{n}}\right)\notag\\
=&\left(\frac{\delta_{n_0}}{\sqrt{2\pi}}\right)^{\frac{2\alpha}{\beta(\alpha-2)}}
\sum_{j\in \mathbb{Z}}\exp\left(-\frac{\alpha}{\beta(\alpha-2)}\left(x_{n_0,j}+\beta y\sqrt{\frac{n_0}{n}}\right)^2\right)\cdot\exp\left(\frac{\alpha\beta}{\alpha-2}\frac{n_0}{n}y^2
\right)\notag\\
\leq&\left(\frac{1}{\sqrt{2\pi}}\right)^{\frac{2\alpha}{\beta(\alpha-2)}}
\delta_{n_0}^{\frac{\alpha}{\alpha-2}}\left(\sqrt{\frac{\pi\beta(\alpha-2)}{\alpha}}+ 2\delta_{n_{0}}\right)\cdot\exp\left(\frac{\alpha\beta}{\alpha-2}\frac{n_0}{n}y^2
\right)\notag\\
=&\left(\frac{\delta_{n_0}}{\sqrt{2\pi}}\right)^{\frac{\alpha}{\alpha-2}}\left(\sqrt{\frac{\beta(\alpha-2)}{2\alpha}}+ \frac{2\delta_{n_{0}}}{\sqrt{2\pi}}\right)\cdot\exp\left(\frac{\alpha\beta}{\alpha-2}\frac{n_0}{n}y^2\right).\label{eq:boud of the first term}
\end{align}
To see the above inequality, we set $y_{n_{0},j}=x_{n_{0},j}+\beta y\sqrt{n_{0}/n}$ and $c=\alpha/\beta(\alpha-2)$. Since $e^{-c x^{2}}$ is increasing on $(-\infty, 0)$ and decrasing on $(0,\infty)$, we can obtain 
$$
\delta_{n_{0}}\sum_{j\in\mathbb{Z}}\exp\left(-cy_{n_{0},j}^{2}\right)\leq  \int_{-\infty}^{\infty} e^{-cx^{2}}\,dx +2\delta_{n_{0}}=\sqrt{\frac{\pi}{c}}+ 2\delta_{n_{0}}.
$$
Recall that $A_{n,\alpha}=\sum_{k\in \mathbb{Z}}p_{n,k}^{\alpha}/q_{n,k}^{\alpha-1}$. So the second term of \eqref{eq:Plancherel expansion} is $A_{n_{0},\alpha}^{2/\alpha}$. Since $\psi(t)=\mathbb{E}e^{tX}e^{-\beta t^{2}/2}=\phi(-it)e^{-\beta t^{2}/2}$, it is easy to check that 
\begin{equation}\label{eq:phi-psi}
\phi\left(\frac{iy}{\sqrt n}\right)=\psi\left(-\frac{y}{\sqrt n}\right)\exp\left(\frac{\beta y^2}{2n}\right). 
\end{equation}
Denote
$$
C_{1}:=\frac{\delta_{n_{0}}}{\sqrt{2\pi }}\left(\sqrt{\frac{\beta(\alpha-2)}{2\alpha}}+ \frac{2\delta_{n_{0}}}{\sqrt{2\pi}}\right)^{\frac{\alpha-2}{\alpha}} A_{n_{0},\alpha}^{2/\alpha}.
$$
Then combing \eqref{eq:p_{n,k} bond 1}, \eqref{eq:q_n,k-Plancherel}, \eqref{eq:Plancherel expansion}, \eqref{eq:boud of the first term} and \eqref{eq:phi-psi}, we can obtain
\begin{align}   
p_{n,k}
&\leq C_{1}\phi\left(\frac{iy}{\sqrt n}\right)^{n-2n_0}\exp\left(\frac{ \beta n_{0} y^{2}}{n}+yx_{n,k}\right) \notag\\
&\leq 	C_{1}\psi\left(-\frac{y}{\sqrt n}\right)^{n-2n_0}\exp\left(\frac{ \beta y^{2}}{2}+yx_{n,k}\right) \label{eq:psi-replace}\\
&\leq C_{1}\psi\left(\frac{x_{n,k}}{\beta\sqrt n}\right)^{n-2n_0}\exp\left(-\frac{x_{n,k}^{2}}{2\beta}\right).\label{eq:minimal}
\end{align}	
For each $x_{n,k}\in \mathcal{L}_{n}$, inequality \eqref{eq:psi-replace} holds for all $y \in \mathbb{R}$. Then we obtain inequality \eqref{eq:minimal} by selecting $y=-x_{n,k}/\beta$.

\textbf{Case 2: $1<\alpha\leq 2,~ n\geq \beta n_0$.}  Combine \eqref{eq:inversion 1}, \eqref{eq:Fourier bound} and \eqref{eq:phi-phi_n} with $t:=s\sqrt{n/n_0}$ to obtain
\begin{align}
p_{n,k}&\leq e^{yx_{n,k}}
\phi\left(\frac{iy}{\sqrt n}\right)^{n-\beta n_0}\cdot \frac{\delta_{n}}{2\pi}
\int_{-\pi/\delta_n}^{\pi/\delta_n}\left|\phi\left(\frac{t}{\sqrt{n}}+\frac{iy}{\sqrt{n}}\right)\right|^{\beta n_0}\, dt \notag\\
&=e^{yx_{n,k}}
\phi\left(\frac{iy}{\sqrt n}\right)^{n-\beta n_0}\cdot \frac{\delta_{n_{0}}}{2\pi}
\int_{-\pi/\delta_{n_{0}}}^{\pi/\delta_{n_{0}}}\left|\phi_{n_{0}}\left(s+iy\sqrt{\frac{n_0}{n}}\right)\right|^{\beta}\, ds. \label{eq:p_{n,k} bond 2}
\end{align}
As shown in the analysis of Case 1, $\phi_{n_0}(s+iy\sqrt{n_0/ n})$ is the Fourier transform of the sequence
$$
\left\{\exp\left(-yx_{n_0,j}\sqrt{\frac{n_0}{n}}\right)\,p_{n_0,j}\right\}_{j\in\mathbb Z}.
$$
Since $1<\alpha\leq 2$ and $1/\alpha + 1/\beta=1$, we apply the Hausdorff--Young inequality in Theorem \ref{thm:HYI-b} to obtain 
\begin{equation}\label{eq:p_n,k Hausdorff-Young}
\frac{\delta_{n_0}}{2\pi}\int_{-\pi/\delta_{n_0}}^{\pi/\delta_{n_0}}\left|
\phi_{n_0}\left(t+iy\sqrt{\frac{n_0}{n}}\right)\right|^\beta dt\leq\left(
\sum_{j\in \mathbb{Z}} \exp\left(-\alpha y x_{n_0,j}\sqrt{\frac{n_0}{n}}\right)p_{n_0,j}^{\alpha}
\right)^{\beta/\alpha}.
\end{equation}
Recall $A_{n_{0},\alpha}=\sum_{j\in \mathbb{Z}}p_{n_0,j}^{\alpha}/q_{n_0,j}^{\alpha-1}$. We now establish an estimate for the right-hand side of \eqref{eq:p_n,k Hausdorff-Young}. 
\begin{align}
\sum_{j \in \mathbb Z}\exp\left(-\alpha y x_{n_0,j}\sqrt{\frac{n_0}{n}}\right)p_{n_0,j}^{\alpha} 
&=\sum_{j \in \mathbb Z}\exp\left(-\alpha y x_{n_0,j}\sqrt{\frac{n_0}{n}}\right)q_{n_0,j}^{\alpha-1}\cdot\frac{p_{n_0,j}^{\alpha}}{q_{n_0,j}^{\alpha-1}} \notag\\
&=\left(\frac{\delta_{n_0}}{\sqrt{2\pi}}\right)^{\alpha-1}\sum_{j \in \mathbb Z}\exp\left(-\frac{\alpha-1}{2}x_{n_0,j}^{2}-\alpha y x_{n_0,j}\sqrt{\frac{n_0}{n}}\right)\frac{p_{n_0,j}^{\alpha}}{q_{n_0,j}^{\alpha-1}} \notag\\
&\leq\left(\frac{\delta_{n_0}}{\sqrt{2\pi}}\right)^{\alpha-1} A_{n_{0},\alpha} \cdot\exp\left( \frac{n_{0}\alpha^{2}y^{2}}{2n(\alpha-1)}\right).\label{eq:bound for right}
\end{align}
Recall that $\beta=\alpha/(\alpha-1)$. Denote 
$$
C_{2}:=\left[\left(\frac{\delta_{n_0}}{\sqrt{2\pi}}\right)^{\alpha-1} A_{n_{0},\alpha}\right]^{\beta/\alpha}=\frac{\delta_{n_0}}{\sqrt{2\pi}}A_{n_{0},\alpha}^{1/(\alpha-1)}.
$$
Then we combine \eqref{eq:p_{n,k} bond 2}, \eqref{eq:p_n,k Hausdorff-Young}, \eqref{eq:bound for right} and \eqref{eq:phi-psi} to obtain 
\begin{align}   
p_{n,k}
&\leq C_{2}\phi\left(\frac{iy}{\sqrt n}\right)^{n-\beta n_0}\exp\left(\frac{n_{0} \beta^{2}y^{2}}{2n}+yx_{n,k}\right) \notag\\
&\leq 	C_{2}\psi\left(-\frac{y}{\sqrt n}\right)^{n-\beta n_0}\exp\left(\frac{ \beta y^{2}}{2}+yx_{n,k}\right) \label{eq:psi-replace 1}\\
&\leq C_{2}\psi\left(\frac{x_{n,k}}{\beta\sqrt n}\right)^{n-\beta n_0}\exp\left(-\frac{x_{n,k}^{2}}{2\beta}\right).\label{eq:minimal 1}
\end{align}	
For each $x_{n,k}\in \mathcal{L}_{n}$, inequality \eqref{eq:psi-replace 1} holds for all $y \in \mathbb{R}$. Then we can obtain inequality \eqref{eq:minimal 1} by selecting $y=-x_{n,k}/\beta$.
\end{proof}

\begin{coro} \label{coro:large deviation}
Let $\alpha>1$ and $\beta=\alpha/(\alpha-1)$. Suppose $D_{\alpha}(S_{n_{0}}||Z_{n_{0}}) <\infty$ for some $n_{0}$. Then there exists $x_{0}\geq 0$ and $\rho \in (0,1)$ such that, for all $n$ large enough, we have
$$
p_{n,k}\leq \rho^{n}\exp\left(-\frac{x_{n,k}^{2}}{2\beta}\right)\psi\left(
\frac{x_{n,k}}{\beta\sqrt n}\right)^{n/2},\quad|x_{n,k}|\geq x_0 \sqrt{n}.
$$
\end{coro}

\begin{proof}	
Recall $\psi(t)=\mathbb{E}e^{tX}e^{-\beta t^{2}/2}$. By Proposition \ref{prop:sub-Gaussian-large t}, we have
$$
\lim_{|t|\to\infty}\psi(t)=0.
$$
Therefore, there exist $t_0>0$ and $0<\rho_{0}<1$ such that $\psi(t)\le \rho_{0}$ whenever $|t|\geq t_0$. Set $x_{0}=\beta t_0>0$. Then, for $|x_{n,k}|\geq x_0 \sqrt{n}$, we have
$$
\psi\left(\frac{x_{n,k}}{\sqrt n\,\beta}\right)\le \rho_{0}<1.
$$
Take $\sqrt{\rho_0}<\rho<1$. For all $n$ large enough and $|x_{n,k}|\geq x_0 \sqrt{n}$, Proposition \ref{prop:moderate-large deviation} gives
\begin{align*}
p_{n,k} &\leq C\psi\left(\frac{x_{n,k}}{\beta\sqrt{n}}\right)^{n-n_{\beta}}\exp\left(-\frac{x_{n,k}^{2}}{2\beta}\right)\\
&=C\psi\left(\frac{x_{n,k}}{\beta\sqrt{n}}\right)^{n/2-n_{\beta}}\psi\left(\frac{x_{n,k}}{\beta\sqrt{n}}\right)^{n/2}\exp\left(-\frac{x_{n,k}^{2}}{2\beta}\right)\\
&\le C\rho_{0}^{n/2-n_{\beta}} \exp\left(-\frac{x_{n,k}^{2}}{2\beta}\right)\psi\left(\frac{x_{n,k}}{\beta\sqrt{n}}\right)^{n/2}\\
&\le  \rho^{n}\exp\left(-\frac{x_{n,k}^{2}}{2\beta}\right)\psi\left(\frac{x_{n,k}}{\beta\sqrt n}\right)^{n/2}.
\end{align*}
\end{proof}


\section{Bulk estimate of $D_\alpha(S_n\|Z_n)$}\label{The bulk estimate of D_alpha}


Set $M_{n}(s)=\sqrt{(s-2)\log n}$ for $s\in\mathbb Z$ and $s\geq 3$. We define
\begin{equation}\label{eq:I_Mn}
I(M_{n}(s))=\sum_{|x_{n,k}|\leq M_{n}(s)}\frac{p_{n,k}^{\alpha}}{q_{n,k}^{\alpha-1}}.
\end{equation}
The goal of this section is to give an asymptotic expansion of $I(M_{n}(s))$ that relies on the Edgeworth expansion of $p_{n,k}$.


\begin{prop}[\cite{BR10}, Theorem 22.1] \label{prop:llt}
If $\mathbb{E} |X|^{s}<\infty$ for some integer $s\geq 2$, then 
$$
\sup_{k\in\mathbb Z}(1+|x_{n,k}|^{s})|p_{n,k}-\tilde q_{n,k}|=o\big(n^{-\frac{s-1}{2}}\big).
$$
\end{prop}

Here, the Edgeworth expansion $\tilde q_{n,k}$ of $p_{n,k}$ of order $s$ is defined by
$$ 
\tilde q_{n,k}=q_{n,k}+q_{n,k}\sum_{\ell=1}^{s-2}Q_{\ell}(x_{n,k})n^{-\ell/2},
$$
where 
\begin{equation*} 
Q_{\ell}(x)=\sum\frac{1}{t_{1}!\cdots t_{\ell}!}\left(\frac{\gamma_{3}}{3!}\right)^{t_{1}}\cdots\left(\frac{\gamma_{\ell+2}}{(\ell+2)!}\right)^{t_{\ell}}H_{\ell+2v}(x).
\end{equation*}
The sum is taken over $t_{1},\cdots,t_{\ell}\in \mathbb Z_{\ge 0}$ such that $t_{1}+2t_{2}+\cdots+\ell t_{\ell}=\ell$, and we set  $v=t_{1}+t_{2}+\cdots+t_{\ell}$. Here,  $\gamma_{j}$ denotes the $j$-th cumulant of $X$, while $H_j(x)$ is  the Chebyshev-Hermite polynomial of degree $j$ with leading term $x^{j}$ defined by 
$$
\varphi^{(j)}(x)=(-1)^{j}H_j(x)\varphi(x).
$$
The degree of $Q_{\ell}(x)$ is at most $3\ell$ that holds if and only $t_{1}=j $ and $t_{2}=\cdots t_{\ell}=0$. 
One can check that 
\begin{equation*} 
H_j(-x)=(-1)^{j}H_j(x),
\end{equation*}
and therefore
\begin{equation} \label{eq:Q}
Q_\ell(-x)=(-1)^{\ell}Q(x).
\end{equation}

\begin{prop} \label{prop:bulk-expansion}
Let $X_{1},\cdots, X_{n}$ be independent copies of a lattice random variable $X$ with mean zero, variance one and maximal span $h>0$. Suppose $\mathbb{E} |X|^{s} <\infty$ for some $s\in\mathbb Z$ and $s\geq 3$. Then, for $\alpha >1$, we have
$$
I(M_{n}(s))=1+\sum_{j=1}^{\lfloor s/2-1 \rfloor}b_{j}n^{-j}+o\big(n^{-\frac{s-2}{2}}\big).
$$
Here, 
$$
b_{j}=\sum\frac{(\alpha)_m}{m_1!\cdots m_{2j}!}\int_{\mathbb R}Q_1^{m_1}(x)\cdots Q_{2j}^{m_{2j}}(x)\varphi(x)\,dx,
$$
where the sum is taken over $m_{1},\cdots,m_{2j}\in\mathbb Z_{\ge 0}$ such that $m_{1}+2m_{2}+\cdots+2jm_{2j}=2j$, and we set $m=m_{1}+m_2+\cdots+m_{2j}$ and write $(\alpha)_{m}=\alpha(\alpha-1)\cdots(\alpha-m+1)$. 
\end{prop}

\begin{proof}
We proceed the proof in the following three steps.

\textbf{(1) Expansion of $p_{n,k}^{\alpha}/q_{n,k}^{\alpha}$}. Proposition \ref{prop:llt} gives the representation of $p_{n, k}/q_{n, k}$ as follows
\begin{equation}\label{eq:p/q-expansion}
\frac{p_{n,k}}{q_{n,k}}=1+R_{n}(x_{n,k})+\frac{o\big(n^{-\frac{s-1}{2}}\big)}{q_{n,k}(1+|x_{n,k}|^s)},
\end{equation}
where 
\begin{equation}\label{eq:R_n}
R_{n}(x_{n,k})=\sum_{\ell=1}^{s-2}Q_{\ell}(x_{n,k})n^{-\ell/2}.
\end{equation}
For all $1\leq \ell\leq s-2$, we have $\deg Q_{\ell}(x)\leq 3\ell\le 3(s-2)$ and hence there exist some constant $C_\ell>0$ such that $|Q_\ell(x)|\leq C_\ell(1+|x|^{3(s-2)})$ for all $x\in\mathbb R$. Set $C=C_1+\cdots+C_{s-2}$. Then we have 
\begin{align*}
\sup_{|x_{n,k}|\leq M_{n}(s)}|R_{n}(x_{n,k})|&\leq \frac{C}{\sqrt{n}}\sup_{|x_{n,k}|\leq M_{n}(s)}\left(1+|x_{n,k}|^{3(s-2)}\right)\\
&\leq \frac{C}{\sqrt{n}}\big(1+((s-2  )\text{log}~n)^{\frac{3(s-2)}{2}}\big)\to 0\quad \text{as}~n\to \infty.
\end{align*}
For $|x_{n,k}|\leq M_{n}(s)$, we have 
$$
\frac{1}{q_{n, k}}=\frac{\sqrt{2\pi n}}{h}e^{x_{n, k}^2/2}\le \frac{\sqrt{2\pi n}}{h}e^{M_n(s)^2/2}=\frac{\sqrt{2\pi }}{h}n^{\frac{s-1}{2}}.
$$
This gives
$$
\sup_{|x_{n,k}|\leq M_{n}(s)}\frac{o\big(n^{-\frac{s-1}{2}}\big)}{q_{n,k}(1+|x_{n,k}|^s)}\to 0 \quad\text{as}~~n\to \infty.
$$
Since the last two terms in the representation \eqref{eq:p/q-expansion} are asymptotically small, we can write
$$
\frac{p_{n,k}^{\alpha}}{q_{n,k}^{\alpha}}=(1+R_{n}(x_{n,k}))^{\alpha}+\frac{o\big(n^{-\frac{s-1}{2}}\big)}{q_{n,k}(1+|x_{n,k}|^s)}.
$$
We further apply the Taylor expansion with the Lagrange remainder for $(1+R_{n}(x_{n,k}))^{\alpha}$ around zero to obtain the following expansion
\begin{equation}\label{eq:p/q-alpha-expanion}
 \frac{p_{n,k}^{\alpha}}{q_{n,k}^{\alpha}}=1+\sum_{m=1}^{s-2}\frac{(\alpha)_{m}}{m!}R_{n}(x_{n, k})^{m}+C_{s}(x_{n,k})n^{-\frac{s-1}{2}}+\frac{o\big(n^{-\frac{s-1}{2}}\big)}{q_{n,k}(1+|x_{n,k}|^s)},
\end{equation}
where
\begin{equation}\label{eq:C_s}
C_{s}(x_{n,k})=\frac{(\alpha)_{s-1}}{(s-1)!}(1+\xi_{n,k})^{\alpha-s+1}\left(\sqrt{n}R_{n}(x_{n,k})\right)^{s-1}, \quad 0<\xi_{n,k}<R_{n}(x_{n,k}).
\end{equation}  

\textbf{(2) Expansion of $I(M_n(s))$}. Our definition of $I(M_n(s))$ in \eqref{eq:I_Mn} and the representation of $p_{n, k}^\alpha/q_{n, k}^\alpha$ in \eqref{eq:p/q-alpha-expanion} give
\begin{align}
I(M_{n}(s)) 
&=\sum_{|x_{n,k}|\leq M_{n}(s)}q_{n,k}+\sum_{m=1}^{s-2}\frac{(\alpha)_{m}}{m!}\sum_{|x_{n,k}|\leq M_{n}(s)}R_{n}(x_{n, k})^{m}q_{n,k} \notag\\
&\quad+\left[\sum_{|x_{n,k}|\leq M_{n}(s)} C_{s}(x_{n,k})q_{n,k}\right]n^{-\frac{s-1}{2}}+\left[\sum_{|x_{n,k}|\leq M_{n}(s)}\frac{1}{1+|x_{n,k}|^s}\right]o\big(n^{-\frac{s-1}{2}}\big).\label{eq:I_Mn-expansion}
\end{align}
we next estimate each term on the right hand side of \eqref{eq:I_Mn-expansion}. For $\ell \geq 0$, there exists some $\widetilde{C}_\ell>0$, such that 
\begin{equation}\label{eq:gaussian bound}
\int_{|x|>M_n(s)} |x|^\ell \varphi(x)\,dx
\leq \widetilde{C}_\ell M_n(s)^{\ell-1} e^{-M_n(s)^2/2}=O\left(\big(\sqrt{(s-2)\log n}\big)^{\ell-1}n^{-\frac{s-2}{2}}\right).
\end{equation}
Combine \eqref{eq:normalizing-const}, Lemma \ref{lem:gaussian-moment-tail} and \eqref{eq:gaussian bound} with $\ell=0$ to obtain
\begin{align}
\sum_{|x_{n,k}|\leq M_n(s)} q_{n,k}&=\sum_{k\in\mathbb{Z}} q_{n,k}-\sum_{|x_{n,k}|> M_{n}(s)}q_{n,k} \notag\\
&=1+O(e^{-2\pi^{2}n/h^{2}})-\int_{|x|>M_{n}(s)}\varphi(x)\,dx+O(\delta_{n}\varphi(M_{n}(s)))\notag\\
&=1+o\big(n^{-\frac{s-2}{2}}\big). \label{eq:q expansion}
\end{align}
We have shown in the previous part that $|R_{n}(x)|\leq C(1+|x|^{3(s-2)})/\sqrt{n}$ for all $x\in\mathbb R$. Then we apply Lemma \ref{lem:gaussian-moment-tail} and \eqref{eq:gaussian bound} with $\ell=3m(s-2)$ to $R_{n}(x)^{m}$ for $m=1,2,\cdots,s-2$ and obtain
\begin{align}
&~~~\sum_{|x_{n,k}|> M_{n}(s)}R_{n}(x_{n, k})^{m}q_{n,k} \notag\\
&\le \frac{C^m}{n^{m/2}}\sum_{|x_{n,k}|> M_{n}(s)}(1+|x_{n,k}|^{3(s-2)})^mq_{n, k}\notag\\
&\le \frac{C^m}{n^{m/2}}\left[\int_{|x|\ge M_n(x)} (1+|x|^{3(s-2)})^m\varphi(x)dx+ O\left(\delta_{n}M_{n}(s)^{3m(s-2)}\varphi(M_{n}(s))\right)\right]\notag\\
&=\frac{C^m}{n^{m/2}}\left[\big(\sqrt{(s-2)\log n}\big)^{3m(s-2)-1}n^{-\frac{s-2}{2}}+o\big(n^{-\frac{s-2}{2}}\big)\right]\notag\\
&=o\big(n^{-\frac{s-2}{2}}\big). \label{eq:R bound 1}
\end{align}
Then apply Lemma \ref{lem:gaussian-moment-a} to $R_{n}(x)^{m}$ and combine \eqref{eq:R bound 1} to obtain 
\begin{align}
\sum_{|x_{n,k}|\leq M_{n}(s)}R_{n}(x_{n, k})^{m}q_{n,k}
 &= \sum_{k\in \mathbb{Z}}R_{n}(x_{n, k})^{m}q_{n,k}-   \sum_{|x_{n,k}|> M_{n}(s)}R_{n}(x_{n, k})^{m}q_{n,k}\notag\\
=&\int_{\mathbb{R}}R_{n}(x)^{m} \varphi(x)~dx+o\big(n^{-\frac{s-2}{2}}\big)\label{eq:R bound 2}.
\end{align}
As defined in \eqref{eq:C_s}, one can see that $C_{s}(x_{n,k})$ is a polynomial of degree at most $3(s-1)(s-2)$, which is even. Therefore there exists some $\widetilde{C}>0$ such that $|C_{s}(x_{n,k})|\leq \widetilde{C}\big(1+x_{n,k}^{3(s-1)(s-2)}\big)$. Then we can apply Lemma \ref{lem:gaussian-moment-a} to obtain
\begin{align*}
\sum_{|x_{n,k}|\leq M_{n}(s)} C_{s}(x_{n,k})q_{n,k} &\leq \widetilde{C} \sum_{k \in \mathbb{Z}} \big(1+x_{n,k}^{3(s-1)(s-2)}\big)q_{n,k}\\
&=\widetilde{C}\int_{\mathbb{R}}\left(1+x^{3(s-1)(s-2)}\right)\varphi(x)\,dx+O(e^{-\pi^{2}n/h^{2}}).
\end{align*}
Since Gaussian random variables have finite moments, the above quantity is finite, and therefore 
\begin{equation} \label{eq:remainder}
 \left[\sum_{|x_{n,k}|\leq M_{n}(s)} C_{s}(x_{n,k})q_{n,k}\right]n^{-\frac{s-1}{2}}=O\big(n^{-\frac{s-1}{2}}\big).
\end{equation}
For the last term in \eqref{eq:I_Mn-expansion}, we can obtain 
\begin{equation*}
    \sum_{|x_{n,k}|\leq M_{n}(s)}\frac{1}{1+|x_{n,k}|^{s}}\leq \frac{1}{\delta_{n}}\sum_{k \in \mathbb{Z}}\frac{\delta_{n}}{1+|x_{n,k}|^{2}}\leq2+\frac{1}{\delta_{n}}\int_{-\infty}^{\infty}\frac{1}{1+x^{2}}\,dx=O\big(\sqrt{n}\big),
\end{equation*}
and hence 
\begin{equation}\label{eq:1+x^s summation}
    \left[\sum_{|x_{n,k}|\leq M_{n}(s)}\frac{1}{1+|x_{n,k}|^s}\right]o\big(n^{-\frac{s-1}{2}}\big)=o\big(n^{-\frac{s-2}{2}}\big).
\end{equation}
By \eqref{eq:I_Mn-expansion}, \eqref{eq:q expansion}, \eqref{eq:R bound 2}, \eqref{eq:remainder} and \eqref{eq:1+x^s summation}, we obtain the the following expansion
\begin{equation}\label{eq:I_M continous}
I(M_{n}(s))=1+\sum_{m=1}^{s-2}\frac{(\alpha)_{m}}{m!}\int_{\mathbb R} R_{n}(x)^{m} \varphi(x)\,dx + o\big(n^{-\frac{s-2}{2}}\big). 
\end{equation}

\textbf{(3) Further reduction.} By the definition of $R_{n}(x)$ in \eqref{eq:R_n},  we have
$$
R_{n}(x)^{m}=\sum_{m_1+\cdots+m_{s-2}=m}\frac{m!}{m_1!\cdots m_{s-2}!}n^{-N/2}Q_1^{m_1}(x)\cdots Q_{s-2}^{m_{s-2}}(x),
$$
where $N=m_1+2m_2+\cdots+(s-2)m_{s-2}.$ Then the main term of $I(M_n(s))-1$ can be written as
\begin{equation}\label{eq:the main term of I(M_n(s))-1}
\sum_{N=1}^{(s-2)^2} n^{-N/2}\sum\frac{(\alpha)_m}{m_1!\cdots m_{s-2}!}
\int_{-\infty}^{\infty}Q_1^{m_1}(x)\cdots Q_{s-2}^{m_{s-2}}(x)\varphi(x)\,dx,	
\end{equation}
where the inner sum is taken over $m_{1},\cdots,m_{2j} \in \mathbb{Z}_{\geq 0}$ such that $m_{1}+2m_{2}+\cdots+(s-2)m_{s-2}=N$ and $1\leq m:=m_1+\ldots+m_{s-2}\leq s-2$. 
Using property \eqref{eq:Q}, we have 
$$Q_1^{m_1}(-x)\cdots Q_{s-2}^{m_{s-2}}(-x)=(-1)^{N}Q_1^{m_1}(x)\cdots Q_{s-2}^{m_{s-2}}(x).
$$  
Hence, if $N$ is odd, $Q_1^{m_1}(x)\cdots Q_{s-2}^{m_{s-2}}(x)$ is also odd. Correspondingly, the integrals in \eqref{eq:the main term of I(M_n(s))-1} vanish. Therefore, we can assume that $N=2j$ for $~1\leq j\leq \lfloor (s-2)^{2}/2\rfloor$ and we necessarily have $m_{\ell}=0$ for $\ell>2j$. Then the main term of $I(M_n(s))-1$ in \eqref{eq:the main term of I(M_n(s))-1} can be further simplified as
\begin{align*}
\sum_{j=1}^{\lfloor (s-2)^{2}/2\rfloor} n^{-j}\sum\frac{(\alpha)_m}{m_1!\cdots m_{2j}!}\int_{-\infty}^{\infty}Q_1^{m_1}(x)\cdots Q_{2j}^{m_{2j}}(x)\varphi(x)\,dx,
\end{align*}
 the inner sum is taken over all $m_{1},\cdots,m_{2j} \in \mathbb{Z}_{\geq 0}$ such that $m_{1}+2m_{2}+\cdots+2jm_{2j}=2j$ and $1\leq m=m_1+\ldots+m_{2j}\leq s-2$.	We denote 
$$
b_{j}=\sum\frac{(\alpha)_m}{m_1!\cdots m_{2j}!}\int_{-\infty}^{\infty}Q_1^{m_1}(x)\cdots Q_{2j}^{m_{2j}}(x)\varphi(x)\,dx.
$$
Since the error in \eqref{eq:I_M continous} is $o\big(n^{-\frac{s-2}{2}}\big)$, it suffices to keep the first $\lfloor (s-2)/2\rfloor$  terms, and we obtain 
$$
I(M_{n}(s))=1+\sum_{j=1}^{\lfloor s/2-1 \rfloor}b_{j}n^{-j}+o\big(n^{-\frac{s-2}{2}}\big).
$$
This concludes the proof.
\end{proof}

    
\section{Proof of Theorem \ref{thm:main}} \label{sec:proof of main result}

We first recall the notations that will be used throughout this section. Let $S_{n}$ be the normalized sum of i.i.d. lattice random variables $X_{1}, \cdots, X_{n}$.  For any $k \in \mathbb Z$, we set $x_{n,k}=(na+kh)/\sqrt{n}$ and $p_{n,k}=\mathbb P(S_{n}=x_{n,k})$. Then $\delta_{n}=h/\sqrt{n }$ is the maximal span of $S_{n}$. Let $\varphi(x)$ denote the standard Gaussian density and define
the quantized Gaussian random variable  $Z_n$ with distribution proportional to $\{q_{n,k}\}_{k\in\mathbb Z}$, where $q_{n,k}=\delta_{n}\varphi(x_{n,k})$. For $\alpha>1$ and $\beta=\alpha/(\alpha-1)$, $\alpha$-R\'enyi divergence of $S_{n}$ and $Z_{n}$ is defined by 
\begin{equation}\label{eq:divergence}
D_{\alpha}(S_{n}||Z_{n})=\frac{1}{\alpha-1}\log \sum_{k \in \mathbb{Z}}\frac{p_{n,k}^{\alpha}}{q_{n,k}^{\alpha-1}}+\log\left(\sum_{k\in\mathbb{Z}}q_{n,k}\right),
\end{equation}
and denote 
\begin{equation}\label{eq:A_n denfinition}
A_{n,\alpha}=\sum_{k \in \mathbb{Z}}\frac{p_{n,k}^{\alpha}}{q_{n,k}^{\alpha-1}}.
\end{equation}

\subsection{Sufficiency} \label{ssec:sufficiency}
\begin{proof}
Using identities \eqref{eq:divergence} and  \eqref{eq:normalizing-const}, one can see that $D_{\alpha}(S_{n}||Z_{n})\to 0$ is equivalent to $A_{n,\alpha}\to 1$. We write $M_n(\eta)=\sqrt{(\eta-2)\log n}$ for $\eta \in \mathbb{Z}_{+}$. We decompose $A_{n,\alpha}$ as 
\begin{equation}\label{eq:A_n expansion}
A_{n,\alpha}=\sum_{|x_{n,k}|\leq M_{n}(\eta)}\frac{p_{n,k}^{\alpha}}{q_{n,k}^{\alpha-1}} +\sum_{x_{n,k}<- M_{n}(\eta)}\frac{p_{n,k}^{\alpha}}{q_{n,k}^{\alpha-1}}+\sum_{x_{n,k}> M_{n}(\eta)}\frac{p_{n,k}^{\alpha}}{q_{n,k}^{\alpha-1}}.
\end{equation}

Since $D_{\alpha}(S_{n_{0}}||Z_{n_{0}})<\infty$ for some $n_{0}$, Proposition \ref{prop:sub-Gaussian-a} yields $\mathbb{E}e^{cX^{2}}<\infty$ for any $0<c<1/(2\beta)$. Thus $X$ has finite absolute moments of all orders and, particularly, $\mathbb{E}|X|^{\eta}<\infty$ for all $\eta\in \mathbb Z_{\ge 1}$. For the first term of \eqref{eq:A_n expansion}, Proposition \ref{prop:bulk-expansion} gives 
\begin{equation}\label{eq:1st term}
\sum_{|x_{n,k}|\leq M_{n}(\eta)}\frac{p_{n,k}^{\alpha}}{q_{n,k}^{\alpha-1}}=1+\sum_{j=1}^{\lfloor \eta/2-1 \rfloor}b_{j}n^{-j}+o\big(n^{-\frac{\eta-2}{2}}\big).
\end{equation}
		
Next we establish an estimate for the second term of \eqref{eq:A_n expansion}. The third term can be handled similarly. Let $x_{0}$ be given in Corollary \ref{coro:large deviation} and let $M_n(\eta)/\sqrt n<x_{1}<x_{0}$. The value of $x_1$ will be determined later. Define
\begin{align*}
\Lambda_1&:=\{k\in \mathbb{Z}:~~ x_{n,k}\le -x_0\sqrt n\}, \\
\Lambda_2&:=\{k\in \mathbb{Z}:\ -x_0\sqrt n\le x_{n,k}<-x_1\sqrt n\},\\
\Lambda_3&:=\{k\in \mathbb{Z}:\ -x_1\sqrt n<x_{n,k}\le -M_n(\eta)\},
\end{align*}
and 
$$
I_j:=\sum_{k\in\Lambda_j}\frac{p_{n,k}^{\alpha}}{q_{n,k}^{\alpha-1}},\qquad j=1,2,3.
$$
Then we can write 
\begin{equation}\label{eq:2nd term}
\sum_{x_{n,k}<- M_{n}(\eta)}\frac{p_{n,k}^{\alpha}}{q_{n,k}^{\alpha-1}}=I_1+I_2+I_3.
\end{equation}

\textbf{(1) Estimate of $I_1$}. Since $|x_{n, k}|\ge x_0\sqrt n$, by Corollary \ref{coro:large deviation}, there exists $\rho \in (0,1)$ such that
$$
p_{n,k}\leq \rho^{n}\exp\left(-\frac{x_{n,k}^2}{2\beta}\right)\psi\left(\frac{x_{n,k}}{\beta\sqrt n}\right)^{n/2}.
$$
Direct calculations show that
$$
I_1\le\rho^{\alpha n}\left(\frac{\sqrt{2\pi}}{\delta_n}\right)^{\alpha-1}\sum_{k\in\Lambda_1}\psi\left(\frac{x_{n,k}}{\beta\sqrt n}\right)^{\alpha n/2}. 
$$
When $n$ is large enough, we select $\ell$ such that $\ell n_{0}\le \alpha n/2$ and $\ell\geq\alpha$. By the strict sub-Gaussian condition, we know that $\psi(t)=\mathbb{E} e^{tX}e^{-\beta t^{2}/2} \leq 1$ for all $t \in \mathbb{R}$. Therefore,  
\begin{align*}
\frac{\delta_{n}}{\beta\sqrt n}\sum_{k\in\Lambda_1}\psi\left(\frac{x_{n,k}}{\beta\sqrt n}\right)^{\alpha n/2}
&\le \frac{\delta_{n}}{\beta\sqrt n}\sum_{k\in\Lambda_1}\psi\left(\frac{x_{n,k}}{\beta\sqrt n}\right)^{\ell n_0}\\
&\le \frac{\delta_{n}}{\beta\sqrt n}\sum_{k\in\mathbb Z}\psi\left(\frac{x_{n,k}}{\beta\sqrt n}\right)^{\ell n_0} \xrightarrow{n\to\infty}\int_{\mathbb R}\psi(t)^{\ell n_{0}}\,dt<\infty.
\end{align*}
The integrability of $\psi(t)^{\ell n_{0}}$ follows from Corollary \ref{coro:integrability}. Then there exists $C_{1}>0$ such that for large $n$ we have
\begin{equation}\label{eq:I_1}
I_1\leq \rho^{\alpha n}\left(\frac{\sqrt{2\pi}}{\delta_n}\right)^{\alpha-1}\frac{\beta\sqrt n}{\delta_n}\cdot 2\int_{\mathbb R}\psi(t)^{\ell n_{0}}\,dt\leq C_{1}\rho^{\alpha n} n^{(\alpha+1)/2}.
\end{equation}

\textbf{(2) Estimate of $I_2$}. For $k\in\Lambda_2$, we have $x_{n,k}/\beta\sqrt{n} \in  \left[-x_0/\beta,-x_1/\beta\right]$, which does not contain $0$. Note $\psi(t)=\mathbb{E} e^{tX}e^{-\beta t^{2}/2}$ is continuous and $0<\psi(t)<1$ for all $t\ne 0$. We define
\begin{equation} \label{under 1}
\rho_1:=\max_{t\in[-x_0/\beta,-x_1/\beta]}\psi(t)<1.
\end{equation}
By Proposition \ref{prop:moderate-large deviation}, there exists some constant $C>0$  such that
$$
p_{n,k}\le C\exp\left(-\frac{x_{n,k}^{2}}{2\beta}\right)\psi\left(\frac{x_{n,k}}{\beta\sqrt n}\right)^{n-n_\beta}.
$$
Combing this with (\ref{under 1}), there exists $C_{2}>0$ such that 
\begin{align}
 I_2
 &\le C^{\alpha}\left(\frac{\sqrt{2\pi}}{\delta_n}\right)^{\alpha-1} \sum_{k\in\Lambda_2}\psi\left(\frac{x_{n,k}}{\beta\sqrt n}
\right)^{\alpha(n-n_\beta)}  \notag\\
&\leq C^{\alpha}\left(\frac{\sqrt{2\pi}}{\delta_n}\right)^{\alpha-1} \left(\frac{(x_{0}-x_{1})\sqrt{n}}{\delta_{n}}+1\right)\rho_{1}^{\alpha(n-n_\beta)} \notag\\
&\leq C_{2} n^{(\alpha+1)/2}	\rho_1^{\alpha(n-n_\beta)}. \label{eq:I_2}
\end{align}
 	
\textbf{(3) Estimate of $I_3$}. By Proposition \ref{prop:moderate-large deviation}, there exists some constant $C>0$  such that
\begin{equation}\label{eq:upper bound}
p_{n,k}\le C\exp\left(-\frac{x_{n,k}^{2}}{2\beta}\right)\psi\left(\frac{x_{n,k}}{\beta\sqrt n}\right)^{n-n_\beta}.
\end{equation}
We first estimate $\psi\left(x_{n,k}/\beta\sqrt n\right)$. Since $D(S_{n_0}||Z_{n_0})<\infty$ for some integer $n_0$, we know from Lemma \ref{lem:phi-phi_n-extension} that $\phi(t)=\mathbb{E}e^{itX}$ is entire. Clearly, $\phi(0)=1$. Hence for sufficiently small $r>0$ we have $\phi(it)\neq0$ whenever $|t|\le r$. Then the function $f(t)=\log \phi(it)$ is analytic for $|t|\le r$. One can check that 
$$
f(0)=0,\quad f'(0)=\mathbb{E}X=0,\quad f''(0)=\text{Var}(X)=1.
$$
Hence, for sufficiently small $r>0$, we have $f(t)= t^{2}/2+o(t^{2})$ for $|t|\le r$. Since $\beta=\alpha/(\alpha-1)>1$ for $\alpha>1$, we have
$$
|f(t)|\le\frac{(\beta+1)t^2}{4},\quad |t|\le r.
$$
Consequently, we have for $|t|\le r$ that
$$
|\psi(t)|=|\phi(-it)|e^{-\beta t^2/2} \le e^{(\beta+1)t^2/4}e^{-\beta t^2/2} 
=e^{-(\beta-1)t^2/4}.
$$
Set $x_1=\beta r$. For $k\in \Lambda_3$, we have $x_{n,k}/\sqrt{n}\beta\in [-x_1/\beta, -\sqrt{(\eta-2)\log n/n}]\subseteq [-r, 0]$. Therefore, 
\begin{equation}\label{eq:bound for psi}
\psi\left(\frac{x_{n,k}}{\sqrt{n}\beta}\right)\le\exp\left(-\frac{(\beta-1)x_{n,k}^2}{4\beta^2n}\right).  
\end{equation}
For $n\ge 2n_{\beta}$ with $n_{\beta}=\max\{2,\beta\}n_{0}$, we have $
\alpha(n-n_{\beta})\ge \alpha n/2.$ Since $0<\psi(t)\le 1$ for $|t|<r$,  combining \eqref{eq:upper bound} and \eqref{eq:bound for psi}, there exists $C_{3} >0$ such that 
\begin{align}
I_3	
&\le C^{\alpha}\left(\frac{\sqrt{2\pi}}{\delta_n}\right)^{\alpha-1}\sum_{k\in \Lambda_3}\psi\left(\frac{x_{n,k}}{\sqrt{n}\beta}\right)	^{\alpha(n-n_{\beta})} \notag \\
&\le C^{\alpha}\left(\frac{\sqrt{2\pi}}{\delta_n}\right)^{\alpha-1}\sum_{k\in \Lambda_{3}}\exp\left(
-\frac{\alpha n}{2}\cdot\frac{(\beta-1)x_{n,k}^2}{4n\beta^2}\right) \notag \\
&=C^{\alpha}\left(\frac{\sqrt{2\pi}}{\delta_n}\right)^{\alpha-1}\sum_{k\in \Lambda_{3}}\exp\left(
-\frac{x_{n,k}^2}{8\beta}\right) \notag \\
&\leq C^{\alpha}\left(\frac{\sqrt{2\pi}}{\delta_n}\right)^{\alpha-1} \frac{1}{\delta_{n}}\int_{-\infty}^{-M_{n}(\eta)+\delta_{n}}\exp\left(-\frac{x^2}{8\beta}\right)\,dx \label{eq:sum-integral}\\
&\leq  2\sqrt{\beta} \left(\frac{\sqrt{2\pi}C}{\delta_n}\right)^{\alpha} \exp\left(-\frac{(M_{n}(\eta)-\delta_{n})^{2}}{8\beta}\right) \label{eq:guassian tail intergral}\\
&\leq C_{3}n^{-\frac{\eta-2}{8\beta}+\frac{\alpha}{2}}. \label{eq:I_3}
\end{align}
Inequality \eqref{eq:sum-integral} follows from the monotonicity of $e^{-x^{2}/(8\beta)}$ and inequality \eqref{eq:guassian tail intergral} follows from a standard Gaussian tail estimate. We select $\eta$ large enough such that $-(\eta-2)/(8\beta)+\alpha/2 <-\eta/(16\beta)$. Combining \eqref{eq:2nd term}, \eqref{eq:I_1}, \eqref{eq:I_2} and \eqref{eq:I_3}, we have
\begin{equation}\label{eq:I_1+I_2+I_3}
\sum_{x_{n,k}<- M_{n}(\eta)}\frac{p_{n,k}^{\alpha}}{q_{n,k}^{\alpha-1}}=I_{1}+I_{2}+I_{3}=o\big(n^{-\frac{\eta}{16\beta}}\big).
\end{equation}
The same estimate holds for the third term of \eqref{eq:A_n expansion}. Combining \eqref{eq:A_n expansion}, \eqref{eq:1st term} and \eqref{eq:I_1+I_2+I_3}, we can obtain
\begin{equation}\label{A_n to 1}
A_{n,\alpha}=1+\sum_{j=1}^{\lfloor \eta/(16\beta) \rfloor}b_{j}n^{-j}+o\big(n^{-\frac{\eta}{16\beta}}\big)\to 1 \quad \text{as}~ n\to \infty.
\end{equation}

Moreover, for any given integer $s\geq 3$, we can choose a sufficiently large $\eta$, such that $(s-2)/2<\eta/(16\beta)<(\eta-2)/(8\beta)-\alpha/2$. Then we combine \eqref{eq:divergence}, \eqref{eq:A_n denfinition}, \eqref{eq:normalizing-const} and \eqref{A_n to 1} to obtain the asymptotic expansion
\begin{align*}
D_{\alpha}(S_{n}||Z_{n})
&=\frac{1}{\alpha-1}\log A_{n,\alpha}+\log\sum_{k\in \mathbb{Z}}q_{n,k}=\frac{1}{\alpha-1} \sum_{j=1}^{\lfloor s/2-1 \rfloor}b_{j}n^{-j}+o(n^{-\frac{s-2}{2}}).
\end{align*}
\end{proof}


\subsection{Necessity} \label{ssec:necessity}
\begin{proof}
Let $\alpha >1$ and $\beta=\alpha/(\alpha-1)$. Suppose that $D_{\alpha}(S_n\|Z_n)\to 0$ as $n\to \infty$. Our goal is to establish the following two claims: (1) $D_{\alpha}(S_{n_{0}}\|Z_{n_{0}}) <\infty$ for some $n_{0}$; and (2) $\mathbb{E}e^{tX}<e^{\beta t^2/2}$ for all $t\in\mathbb R,~t\neq 0$. Claim (1) is immediate. It remains to prove claim (2).

Since $D_{\alpha}(S_{n}\|Z_{n})\to 0$ as $n\to \infty$, there exists some $n_{0}$ such that for $n\geq n_{0}$, $D_{\alpha}(S_{n}\|Z_{n})<\infty$. We know from Proposition \ref{prop:sub-Gaussian-b} that there exists $C>0$ such that for all $t\in \mathbb{R}$,
\begin{equation*}\label{eq:X}
\mathbb{E}e^{tX}\leq C^{1/n}A_{n,\alpha}^{1/n\alpha }e^{\beta t^{2}/2}.
\end{equation*}
Using \eqref{eq:divergence} and \eqref{eq:normalizing-const}, we know that $D_{\alpha}(S_{n}||Z_{n})\to 0$ is equivalent to $A_{n,\alpha}\to 1$ as $n \to \infty$. 
Then we can let $n\to\infty$ and obtain that
$$
\mathbb{E}e^{tX}\le e^{\beta t^2/2} \quad \text{for all}~t\in\mathbb R.
$$

So it suffices to show that $\mathbb{E}e^{-tX}<e^{\beta t^2/2}$ for $t\neq 0$. Suppose, to the contrary, that there exists $t_0\neq 0$ such that $\mathbb{E}e^{-t_0X}=e^{\beta t_0^2/2}$.
Then  we get
\begin{equation}\label{eq:equality for u}
e^{n\beta t_0^2/2}=\left(\mathbb{E}e^{-t_0X}\right)^{n}=\mathbb{E}e^{-t_0\sqrt{n}S_{n}}=\sum_{k\in \mathbb{Z}}p_{n,k} e^{-t_0\sqrt{n}x_{n,k}},
\end{equation}
where $x_{n,k}=(na+kh)/\sqrt{n}$. We define two sequences of functions $\{f_{n}\}_{n=1}^\infty$ and $\{g_{n}\}_{n=1}^\infty$ on $\mathcal{L}_{n}$ by 
\begin{equation*} \label{eq:definition-a,b}
f_{n}(x_{n,k})=\frac{p_{n,k}/q_{n,k}^{1/\beta}}{A_{n,\alpha}^{1/\alpha}},\qquad g_{n}(x_{n,k})
=\frac{q_{n,k}^{1/\beta}e^{-t_{0}\sqrt{n}x_{n,k}}}{B_{n,\beta}^{1/\beta}} ,
\end{equation*}
where
\begin{equation}\label{eq:B_n}
B_{n,\beta}=\sum_{k\in\mathbb{Z}}q_{n,k}e^{-\beta t_{0}\sqrt{n}x_{n,k}}=\big(1+O(e^{-2\pi^2n/h^2})\big)e^{n\beta^{2} t_{0}^{2}/2}.
\end{equation}
The second equality of \eqref{eq:B_n} follows from \eqref{eq:translation}. Then we introduce probability distributions $F_n$, $G_n$ and $Q_{n}$  on $\mathcal{L}_{n}$, which are defined by
\begin{equation*} \label{eq:F-G-Q}
F_n(\{x_{n,k}\})=f_n(x_{n,k})^{\alpha},
\qquad
G_n(\{x_{n,k}\})=g_n(x_{n,k})^{\beta},
\qquad
Q_n(\{x_{n,k}\})=q_{n,k}/C_n.
\end{equation*}
Here, the normalizing constant $C_n$ is
\begin{equation}\label{eq:C_n}
C_n=\sum_{k\in\mathbb{Z}}q_{n,k}=1+O\big(e^{-2\pi^2n/h^{2}}\big),
\end{equation}
where the second equality follows from \eqref{eq:normalizing-const}. We next compare $F_{n}$, $G_{n}$ and $Q_{n}$ in total variation distance. 

\textbf{(1) Estimate of $\|F_{n}-G_{n}\|_{\mathrm{TV}}$.} One can check that
\begin{equation*}\label{eq:a,b-norm}
\|f_{n}\|_{\alpha}^\alpha=\sum_{k\in\mathbb{Z}}f_{n}(x_{n,k})^{\alpha}=1,
\qquad
\|g_{n}\|_{\beta}^\beta=\sum_{k\in\mathbb{Z}}g_{n}(x_{n,k})^{\beta}=1.
\end{equation*}
we combine \eqref{eq:equality for u}, \eqref{eq:B_n} and $A_{n,\alpha}\to1$ as $n \to \infty$ to obtain
\begin{equation*}\label{eq:ab-norm}
\|f_{n}g_{n}\|_{1}=\sum_{k\in\mathbb{Z}}f_{n}(x_{n,k})g_{n}(x_{n,k})=
\frac{\sum_{k\in\mathbb{Z}}p_{n,k} e^{-t_{0}\sqrt{n}x_{n,k}}}{A_{n,\alpha}^{1/\alpha}\cdot B_{n,\beta}^{1/\beta}}=\frac{e^{n\beta t_{0}^{2}/2}}{A_{n,\alpha}^{1/\alpha}\cdot B_{n,\beta}^{1/\beta}}\to 1 .
\end{equation*}
Then we can apply the sequential stability of H\"{o}lder's inequality in Corollary \ref{coro:holder-stabi-seq} to $\{f_{n}\}_{n\geq1}$ and $\{g_{n}\}_{n\geq1}$ and obtain
\begin{equation}\label{eq:F-G}
\|F_{n}-G_{n}\|_{\mathrm{TV}}=\sum_{k\in\mathbb{Z}}\left|f_{n}(x_{n,k})^{\alpha}-g_{n}(x_{n,k})^{\beta}\right|\to  0 \quad \text{as}~n\to \infty.
\end{equation}

\textbf{(2) Estimate of ${\|F_{n}-Q_{n}\|_{\mathrm{TV}}}$.} Combining \eqref{eq:C_n} and $A_{n,\alpha}\to 1$ as $n\to \infty$, we can get
\begin{equation*}\label{eq:a-Q}
\sum_{k\in\mathbb{Z}}f_{n}(x_{n,k})Q_{n}(\{x_{n,k}\})^{1/\beta}=\sum_{k\in \mathbb{Z}}\frac{p_{n,k}/q_{n,k}^{1/\beta}}{A_{n,\alpha}^{1/\alpha}}\frac{q_{n,k}^{1/\beta}}{C_{n}^{1/\beta}}=\frac{1}{A_{n,\alpha}^{1/\alpha}\cdot C_n^{1/\beta}}\to  1 .
\end{equation*}
Then we apply the sequential stability of H\"{o}lder's inequality in Corollary \ref{coro:holder-stabi-seq} to the sequences $\{f_{n}\}_{n\geq 1}$ and $\{Q_{n}^{1/\beta}\}_{n\geq 1}$ and obtain 
\begin{equation}\label{eq:F-Q}
\|F_n-Q_n\|_{\mathrm{TV}}=\sum_{k\in\mathbb{Z}}\left|f_{n}(x_{n,k})^{\alpha}- Q_{n}(\{x_{n,k}\})	\right|\to  0 \quad \text{as}~n\to \infty.
\end{equation}

\textbf{(3) Estimate of $\|Q_{n}-G_{n}\|_{\mathrm{TV}}$.} Recall that $q_{n,k}=\delta_{n}\varphi(x_{n,k})$, where $x_{n,k}=(na+kh)/\sqrt{n}$, $\delta_{n}=h/\sqrt{n}$ and $\varphi(x)$ is the standard Gaussian density. Then we can write
\begin{align*}
Q_{n}(\{x_{n,k}\})&=\frac{q_{n,k}}{\sum_{k\in\mathbb{Z}}q_{n,k}}=\frac{\varphi(x_{n,k})}{\sum_{k\in \mathbb Z}\varphi(x_{n,k})}, \\
G_{n}(\{x_{n,k}\})&=\frac{q_{n,k}e^{-\beta t_{0}\sqrt{n}x_{n,k}}}{\sum_{k\in\mathbb{Z}}q_{n,k}e^{-\beta t_{0}\sqrt{n}x_{n,k}}}=\frac{	\varphi(x_{n,k}+\beta t_{0}\sqrt{n})}{\sum_{j\in\mathbb{Z}}\varphi(x_{n,j}+\beta t_{0}\sqrt{n})}.
\end{align*}
Hence, $Q_{n}$ is the quantized Gaussian distribution on $\mathcal{L}_{n}$ arising from the standard Gaussian $\mathcal{N}(0,1)$, whereas $G_{n}$ is its counterpart induced by the shifted Gaussian $\mathcal{N}(-\beta t_{0} \sqrt{n},\,1)$. Define
$$
\mathcal{T}_n=\left\{x_{n,k}\in\mathcal{L}_{n}:|x_{n,k}|<\frac{|\beta t_{0}|\sqrt{n}}{2}\right\}.
$$
Then it is not hard to see that 
$$
Q_n(\mathcal{T}_n)\to  1
\quad \text{and}\quad 
G_n(\mathcal{T}_n)\to  0 \quad \text{as}~n\to \infty.
$$
Therefore we have
\begin{equation}\label{eq:Q-G}
\|Q_n-G_n\|_{\mathrm{TV}}\geq2|Q_n(\mathcal{T}_n)-G_n(\mathcal{T}_n)|\to  2 \quad \text{as}~n\to \infty.
\end{equation}
However, by triangle inequality, we combine \eqref{eq:F-G} and \eqref{eq:F-Q} and obtain
$$
\|Q_n-G_n\|_{\mathrm{TV}}\leq\|Q_n-F_n\|_{\mathrm{TV}}+\|F_n-G_n\|_{\mathrm{TV}}\to  0\quad \text{as}~n\to \infty.
$$
This contradicts \eqref{eq:Q-G}. Hence, there is no $t_0\neq 0$ such that
$$
\mathbb{E}e^{-t_0X}=e^{\beta t_{0}^2/2}.
$$
This concludes the proof.
\end{proof}

\begingroup
\let\raggedright\centering
\part{CLT in R\'enyi divergence of order $\alpha=\infty$}
\endgroup


\section{Laplace transform estimates}\label{sec:Laplace transform estimates}
We first establish the entire extension of the Laplace transform $L(t)=\mathbb{E}e^{tX}$ and the real analyticity
and convexity of $K(t)=\log L(t)$. Under the condition $L(t)\leq e^{t^2/2}$,
we then derive estimates for the first two derivatives of
$A(t)=t^2/2-K(t)$. Finally, we compare
$e^{-iuK'(t)}L(t+iu)/L(t)$ with its quadratic approximation
$1-K''(t)u^2/2$ and Gaussian approximation $e^{-K''(t)u^2/2}$.
The resulting error bounds hold for every $u\in\mathbb R$, with
constants independent of $t$ whenever $A(t)\leq1$.

\begin{lem}\label{lem:entire-mgf}
If $L(t)<\infty$ for every $t\in \mathbb{R}$, then
\begin{enumerate}
\item  $L(z)=\E e^{zX}$ is entire and
$$
L^{(m)}(z)=\E[X^m e^{zX}],\qquad m\ge0.
$$
\item $K(t)=\log L(t)$ is real analytic, and $K''(t)>0$ for every $t\in \mathbb{R}$.
\end{enumerate}
\end{lem}

\begin{proof}
(1) The proof that $L$ is entire is the same as that of Lemma \ref{lem:phi-phi_n-extension}. Hence, for any $z \in \mathbb{C}$, we have \begin{equation}\label{eq:mgf-power-series}
L(z)=\mathbb{E}\left[\sum_{n=0}^{\infty}\frac{z^nX^n}{n!}\right]=\sum_{n=0}^{\infty}\frac{\mathbb{E}X^n}{n!}z^n.
\end{equation}
Moreover, the power series \eqref{eq:mgf-power-series} may be differentiated term by term. Therefore, for every integer $m\geq 0$,
\begin{equation}\label{eq:derivative-expansion}
L^{(m)}(z)=\sum_{n=m}^{\infty}\frac{\mathbb{E}X^n}{(n-m)!}z^{n-m}=
\sum_{n=0}^{\infty}
\frac{\mathbb{E}X^{n+m}}{n!}z^n.
\end{equation}
For any fixed $R>0$ and all $|z|\leq R$ and $n\in \mathbb{Z}_{\geq 0}$, we have
$$
\left|\sum_{\ell=0}^{n}	\frac{X^{\ell+m}z^\ell}{\ell!}\right|\leq|X|^m\sum_{\ell=0}^{n}\frac{R^\ell|X|^\ell}{\ell!}\leq |X|^m e^{R|X|}.
$$
For every $x\geq0$, we have $x^m\leq m!\,e^{x}$. Hence
$$
\mathbb{E}\left[|X|^m e^{R|X|}\right]\leq m!\,\mathbb{E}e^{(R+1)|X|}\leq m!\left[L(R+1)+L(-(R+1))\right]<\infty.
$$
By Lebesgue's dominated convergence theorem and \eqref{eq:derivative-expansion}, we have
\begin{equation}\label{eq:entireder}
L^{(m)}(z)=\mathbb{E}\left[	\sum_{k=0}^{\infty}\frac{X^{k+m}z^k}{k!}\right] =	\mathbb{E}\bigl[X^m e^{zX}\bigr], \quad m\in \mathbb{Z}_{\geq 0}.
\end{equation}
(2) Since $L$ is entire and $L(t)>0$ for every $t\in\mathbb{R}$, $K(t)=\log L(t)$ is real analytic on $\mathbb{R}$. By \eqref{eq:entireder}, for any $t \in \mathbb{R}$, we have 
\begin{equation}\label{eq:2nd-derivative}
K'(t) =\frac{\mathbb{E}[Xe^{tX}]}{\mathbb{E}e^{tX}}, \qquad
K''(t) =\frac{\mathbb{E}[X^2e^{tX}]\,\mathbb{E}e^{tX}-\bigl(\mathbb{E}[Xe^{tX}]\bigr)^2}{\bigl(\mathbb{E}e^{tX}\bigr)^2}\ge0.
\end{equation}
The inequality follows from the Cauchy--Schwarz inequality. Equality would force $X$ to be constant almost surely, contradicting $\operatorname{Var}(X)=1$. Therefore, inequality in \eqref{eq:2nd-derivative} has to be strict.
\end{proof}

\begin{lem}\label{lem:A-bound}
Suppose that $L(t)\leq e^{t^2/2}$ for all $t\in\mathbb{R}$, then
$$
A''(t)\leq 1,	\qquad	A'(t)^2\leq 2A(t).
$$
\end{lem}

\begin{proof}
Since $L(t)\leq e^{t^2/2}$ for every $t\in\mathbb{R}$ and , by \eqref{eq:2nd-derivative}, we have
$$
A(t)=\frac{t^2}{2}-\log L(t)\geq 0 \quad \text{and} \quad A''(t)=1-K''(t)\leq 1. 		
$$
Fix $t\in\mathbb{R}$. For any $h\in\mathbb{R}$,  Taylor expansion with Lagrange remainder gives 
\begin{align*}
0\leq A(t+h)=A(t)+hA'(t)+\frac{h^2}{2}A''(\xi)\leq
A(t)+hA'(t)+\frac{h^2}{2},
\end{align*}
where  $\xi$ lies between $t$ and $t+h$. Taking $h=-A'(t)$, we obtain 
$$
A'(t)^2\leq 2A(t).
$$
\end{proof}

\begin{rmk}\label{rem:curvature}
By the preceding discussion, if $L(t)\leq e^{t^2/2}$, then $K$ is real analytic, and consequently $A$ is real analytic as well. In particular, both $A$ and $A''$ are continuous. We now show that, under condition \eqref{eq:SG}, \eqref{eq:C} is equivalent to \eqref{eq:C1} and \eqref{eq:C2}.

Assume first that \eqref{eq:C} holds. Let $t_0$ satisfy $A(t_0)=0$. Since $A\geq 0$, the point $t_0$ is a global minimum of $A$, and hence $A''(t_0)\geq 0$. 
On the other hand, \eqref{eq:C} implies $A''(t_0)\leq0$. Thus $A''(t_0)=0$, proving \eqref{eq:C1}. Next, let $|t_j|\to\infty$ with $A(t_j)\to0$. By \eqref{eq:C}, $\limsup_{j\to\infty} A''(t_j)\leq\lim_{A(t)\to 0}(A''(t))_{+}=0$, which proves \eqref{eq:C2}.

Conversely, suppose that \eqref{eq:C1} and \eqref{eq:C2} hold, while \eqref{eq:C} fails. Then there exist $\varepsilon>0$ and a sequence $\{t_j\}_{j=1}^{\infty}$ such that
$$
A(t_j)\to0,\qquad A''(t_j)\geq\varepsilon.
$$
If $\{t_j\}_{j=1}^{\infty}$ has a bounded subsequence, then, after passing to a further subsequence, $t_j\to t_0$ for some $t_0\in\mathbb R$. By continuity of $A$, we have $A(t_0)=0$, and hence \eqref{eq:C1} gives $A''(t_0)=0$. However, continuity of $A''$ yields
$A''(t_0)=\lim_{j\to\infty}A''(t_j)\geq\varepsilon$, a contradiction.	Otherwise $\{t_j\}_{j=1}^{\infty}$ has no bounded subsequence, so $|t_j|\to\infty$. Since $A(t_j)\to0$, condition \eqref{eq:C2} gives $\limsup_{j\to\infty}A''(t_j)\leq0$
contradicting $A''(t_j)\geq\varepsilon$. Therefore \eqref{eq:C} holds.
\end{rmk}

The approximation of $e^{-iuK'(t)}L(t+iu)/L(t)$ is crucial in establishing an estimate of the ratio $p_{n, k}/q_{n, k}$ in the subsequent section. As mentioned before, $L(t+iu)/L(t)$ is the characteristic function of an exponentially tilted random variable $X_t$ with mean $K'(t)$ and variance $K''(t)$. Hence, $e^{-iuK'(t)}L(t+iu)/L(t)$ is the characteristic function of the centered random variable $X_t-K'(t)$. The following lemma gives a uniform  estimate for this ratio, together with the corresponding Gaussian approximation, for all $u\in\mathbb R$ and all $t$ satisfying $A(t)\leq1$.

\begin{lem}\label{lem:uniform-characteristic-expansion}
Suppose that $L(t)\le e^{t^2/2}$ for all $t\in\mathbb{R}$. Then there exists an absolute constant $C<\infty$ such that, for all $t\in\mathbb R$ satisfying $A(t)\le 1$ and all $u\in\mathbb R$,
\begin{equation}\label{eq:characteristic-expansion}
\left|e^{-iuK'(t)}\frac{L(t+iu)}{L(t)}-1+\frac12K''(t)u^2\right|\le C|u|^3. 
\end{equation}
Consequently,
\begin{equation}	\label{eq:uniform-characteristic-expansion 1}
\left|e^{-iuK'(t)}\frac{L(t+iu)}{L(t)}-e^{-K''(t)u^2/2}\right|\leq C|u|^3+\frac{1}{8}K''(t)^{2}u^4.
\end{equation}
\end{lem}

\begin{proof}
For $t\in\mathbb R$, we define an exponentially tilted random variable $X_{t}$ with distribution $$
\mathbb{P}(X_t=x)=\frac{e^{tx}}{L(t)}\cdot\mathbb P(X=x), \quad x\in a+h\mathbb{Z}.
$$
Note that
$$
\mathbb EX_t=K'(t),	\qquad \Var(X_t)=K''(t).
$$
Setting $Y_t=X_t-K'(t)$, we have 
\begin{equation}\label{eq:1-2-moments}
 \mathbb EY_t=0, \quad \mathbb EY_t^2=K''(t),\quad \mathbb E e^{iuY_t}=e^{-iuK'(t)}\frac{L(t+iu)}{L(t)}.
\end{equation}	
Recall that 
$$
K(t)=\log L(t) \leq \frac{1}{2} t^{2},\quad A(t)=\frac{1}{2} t^{2}-K(t), \quad A'(t)=t-K'(t).
$$
Moreover, for all  $t\in\mathbb R$ satisfying
$A(t)\le1$, by Lemma \ref{lem:A-bound}, $|A'(t)|\le\sqrt2$. Then for $s\in\mathbb R$, we have 
\begin{align*}
\log\mathbb E e^{sY_t}
&=\log \frac{L(t+s)e^{-sK'(t)}}{L(t)} \\
&=	K(t+s)-K(t)-sK'(t)\\
&\leq \frac{1}{2}(t+s)^{2}-\left[\frac{1}{2}t^{2}-A(t)\right]-s[t-A'(t)]\\
&=\frac{s^2}{2}+A(t)+sA'(t)\\
&\le\frac{s^2}{2}+1+\sqrt2|s|.
\end{align*}
Taking $s=\pm1$, we can obtain $\log\mathbb E e^{Y_t}\leq 3/2+\sqrt{2}$ and $\log\mathbb E e^{-Y_t}\leq 3/2+\sqrt{2}$. And since $|x|^{3}\leq 6e^{|x|}$, we have 
\begin{equation}\label{eq:Y-3rd-momemt}
\mathbb E|Y_t|^3 \leq 6\mathbb E e^{|Y_t|}\leq 6\mathbb E e^{Y_t}+6\mathbb E e^{-Y_t}\leq 12e^{3/2+\sqrt2}.
\end{equation}
Applying Taylor's formula with integral remainder and using \eqref{eq:1-2-moments} and \eqref{eq:Y-3rd-momemt}, we obtain
\begin{align}
&\left|e^{-iuK'(t)}\frac{L(t+iu)}{L(t)}-1+\frac12K''(t)u^2\right|\notag\\
=&\left|\mathbb E e^{iuY_t}-1+\frac12K''(t)u^2\right|\notag\\
=&\left|\mathbb E \left[1+iuY_t -\frac{u^2}{2}Y_t^2-\frac{i}{2}\int_{0}^{uY_{t}} (uY_{t}-s)^{2} e^{is}\,ds\right]-1+\frac12K''(t)u^2\right|\notag\\
=&\left|1-\frac12K''(t)u^2+ \mathbb E \left[\frac{i(uY_{t})^{3}}{2}\int_{0}^{1} (1-r)^{2} e^{iruY_{t}}\,dr\right]-1+\frac12K''(t)u^2\right|\notag\\
\leq&~\mathbb E\left[\frac{|u Y_{t}|^3}{2}\cdot \frac{1}{3}\right]\leq 2e^{3/2+\sqrt2}|u|^3. \label{eq:uniform-characteristic-expansion2}
\end{align}
Furthermore, combining $0\le e^{-x}-1+x\le\frac{x^2}{2}$ for $x\ge0$ and $K''(t)\geq0$, we have 
\begin{equation}\label{eq:K''-expansion}
	0\leq e^{-K''(t)u^2/2}-1+\frac{1}{2}K''(t)u^2\leq \frac{1}{8}K''(t)^{2}u^4.
\end{equation}
By \eqref{eq:uniform-characteristic-expansion2} and \eqref{eq:K''-expansion} with the triangle inequality, we have  
\begin{align*}
	&\left|e^{-iuK'(t)}\frac{L(t+iu)}{L(t)}-e^{-K''(t)u^2/2}\right|\\
	\leq&\left|e^{-iuK'(t)}\frac{L(t+iu)}{L(t)}-1+\frac{1}{2}K''(t)u^2\right|+\left|e^{-K''(t)u^2/2}-1+\frac{1}{2}K''(t)u^2\right| \notag \\
	\leq& 2e^{3/2+\sqrt2}|u|^3+\frac{1}{8}K''(t)^{2}u^4.\notag
\end{align*}
\end{proof}


\section{Uniform asymptotics for $p_{n,k}/q_{n,k}$ }\label{sec:pointwise asymptotics}
In this section, we start with the shifted Fourier inversion formula \eqref{eq:inversion 1} for $p_{n,k}$, and derive, under conditions \eqref{eq:SG}, \eqref{eq:C} and \eqref{eq:U}, a uniform asymptotic estimate of $p_{n,k}/q_{n,k}$
in the region where $A(x_{n,k}/\sqrt n)$ is small.

\begin{lem}\label{lem:p/q-ratio}
Suppose condition \eqref{eq:SG} holds. For every $n\ge1$, $k\in\Z$, we have
\begin{equation}\label{eq:p/q-ratio}
\frac{p_{n,k}}{q_{n,k}}
=\frac{\sqrt{2\pi n}}{h}e^{-nA(x_{n,k}/\sqrt{n})}I_n\left(\frac{x_{n,k}}{\sqrt{n}}\right),
\end{equation}
where
\begin{equation}\label{eq:I}
I_n(t)=\frac{h}{2\pi}\int_{-\pi/h}^{\pi/h}
e^{-intu}\left[\frac{L(t+iu)}{L(t)}\right]^n\,du.
\end{equation}
\end{lem}

\begin{proof}
Under condition\eqref{eq:SG}, Lemma \ref{lem:entire-mgf} shows that $L$ is an entire function on $\mathbb{C}$. Since $\phi(t)=\mathbb{E} e^{itX} = L(it)$ and $\phi_{n}(t)=\phi(t/\sqrt n)^n$, the characteristic functions $\phi$ and $\phi_{n}$ of $X$ and $S_{n}$, respectively, are also entire. Then, Lemmas \ref{lem:periodicity of F_n} and \ref{lem:+iy} hold, and hence, \eqref{eq:inversion 1} is applicable. Therefore, for every $n\ge1$, $k\in\Z$ and $y\in\R$, we have
$$
p_{n,k}=e^{yx_{n,k}}\cdot
\frac{\delta_n}{2\pi}\int_{-\pi/\delta_n}^{\pi/\delta_n}
e^{-itx_{n,k}}\phi\big((t+iy)/\sqrt{n}\big)^{n}\,dt.
$$
Setting $y=-x_{n,k}$ and using $\phi(t)=L(it)$, we obtain 
\begin{align*}
p_{n,k}&=e^{-x_{n,k}^{2}} L\left(\frac{x_{n,k}}{\sqrt{n}}\right)^{n}
\cdot\frac{\delta_n}{2\pi}\int_{-\pi/\delta_n}^{\pi/\delta_n}
e^{-itx_{n,k}}\left[\frac{L\big((x_{n,k}+it)/\sqrt{n}\big)}{L\big(x_{n,k}/\sqrt{n}\big)}\right]^n\,dt \notag\\
&=\exp\left[-\frac{x_{n,k}^{2}}{2}-nA\left(\frac{x_{n,k}}{\sqrt{n}}\right)\right]\cdot\frac{h}{2\pi}\int_{-\pi/h}^{\pi/h}
e^{-i\sqrt{n}sx_{n,k}}\left[\frac{L\big(x_{n,k}/\sqrt{n}+is\big)}{L\big(x_{n,k}/\sqrt{n}\big)}\right]^n\,ds. 
\end{align*}
The second identity uses the fact that $A(t)=t^{2}/2-\log L(t)$ and the change of variables $s=t/\sqrt{n}$. 
Combining $q_{n,k}=\delta_{n}\varphi(x_{n,k})$, we obtain \eqref{eq:p/q-ratio}.
\end{proof}

\begin{rmk} We shall use the following two estimates repeatedly.
\begin{equation}\label{eq:L/L-bound}
 \frac{|L(t+iu)|}{L(t)}=\frac{|\mathbb{E} e^{(t+iu)X}|}{\mathbb{E} e^{tX}}\leq \frac{\mathbb{E} \big[e^{tX}|e^{iuX}|\big]}{\mathbb{E} e^{tX}}=1,
\end{equation}
\begin{equation}\label{eq:I_n-bound}
|I_{n}| \leq \frac{h}{2\pi}\int_{-\pi/h}^{\pi/h} |e^{-intu}|\cdot 1 \,du=1.
\end{equation}		
\end{rmk}

\begin{prop} \label{prop:uniform-integral-estimate}
Suppose that \eqref{eq:SG}, \eqref{eq:C} and \eqref{eq:U} hold. Then there exists constant $\delta>0$ such that, for every integer $n\ge2$ and $t \in \mathbb{R}$  satisfying $A(t)\le\delta$,
\begin{equation}\label{eq:uniform-integral-estimate}
I_n(t)=\frac{h}{\sqrt{2\pi nK''(t)}}\exp\left[-\frac{nA'(t)^2}{2K''(t)}\right]+O\Big(\frac{1}{n}\Big).
\end{equation}
Consequently, for every $k\in\mathbb Z$ satisfying $A(x_{n,k}/\sqrt n)\le\delta$,
\begin{equation}\label{eq:uniform-lattice-ratio}
\frac{p_{n,k}}{q_{n,k}}=\frac{\exp\left[-nA(x_{n,k}/\sqrt n)-\dfrac{nA'(x_{n,k}/\sqrt n)^2}{2K''(x_{n,k}/\sqrt n)}\right]}{\sqrt{K''(x_{n,k}/\sqrt n)}}+O\left(\frac{1}{\sqrt n}\right).
\end{equation}
\end{prop}

\begin{proof}
For  $t, u \in \mathbb{R}$, denote
\begin{equation}\label{eq:psi_t}
\psi_t(u)=e^{-iuK'(t)}\frac{L(t+iu)}{L(t)}.
\end{equation}
Note that $|\psi_t(u)|\leq 1$. And by \eqref{eq:2nd-derivative}, we know $K''(t) >0$.

\textbf{Uniform estimates near $u=0$.} By Lemma \ref{lem:uniform-characteristic-expansion}, there exists a constant $C\ge1$, such that, for $t \in \mathbb R$ satisfying $A(t)\leq 1$ and $u\in\mathbb R$,
\begin{equation}\label{eq:psi-local-expansion}
\left|\psi_t(u)-1+\frac{1}{2}K''(t)u^2\right|\le C|u|^3.
\end{equation}
Taking $u=1$, by triangle inequality and $|\psi_t(1)|\le1$, we obtain
\begin{equation}\label{eq:sigma upper bound}
K''(t)\le2\left|\psi_t(1)-1+\frac{1}{2}K''(t)\right|+2|\psi_t(1)-1|\le 2C+4.
\end{equation}
Choose $0<\eta<\min\left\{1, \pi/h\right\}$ sufficiently small so that $(2C+4)\eta^2\le1$ and $C\eta\le\frac{1}{8}$.
For $A(t)\le 1$ and $|u|\le\eta\leq 1$, we combine \eqref{eq:uniform-characteristic-expansion 1} and \eqref{eq:sigma upper bound} to obtain  
\begin{equation}\label{eq:psi-one-step-bound}
\left|\psi_t(u)-e^{-K''(t)u^2/2}\right|
\leq C|u|^3+\frac{1}{8}K''(t)^{2}u^4\leq C_1|u|^3,
\end{equation}
where $C_1=C+(2C+4)^2/8=C+(C+2)^2/2$.

And by \eqref{eq:C}, we may choose $\delta_{*}\in(0,1]$, such that $A''(t)\le\frac{1}{2}$
whenever $A(t)\le\delta_{*}$. Then $K''(t)=1-A''(t)\geq \frac{1}{2}$, and hence we can also establish the upper bounds for $\psi_{t}(u)$ and $e^{-K''(t)u^2/2}$. For $A(t)\le\delta_{*}$ and $|u|\le\eta$, we have $1-K''(t)u^2/2\geq 1-(2C+4)\eta^2/2\geq1/2$. Combining \eqref{eq:psi-local-expansion} and the triangle inequality, we obtain 
\begin{align}
|\psi_t(u)|&\le1-\frac{1}{2}K''(t)u^2+C|u|^3\notag\\
&\le1-\frac{1}{4}u^2+C\eta u^2	\notag\\
&\le 1-\frac{1}{8}u^2\le e^{-u^2/8},
\label{eq:psi-local-decay}
\end{align}
and 
\begin{equation}\label{eq:sigma-local-decay}
e^{-K''(t)u^2/2}\le e^{-u^2/4}\le e^{-u^2/8}.
\end{equation}

\textbf{The integral estimate.} For above fixed $\eta$, the assumption \eqref{eq:U} provides constants $\widetilde{\delta}>0$ and $\rho\in(0,1)$ such that
\begin{equation}\label{eq:psi-far-frequency-bound}
|\psi_t(u)|\le\rho,	\qquad A(t)\le\widetilde{\delta},\quad \eta\le|u|\le\frac{\pi}{h}.
\end{equation}
Set 
\begin{equation}\label{eq:delta}
\delta=\min\left\{\delta_{*},\widetilde{\delta}\right\}. 
\end{equation}
Then all the preceding estimates hold whenever $A(t)\le\delta$. Suppose $n\ge2$ and  $A(t)\le\delta$. Since
$A'(t)=t-K'(t)$, we have
$$
I_n(t)=\frac{h}{2\pi}\int_{-\pi/h}^{\pi/h}e^{-intu}\left(\frac{L(t+iu)}{L(t)}\right)^n\,du=\frac{h}{2\pi}\int_{-\pi/h}^{\pi/h}e^{-inuA'(t)}\psi_t(u)^n\,du.
$$
Combining the Gaussian integral
$$
\frac{h}{2\pi}\int_{\mathbb R}e^{-inuA'(t)}e^{-nK''(t)u^2/2}\,du=\frac{h}{\sqrt{2\pi nK''(t)}}\exp\left[-\frac{nA'(t)^2}{2K''(t)}\right],
$$
and $|e^{-inuA'(t)}|=1$, we have 
\begin{align}\label{eq:I-integral-decomposition}
\left|I_n(t)-\frac{h}{\sqrt{2\pi nK''(t)}}\exp\left[-\frac{nA'(t)^2}{2K''(t)}\right]\right|
&\le\frac{h}{2\pi}\int_{|u|\le\eta}\left|\psi_t(u)^n-e^{-nK''(t)u^2/2}\right|\,du\notag\\
&\quad+\frac{h}{2\pi}\int_{\eta<|u|\le\pi/h}|\psi_t(u)|^n\,du\notag\\
&\quad +\frac{h}{2\pi}\int_{|u|>\eta}e^{-nK''(t)u^2/2}\,du.
\end{align}

\textbf{(1) Estimate of the first term in \eqref{eq:I-integral-decomposition}.} For $|u|\le\eta$, combining \eqref{eq:psi-one-step-bound}, \eqref{eq:psi-local-decay} and \eqref{eq:sigma-local-decay}, we have 
\begin{align}
\left|\psi_t(u)^n-e^{-nK''(t)u^2/2}\right|
&=\left|\big(\psi_t(u)-e^{-K''(t)u^2/2}\big)\sum_{j=0}^{n-1}\psi_t(u)^{j}e^{-(n-1-j)K''(t)u^2/2}\right| \notag \\
&\leq\left|\psi_t(u)-e^{-K''(t)u^2/2}\right|\sum_{j=0}^{n-1}|\psi_t(u)|^{j}e^{-(n-1-j)K''(t)u^2/2}\notag \\
&\le C_1|u|^3 \sum_{j=0}^{n-1}e^{-ju^2/8}e^{-(n-1-j)u^2/8}\notag \\
&= C_1|u|^3\cdot n e^{-(n-1)u^2/8} \leq C_1n|u|^3e^{-nu^2/16}. \label{eq:psi-n-sigma}
\end{align}
Here, the last inequality follows from $n-1\ge n/2$ for $n\geq 2$. After the change of variables $v=\sqrt n\,u$ and direct computation, we obtain
\begin{align}
\int_{|u|\le\eta}\left|\psi_t(u)^n-e^{-nK''(t)u^2/2}\right|\,du
&\le C_1n\int_{\mathbb R}|u|^3e^{-nu^2/16}\,du\notag\\
&=\frac{C_1}{n}\int_{\mathbb R}|v|^3e^{-v^2/16}\,dv\notag\\
&=\frac{256C_1}{n}.\label{eq:first-term}
\end{align}

\textbf{(2) Estimate of the second term in \eqref{eq:I-integral-decomposition}.} By \eqref{eq:psi-far-frequency-bound}, we have
\begin{equation}\label{eq:second-term}
\int_{\eta<|u|\le\pi/h}|\psi_t(u)|^n\,du\le\frac{2\pi}{h}\rho^n.
\end{equation}

\textbf{(3) Estimate of the third term in \eqref{eq:I-integral-decomposition}.} By $K''(t)\geq1/2$ and the standard Gaussian tail bound, we have 
\begin{equation}\label{eq:third-term}
\int_{|u|>\eta}e^{-nK''(t)u^2/2}\,du\le2\int_\eta^\infty e^{-nu^2/4}\,du \le2\int_\eta^\infty \frac{u}{\eta} e^{-nu^2/4}\,du=\frac{4}{\eta n}e^{-n\eta^2/4}. 
\end{equation}
Combining \eqref{eq:I-integral-decomposition}, \eqref{eq:first-term}, \eqref{eq:second-term} and 
\eqref{eq:third-term}, we can obtain
\begin{equation}\label{eq:uniform-integral-estimate 1}
\left|I_n(t)-\frac{h}{\sqrt{2\pi nK''(t)}}\exp\left[-\frac{nA'(t)^2}{2K''(t)}\right]\right|\leq \frac{128C_1h}{\pi n}+\rho^n+\frac{2h}{\pi\eta n}e^{-n\eta^2/4}.
\end{equation}
Since $\eta>0$ and $\rho\in(0,1)$ are fixed, as $n\to\infty$,
the right-hand side of \eqref{eq:uniform-integral-estimate 1} is $O(n^{-1})$, uniformly over
$A(t)\leq\delta$. Then we get \eqref{eq:uniform-integral-estimate}.

\textbf{The lattice probability ratio.} For every $k\in\mathbb Z$ satisfying $A(x_{n,k}/\sqrt n)\le\delta$, we combine \eqref{eq:p/q-ratio} and \eqref{eq:uniform-integral-estimate 1} to get 
\begin{align*}
\frac{p_{n,k}}{q_{n,k}}
&=	\frac{\sqrt{2\pi n}}{h}e^{-nA(x_{n,k}/\sqrt n)}\left[\frac{h}{\sqrt{2\pi nK''(x_{n,k}/\sqrt n)}}\exp\left[-\frac{nA'(x_{n,k}/\sqrt n)^2}{2K''(x_{n,k}/\sqrt n)}\right]+O\left(\frac1n\right)\right]\\
&=\frac{\exp\left[-nA(x_{n,k}/\sqrt n)-\dfrac{nA'(x_{n,k}/\sqrt n)^2}{2K''(x_{n,k}/\sqrt n)}\right]}{\sqrt{K''(x_{n,k}/\sqrt n)}}+O\left(\frac{1}{\sqrt n}\right).
\end{align*}
\end{proof}

\begin{rmk} \label{rem:explanation-concentration}
The uniform separation \eqref{eq:U} condition is used only to bound the $n$-th moment of $\psi_t(u)$ in \eqref{eq:second-term} uniformly with respect to $t$ such that $A(t)<\delta$. It may appear to be a purely technical assumption, but this is not the case. The function $\psi_t(u)$ defined in \eqref{eq:psi_t} is the characteristic function of the exponentially tilted random variable $X_t-K'(t)$, where the distribution of $X_{t}$ is given by
$$
\mathbb{P}(X_t=x)=\frac{e^{tx}}{L(t)}\cdot\mathbb P(X=x).
$$
If the \eqref{eq:U} condition fails, there exists a sequence $\{t_k\}_{k=1}^\infty$ such that $A(t_k)\to 0$, and the tilted random variables $\{X_{t_k}\}_{k=1}^\infty$ may be asymptotically concentrated on proper sublattices, in which case $\sup_{k}|\psi_{t_k}(u)|\approx 1$ at some nonzero frequencies (see the example in Section \ref{sec:failure US}). If this occurs, the $n$-th moment of certain $\psi_{t_k}(u)$ in \eqref{eq:second-term} would be bounded away from zero. 


Suppose that the tilted random variable $X_{t_k}$ is asymptotically concentrated on the sub-lattice $a+h_k\mathbb Z$, i.e.,  $\P(X_{t_k}\in a+h_k\mathbb Z)\to 1$ as $k\to \infty$. Here, $h_k$ is a multiple of $h$. Let $X_{t_k,1},\cdots,X_{t_k,n}$ be independent copies of $X_{t_k}$. Write $S_{t_k,n}=(X_{t_k,1}+\cdots+X_{t_k,n})/\sqrt{n}$. Then $S_{t_k,n}$ has maximal span $h/\sqrt{n}$, but asymptotically concentrates on the sub-lattice $\mathcal{L}_{k,n}=\sqrt{n}a+(h_k/\sqrt{n})\mathbb Z$, i.e., $\mathbb{P}(S_{t_k,n}\in \mathcal{L}_{k,n})\to1$ as $k\to \infty$. The definition of $X_{t_k}$ yields
$$
\mathbb{P}(S_{n}=x)=e^{-\sqrt{n}t_{k}x}L(t_k)^n\mathbb{P}(S_{t_k,n}=x).
$$
Therefore, we have
\begin{align*}
\Gamma_{n}=\sup_{j}\frac{p_{n,j}}{q_{n,j}}
&\geq \sup_{ x\in \mathcal{L}_{k,n}} \frac{\mathbb{P}(S_{n}=x)}{\frac{h}{\sqrt{ n}}\varphi(x)}\\
&=\sup_{ x\in \mathcal{L}_{k,n}}\frac{e^{-\sqrt{n}t_{k}x}L(t_{k})^n\mathbb{P}(S_{t_{k},n}=x)}{\frac{h}{\sqrt{ n}}\varphi(x)}\\
&=\sup_{ x\in \mathcal{L}_{k,n}}\frac{e^{-nA(t_{k})}\mathbb{P}(S_{t_{k},n}=x)}{\frac{h}{\sqrt{ n}}\varphi(x-\sqrt nt_k)}\\
&\geq \frac{e^{-nA(t_{k})}\sum_{ x\in \mathcal{L}_{k,n}}\mathbb{P}(S_{t_{k},n}=x)}{\frac{h}{\sqrt{ n}}\sum_{ x\in\mathcal{L}_{k,n}}\varphi(x-\sqrt{n}t_k)}\\
&\geq \frac{e^{-nA(t_{k})}\mathbb{P}(S_{t_k,n}\in \mathcal{L}_{k,n})}{\frac{h}{\sqrt{ n}}\sum_{ x\in (h_{k}/\sqrt n)\mathbb{Z}}\varphi(x)}\\
&\geq \frac{e^{-nA(t_{k})}\mathbb{P}(S_{t_k,n}\in \mathcal{L}_{k,n})}{\frac{h}{\sqrt{ n}}\big(\frac{1}{\sqrt{2\pi}}+\frac{\sqrt{n}}{h_{k}}\big)}.
\end{align*}
In the last inequality, we bound the discrete Gaussian sum $\sum_{ x\in (h_{k}/\sqrt n)\mathbb{Z}}\varphi(x)$ using the fact that $\varphi(x)$ is decreasing on $[0,\infty)$. Since $A_{t_{k}}\to 0$ and $\mathbb{P}(S_{t_k,n}\in \mathcal{L}_{k,n}) \to 1$ as $k\to \infty$ for any fixed $n$, we have 
$$
\Gamma_{n}\geq \limsup_{k\to \infty} \left(\frac{h}{\sqrt{2\pi n}}+\frac{h}{h_{k}}\right)^{-1}.   
$$
Letting $n \to \infty$, we obtain
$$
\liminf_{n \to \infty} \Gamma_{n} \geq \limsup_{k\to \infty} \frac{h_{k}}{h}\geq 2,
$$
which gives
$$
\liminf_{n \to \infty}D_\infty(S_n\|Z_{n})\ge \log 2. 
$$
Thus the entropic central limit theorem fails. This phenomenon will also be illustrated by the example in Section \ref{sec:lower bound}.
\end{rmk}


\section{A lower bound for the supremum ratio}
\label{sec:A lower bound for the maximal ratio}

Under condition \eqref{eq:SG}, we establish a lower bound for the supremum ratio $\Gamma_{n}$. The bound involves $A(t)$, $|L(t+iu)/L(t)|$, and Gaussian lattice sums, for which we obtain uniform estimates
by Poisson summation.

For $n\ge1$ and $t,u\in\R$, denote
\begin{align}\label{eq:necessity-gaussian-sum}
G_n(t,u)&=\delta_{n}\sum_{k\in\mathbb Z}\varphi(x_{n,k}-\sqrt{n}t)e^{i\sqrt nux_{n,k}}\notag\\
\end{align}
Setting $b=\sqrt{n}t$ and $u=0$ in \eqref{eq:gaussian moment remainder}, we can get, for $C_{h}=2/(1-e^{-n\pi^2/h^2})$, that
\begin{equation}\label{eq:G-0-sum}
	0<G_{n}(t,0)\leq 1+C_{h}e^{-n\pi^2/(2h^2)}.
\end{equation}
For all $t\in\mathbb R$ and $u\in [-\pi/h,\pi/h]$, combining $|e^{iunt}|=1$, \eqref{eq:gaussian moment remainder} with $b=\sqrt{n}t$ and triangle inequality, we can obtain 
\begin{align}
|G_{n}(t,u)|&=\left|\delta_{n}\sum_{k\in\mathbb Z}\varphi(x_{n,k}-\sqrt{n}t)e^{i\sqrt n ux_{n,k}}\right|\notag\\
&=\left|\delta_{n}\sum_{k\in\mathbb Z}\varphi(x_{n,k}-\sqrt{n}t)e^{i\sqrt nu(x_{n,k}-\sqrt{n}t)}
\right|	\notag	\\
&\leq e^{-nu^2/2}+C_{h}e^{-n\pi^2/(2h^2)}.\label{eq:G_n-bound}
\end{align}
And for all $t \in \mathbb{R}$, $|G_{n}(t,u)|\leq G_{n}(t,0)$.

\begin{lem}
\label{lem:test}
Suppose that \eqref{eq:SG} holds. Then, for every integer $n\ge1$ and $t,u\in\R$,
\begin{equation}\label{eq:test}
\Gamma_n\ge
\frac{e^{-nA(t)}\left(1+\left|L(t+iu)/L(t)\right|^n\right)}
{G_n(t,0)+|G_n(t,u)|}.
\end{equation}
\end{lem}
\begin{proof}
For $t\in\mathbb R$, we define an exponentially tilted random variable $X_{t}$ with distribution $$
\mathbb{P}(X_t=x)=\frac{e^{tx}}{L(t)}\cdot\mathbb P(X=x), \quad x\in a+h\mathbb{Z}.
$$
Let $X_{t,1},\cdots,X_{t,n}$ be independent copies of $X_{t}$ and write $S_{t,n}=(X_{t,1}+\cdots+X_{t,n})/\sqrt{n}$. Then we have
$$
\mathbb{P}(S_{t,n}=x_{n,k})=\frac{e^{\sqrt{n}tx_{n,k}}p_{n,k}}{L(t)^n},\quad \mathbb{E}e^{i\sqrt{n}uS_{t,n}}=\frac{L(t+iu)^n}{L(t)^n}.
$$
For any nonnegative function $f$ on $\mathcal{L}_{n}$, we have 
\begin{align}
\Gamma_{n}=\sup_{k}\frac{p_{n,k}}{q_{n,k}}
&=\sup_{ k}\frac{e^{-\sqrt{n}tx_{n,k}}L(t)^n\mathbb{P}(S_{t,n}=x_{n,k})}{\delta_{n}\varphi(x_{n,k})}\notag\\
&=\sup_{ k}\frac{e^{-nA(t)}\mathbb{P}(S_{t,n}=x_{n,k})}{\delta_{n}\varphi(x_{n,k}-\sqrt nt)}\notag\\
&=\sup_{ k}\frac{e^{-nA(t)}f(x_{n,k})\mathbb{P}(S_{t,n}=x_{n,k})}{\delta_{n}f(x_{n,k})\varphi(x_{n,k}-\sqrt nt)}\notag\\
&\geq \frac{e^{-nA(t)}\sum_{ k}f(x_{n,k})\mathbb{P}(S_{t,n}=x_{n,k})}{\delta_{n}\sum_{ k}f(x_{n,k})\varphi(x_{n,k}-\sqrt nt)}. \label{eq:test-factor}
\end{align}
Set $f(x)=1+\operatorname{Re}\big(e^{i(\sqrt nux-\theta)}\big)$, where $\theta$ satisfies $e^{-i\theta}L(t+iu)^n=|L(t+iu)|^n$. Clearly, $0\leq f(x)\leq 2$. On the one hand, we have 
\begin{align}\label{eq:factor1}
    \sum_{k\in\mathbb Z}f(x_{n,k})\mathbb{P}(S_{t,n}=x_{n,k})&=\sum_{k\in\mathbb Z}\left[1+\operatorname{Re}\big(e^{i(\sqrt nux-\theta)}\big)\right]\mathbb{P}(S_{t,n}=x_{n,k})\notag\\
    &=\sum_{k\in\mathbb Z} \mathbb{P}(S_{t,n}=x_{n,k})+ \operatorname{Re}(e^{-i\theta}\mathbb{E} e^{i\sqrt{n}uS_{t,n}})\notag \\
    &=1+\left|\frac{L(t+iu)}{L(t)}\right|^n.
\end{align}
On the other hand, \eqref{eq:necessity-gaussian-sum} gives
\begin{align}
\delta_{n}\sum_{ k}f(x_{n,k})\varphi(x_{n,k}-\sqrt nt)
&=G_n(t,0)+\operatorname{Re}\bigl(e^{-i\theta}G_n(t,u)\bigr)\notag\\
&\le G_n(t,0)+|G_n(t,u)|. \label{eq:factor2}
\end{align}
Combining \eqref{eq:test-factor}, \eqref{eq:factor1} and \eqref{eq:factor2}, we obtain \eqref{eq:test}.
\end{proof}


\section{Proof of Theorem \ref{thm:main1}}
\label{sec:proof}

We first recall the notations that will be used throughout this section. Let $S_{n}$ be the normalized sum of i.i.d. lattice random variables $X_{1}, \cdots, X_{n}$.  For any $k \in \mathbb Z$, we set $x_{n,k}=(na+kh)/\sqrt{n}$ and $p_{n,k}=\mathbb P(S_{n}=x_{n,k})$. Then $\delta_{n}=h/\sqrt{n }$ is the maximal span of $S_{n}$. Let $\varphi(x)$ denote the standard Gaussian density and define
the quantized Gaussian random variable  $Z_n$ with distribution proportional to $\{q_{n,k}\}_{k\in\mathbb Z}$, where $q_{n,k}=\delta_{n}\varphi(x_{n,k})$. The infinite-order R\'enyi divergence of $S_{n}$ from $Z_{n}$ is defined by 
\begin{equation}\label{eq:infinite divergence}
	D_{\infty}(S_{n}||Z_{n})=\log \sup_{k \in \mathbb Z} \frac{\mathbb P(S_{n}=x_{n,k})}{\mathbb P(Z_{n}=x_{n,k})} =\log \Gamma_{n}+\log \sum_{j\in\mathbb{Z}}q_{n,j}, \quad \Gamma_{n}=\sup_{k \in \mathbb Z} \frac{p_{n,k}}{q_{n,k}}.
\end{equation}
By \eqref{eq:normalizing-const}, $\sum_{j\in\mathbb{Z}}q_{n,j}\to 1$ as $n \to \infty$. Thus $	D_{\infty}(S_{n}||Z_{n})\to 0$ is equivalent to $\Gamma_{n}\to 1$ as $n \to \infty$.

\subsection{Sufficiency }

\begin{proof}

Assume conditions \eqref{eq:SG}, \eqref{eq:C} and \eqref{eq:U}. Condition \eqref{eq:C} asserts that
$$
\lim_{A(t) \to0}(A''(t))_+=0.
$$
Therefore, for any $\varepsilon\in(0,1/2)$, there exists $\delta(\varepsilon)>0$ such that $(A''(t))_+<\varepsilon$ whenever $A(t)\le\delta(\varepsilon)$. Set $\widehat{\delta}=\min\{\delta, \delta(\varepsilon)\}$ where $\delta$ is chosen in \eqref{eq:delta}. For this fixed $\varepsilon$, $\widehat{\delta}>0$ is independent of $n$. Then we have 
\begin{equation} \label{eq:K-bound}
K''(t)=1-A''(t)>1-\varepsilon, \quad \text{whenever}\quad A(t)\le \widehat{\delta}.
\end{equation}
Invoke the probability ratio $p_{n, k}/q_{n, k}$ in \eqref{eq:p/q-ratio} as follows  
\begin{equation}\label{eq:p/q-ratio 1}
\frac{p_{n,k}}{q_{n,k}}
=\frac {\sqrt{2\pi n}}{h}e^{-nA(x_{n,k}/\sqrt{n})}I_n\left(\frac{x_{n,k}}{\sqrt{n}}\right).
\end{equation}
When $A(x_{n,k}/\sqrt{n})>\widehat{\delta}$, we obtain from \eqref{eq:p/q-ratio 1} and $|I_n\left(x_{n,k}/\sqrt{n}\right)|\leq 1$ (see \eqref{eq:I_n-bound}) that
\begin{equation}	\label{eq:suff-far}
	\frac{p_{n,k}}{q_{n,k}}
	\le\frac{\sqrt{2\pi n}}he^{-n\widehat{\delta}}.
\end{equation}
Otherwise, $A(x_{n,k}/\sqrt{n})\leq\widehat{\delta}$. Assumption \eqref{eq:SG} gives $A(t)\geq 0$ for all $t\in \mathbb{R}$. Combining with \eqref{eq:K-bound} and \eqref{eq:uniform-lattice-ratio}, we obtain the following bound with an absolute constant $C>0$,
\begin{equation}\label{eq:suff-near}
\frac{p_{n,k}}{q_{n,k}}\leq \frac{1}{\sqrt{K''(x_{n,k}/\sqrt n)}}+\frac C{\sqrt n}
\le\frac{1}{\sqrt{1-\varepsilon}}+\frac C{\sqrt n}.
\end{equation}
Recall $\Gamma_n=\sup_{k\in \mathbb{Z}}\frac{p_{n,k}}{q_{n,k}}$, Using \eqref{eq:suff-far} and \eqref{eq:suff-near}, we have 
$$
\limsup_{n\to \infty}\Gamma_n\leq \limsup_{n\to \infty} \left[\frac{\sqrt{2\pi n}}he^{-n\widehat{\delta}}+\frac{1}{\sqrt{1-\varepsilon}}+\frac C{\sqrt n}\right]=\frac{1}{\sqrt{1-\varepsilon}}.
$$
Let $\varepsilon\to 0$ and we obtain 
\begin{equation}\label{eq:sup}
\limsup_{n\to \infty}\Gamma_n\leq 1. 
\end{equation}
On the other hand, we have
$$
1=\sum_{k \in \mathbb{Z}} p_{n,k}=\sum_{k \in \mathbb{Z}} \frac{p_{n,k}}{q_{n,k}}\cdot q_{n,k}\leq \Gamma_n \sum_{k \in \mathbb{Z}}q_{n,k},
$$
where $\sum_{k \in \mathbb{Z}}q_{n,k}\to 1 $. This gives 
\begin{equation}\label{eq:inf}
	\liminf_{n\to \infty}\Gamma_n\geq 1. 
\end{equation}
Putting together \eqref{eq:sup} and \eqref{eq:inf}, we obtain  $\lim_{n\to \infty}\Gamma_n= 1$ which is equivalent to that
$$
\lim_{n\to \infty} D_{\infty}(S_{n}|\|Z_{n})=0.
$$
This completes the proof of the sufficiency part of Theorem \ref{thm:main1}.
\end{proof}

\subsection{Necessity}

\begin{proof}
Now we assume that $D_\infty(S_n\|Z_n)\to 0$, equivalently, $\Gamma_n\to1$ as $n\to \infty$.
	
\textbf{(1) Necessity of \eqref{eq:SG}.} For every $t\in\R$ and $n\in \Z_{\geq 1}$, we have
\begin{align*}
L(t)=\mathbb{E} e^{tX}=\big[\mathbb{E}e^{t\sqrt{n}S_{n}}\big]^{1/n}
&=\Big[\sum_{k\in\Z}p_{n,k}e^{t\sqrt n\,x_{n,k}}\Big]^{1/n}\\
&\le\Big[\Gamma_n\sum_{k\in\Z}q_{n,k}e^{t\sqrt n\,x_{n,k}}\Big]^{1/n}\\
&=\left[\Gamma_ne^{nt^2/2}G_{n}(t,0)\right]^{1/n}\\
&\leq\Gamma_n^{1/n}e^{t^2/2}\Big[1+C_{h}e^{-n\pi^2/(2h^2)}\Big]^{1/n}, 
\end{align*}
where $C_{h}=2/(1-e^{-n\pi^2/h^2})$, and the last two steps follow from \eqref{eq:necessity-gaussian-sum} and \eqref{eq:G-0-sum}. By assumption $\Gamma_n\to1$, we let $n \to \infty$ and obtain
$$
L(t)=\mathbb{E} e^{tX}\leq e^{t^2/2},\quad \text{for all} ~t \in \mathbb{R}.
$$

\textbf{(2) Necessity of \eqref{eq:C}.} Suppose \eqref{eq:C} fails. Then there exists $0<\varepsilon_0\leq 1$ such that for every $\delta>0$, one can find $t\in \mathbb R$ such that $A(t)\leq \delta$ and  $(A''(t))_{+}> \varepsilon_0$. Set $\varepsilon=\varepsilon_{0}/2$. For the sequence $\delta_n=n^{-2}$, there exists a sequence $t_n$ such that
\begin{equation}\label{eq:C-condition}
A(t_n)\leq \delta_n,\qquad A''(t_n)\geq \varepsilon_,\qquad 0<K''(t_n)\leq 1-\varepsilon.
\end{equation}
The bounds on $K''(t_n)$ follow from \eqref{eq:2nd-derivative} and the equation $K''(t)=1-A''(t)$. Recall that $\psi_t(u)=e^{-iuK'(t)}L(t+iu)/L(t)$
and  $|\psi_t(u)|\le1$ for $t,u\in\R$. Fix $u \in \mathbb R\setminus\{0\}$ and denote $u_{n}=u/\sqrt{n}$.  By \eqref{eq:uniform-characteristic-expansion 1}, there exists a constant $C>0$ such that
\begin{align*}
&\left|\psi_{t_n}(u_{n})^n-e^{-nK''(t_n)u_{n}^{2}/2}\right|\\
\leq&\left|\psi_{t_n}(u_{n})-e^{-K''(t_n)u_{n}^{2}/2}\right|
\sum_{j=0}^{n-1}\left|\psi_{t_n}(u_{n})\right|^j e^{-(n-1-j)K''(t_n)u_{n}^{2}/2}\notag\\
\leq&\left[C\left|\frac{u}{\sqrt{n}}\right|^3+\frac{1}{4}K''(t_{n})^{2}\left|\frac{u}{\sqrt{n}}\right|^{4}\right]\cdot n= O_{u}(n^{-1/2}).
\end{align*}
Then we get
\begin{equation}\label{eq:necessity-modulus-limit}
\left|\frac{L(t_n+iu_n)}{L(t_n)}\right|^n=|\psi_{t_n}(u_n)|^{n}=\exp\left[-\frac{K''(t_n)u^2}{2}\right]+O_u(n^{-1/2}).
\end{equation}
Since $A(t_{n})\leq n^{-2}$, we have $e^{-nA(t_{n})}\to 1$ as $n\to \infty$. Combining \eqref{eq:test}, \eqref{eq:G-0-sum}, \eqref{eq:G_n-bound}, \eqref{eq:necessity-modulus-limit} and \eqref{eq:C-condition}, we obtain
\begin{align*}\label{eq:necessity-curvature-gap}
\liminf_{n\to\infty}\Gamma_n&\ge  \liminf_{n\to\infty}  \frac{e^{-nA(t_{n})}\left(1+\left|L(t_{n}+iu_{n})/L(t_{n})\right|^n\right)}
{G_n(t_{n},0)+|G_n(t_{n},u_{n})|}\\
&\geq\liminf_{n\to\infty}  \frac{e^{-nA(t_{n})}\left(1+e^{-(1-\varepsilon)u^2/2}+O_u(n^{-1/2})\right)}
{1+e^{-u^2/2}+2C_{h}e^{-n\pi^2/(2h^2)}}\\
&=\frac{1+e^{-(1-\varepsilon)u^2/2}}{1+e^{-u^2/2}}>1.
\end{align*}
The strict inequality follows from $\varepsilon>0$ and $u\neq0$. This contradicts the hypothesis $\Gamma_n\to1$, and thus proves the necessity of condition \eqref{eq:C}.

\textbf{(3) Necessity of \eqref{eq:U} and the divergence gap.} For this step, assume only \eqref{eq:SG} and suppose that \eqref{eq:U} fails.  Then there exists $\eta_0\in(0,\pi/h)$
such that, for every $\delta>0$ and every $\rho<1$,
\begin{equation*}
	\sup_{A(t)\leq\delta}\sup_{\eta_0\leq|u|\leq\pi/h}
	\left|\frac{L(t+iu)}{L(t)}\right|>\rho.
\end{equation*}
For each integer $j\geq1$, take $\delta=1/j$ and
$\rho=1-1/j$. By the definition of the supremum, there exist
$t_j,u_j\in\mathbb R$ such that
\begin{equation}\label{eq:Ufailure}
A(t_j)\le\frac1j,\qquad\eta_{0}\le|u_j|\le\frac\pi h,\qquad\left|\frac{L(t_j+iu_j)}{L(t_j)}\right|>1-\frac1j.
\end{equation}
By \eqref{eq:G-0-sum} and \eqref{eq:G_n-bound}, for $C_{h}=2/(1-e^{-n\pi^2/h^2})$ and $\eta_{0}\le|u_j|\le\frac\pi h$, we have 
\begin{equation}\label{eq:G-boun 2}	
G_{n}(t_{j},0)\leq 1+C_{h}e^{-n\eta_{0}^{2}/2},\quad |G_{n}(t_{j},u_{j})|\leq (1+C_{h})e^{-n\eta_{0}^{2}/2}.
\end{equation}
Fix $n\ge1$. Applying \eqref{eq:test} at $t=t_j, u=u_j$, and combining \eqref{eq:Ufailure} and \eqref{eq:G-boun 2}, we obtain
\begin{align*}
\Gamma_n
&\ge\frac{e^{-nA(t_j)}\left(1+\left|L(t_j+iu_j)/L(t_j)\right|^n\right)}{G_{n}(t_{j},0)+|G_{n}(t_{j},u_{j})|}\\
&\geq \frac{e^{-nA(t_j)}(1+(1-1/j)^{n})}{1+(1+2C_{h})e^{-n\eta_{0}^{2}/2}}\xrightarrow{j\to\infty} \frac{2}{1+(1+2C_{h})e^{-n\eta_{0}^{2}/2}}.
\end{align*}
Now let $n\to\infty$ and we get 
$$
\liminf_{n\to\infty}\Gamma_n\ge2.
$$ 
Since $\sum_{k \in \mathbb{Z}}q_{n,k}\to 1 $ as $n\to \infty$, \eqref{eq:infinite divergence} gives
$$
\liminf_{n\to\infty}D_\infty(S_n\|Z_n)\ge\log2.
$$
This contradicts the convergence $D_\infty(S_n\|Z_n)\to 0$ as $n\to\infty$. Hence, condition \eqref{eq:U} is necessary.
\end{proof}


\section{Proof of Theorem \ref{thm:obstruction}}
\label{sec:example}

\subsection{Construction of the distribution}\label{sec:construction}

The distribution asserted in Theorem \ref{thm:obstruction} is supported on the set $\{-1, 0, 1\}\cup\mathcal{R}$, where
$$
\mathcal{R}=\cup_{j\ge1}\mathcal{R}_j, \qquad \mathcal{R}_j=\{\pm(b_j-1),\ \pm(b_j+1)\}, \qquad b_j=10j^2.
$$
For $x\in\mathcal R_j$, we assign the weight
\begin{equation}\label{eq:ce-remote-masses}
p_x=\frac{\omega_j}{2}e^{(-x^2+1)/2}, \quad w_j=\frac{j}{j+1}.
\end{equation}
Write
$$
r=\sum_{x\in\mathcal R}p_x,\qquad v=\sum_{x\in\mathcal R}x^2p_x.
$$
One can check that $0<r<v<1$. Then we complete the definition by setting
\begin{equation}\label{eq:ce-central-masses}
p_0=v-r,\qquad
p_1=p_{-1}=\frac{1-v}{2}.
\end{equation}
Then \eqref{eq:ce-remote-masses} and  \eqref{eq:ce-central-masses} together define a symmetric probability distribution. Let $X$ be an integer-valued random variable with this distribution. Clearly, $\mathbb E X=0$ and $\Var(X)=1$. The support of $X$ contains both $0$ and $1$. Therefore, its maximal span is one.

\subsection{Strict sub-Gaussian inequality} \label{ssec:the strict Laplace inequality}
Write
$$
\psi(t)=e^{-t^2/2}L(t),
$$
where $L(t)=\E e^{tX}$ is the Laplace transform of $X$. The strict sub-Gaussian inequality $L(t)<e^{t^2/2}$ for all $t\ne 0$ is then equivalent to $0<\psi(t)<1$ for all $t\ne 0$. We define
\begin{equation}\label{eq:gt}
g(t)=e^{-t^2/2}\cosh t.
\end{equation}
Direct calculations give
\begin{equation}\label{eq:peak-decomposition}
\psi(t)=(1-v)g(t)+(v-r)e^{-t^2/2}+\sum_{j\ge1}\omega_j\bigl[g(t-b_j)+g(t+b_j)\bigr].
\end{equation}
Since $g(t)$ is even, $\psi(t)$ is also even. It would suffice to show that $0<\psi(t)<1$ for $t>0$.


\begin{proof}	
Clearly, $(0,+\infty)$ can be partitioned as $(0,1]\cup(1,5]\cup (\cup_{j\geq 1} V_{j})$, where 
$$
V_j=\left(\frac{b_{j-1}+b_j}{2},\frac{b_j+b_{j+1}}{2}\right],\quad j\geq 1 \quad \text{and} \quad b_{0}=0.
$$

\textbf{Case 1: $0<t\le1$.}
By definition, we have 
\begin{align}
L(t) &=(v-r)+\frac{1-v}{2}\left(e^{t}+e^{-t}\right)+\sum_{x \in \mathcal{R}} p_{x}e^{tx}\notag\\
&=\Big(\sum_{x \in \mathcal{R}} x^2p_{x}-\sum_{x \in \mathcal{R}} p_{x}\Big)+\cosh t\Big(1 -\sum_{x \in \mathcal{R}} x^2p_{x}\Big)+\sum_{x \in \mathcal{R}} p_{x}\cosh(tx)\notag\\
&=\cosh t+\sum_{x\in\mathcal R}p_x\bigl[\cosh(tx)-1-x^2(\cosh t-1)\bigr].\label{eq:L-near-origin}
\end{align}
The second identity uses $\sum_{x \in \mathcal{R}} p_{x}e^{tx}=\sum_{x \in \mathcal{R}} p_{x}e^{-tx}=\sum_{x \in \mathcal{R}} p_{x}\cosh(tx)$, which follows from the symmetry of the distribution.
Note $\cosh t-1\ge t^2/2$. Taylor's formula with integral remainder gives
\begin{align*}
\cosh(tx)-1-x^2(\cosh t-1)
\le&\cosh(tx)-1-\frac{x^2t^2}{2}\\
=&\frac{x^4t^4}{6}\int_0^1(1-s)^3\cosh(stx)\,ds\\
\le&\frac{x^4t^4}{24}\cosh x. 
\end{align*}
This, together with \eqref{eq:L-near-origin}, gives
\begin{align}\label{eq:L-origin-upper}
L(t)\le\cosh t+\frac{t^4}{24}\sum_{x\in\mathcal R}p_x\cdot x^4\cosh x. 
\end{align}
We next establish an estimate of the summation in \eqref{eq:L-origin-upper}. For $x\in \mathcal R_{j}$, it is clear that $b_{j}-1\leq|x|\leq b_{j}+1$. Then we have 
\begin{align}
\sum_{x\in\mathcal R}p_x\cdot x^4\cosh x&=\sum_{j\geq 1}\sum_{x\in\mathcal R_j}p_x\cdot x^4\cosh x  \notag\\
&\le \sum_{j\geq 1}(b_{j}+1)^4 \sum_{x\in\mathcal R_j} \frac{w_j}{2}e^{(-x^{2}+1)/2} \cdot e^{b_{j}+1}  \notag\\
&\le \sum_{j\geq 1}(b_{j}+1)^4\cdot 2e^{-b_{j}^{2}/2+b_{j}} \cdot e^{b_{j}+1}  \notag\\
&=\sum_{j\geq 1}\exp\left(-b_{j}^{2}/2+2b_{j}+4\log (b_{j}+1)+\log (2e)\right)  \notag\\
&\leq \sum_{j\geq 1}e^{-b_{j}^2/8}<1. \label{eq:H-bound}
\end{align}
On the other hand, since $(2m)!\ge2^m m!$, we have
\begin{equation}\label{eq:gaussian-cosh-gap}
e^{t^2/2}-\cosh t=\sum_{m\ge2}\left(\frac{1}{2^m m!}-\frac{1}{(2m)!}\right)t^{2m}\ge\frac{t^4}{12},\qquad t\in\mathbb R.
\end{equation}
Hence, by \eqref{eq:L-origin-upper}, \eqref{eq:H-bound} and \eqref{eq:gaussian-cosh-gap}, we can obtain
$$
L(t)\le \cosh t+\frac{t^4}{24}<\cosh t+\frac{t^4}{12}<e^{t^2/2}, \qquad 0<t\le1.
$$

\textbf{Case 2: $1< t\le5$.} Recall that $g(t)=e^{-t^2/2}\cosh t$ as defined in \eqref{eq:gt}. Note $g(t)$ is decreasing on $[0, \infty)$, which follows from $g'(t)=g(t)(\tanh t-t)<0$ for $t>0$. Therefore, for any $t\geq1$, we have
\begin{equation}
g(t)\leq g(1)=e^{-1/2}\cosh 1\le\frac{23}{24}.\label{eq:g1-bound}
\end{equation}
For $|t|\ge4$, it holds 
\begin{equation}\label{eq:g-tail}
e^{-t^2/2}\leq g(t)=e^{-t^2/2}\cosh t\le e^{-t^2/2+|t|}\le e^{-t^2/4}.
\end{equation}
For $j\geq 1$, it is clear that $|t\pm b_{j}|\geq 10j-5\geq 5$. Using \eqref{eq:peak-decomposition}, \eqref{eq:g1-bound}, \eqref{eq:g-tail} and $\omega_j=j/(j+1)<1$, we can obtain for $1<t\leq5$ that
\begin{align*}
\psi(t)&=(1-v)g(t)+(v-r)e^{-t^2/2}+\sum_{j\ge1}w_{j}\bigl(g(t-b_j)+g(t+b_j)\bigr)\notag\\
&\leq (1-r)g(t)+\sum_{j\ge1}\bigl(g(t-b_j)+g(t+b_j)\bigr)\notag\\
&\leq g(1)+\sum_{j\ge1}\bigl(e^{-(t-b_j)^2/4}+e^{-(t+b_j)^2/4}\bigr)\notag\\
&\leq g(1)+2\sum_{j\ge1}e^{-(10j-5)^2/4} \notag\\
&\leq \frac{23}{24}+4e^{-25/4} <1. 
\end{align*}

\textbf{Case 3: $t > 5$.} For any fixed $t>5$, there exists $j\ge1$ such that $t\in V_j=\left((b_{j-1}+b_j)/2,
(b_j+b_{j+1})/2\right]$. For $m> j$, we have
\begin{equation}\label{eq:t+distance}
b_m-t\ge b_m-\frac{b_j+b_{j+1}}{2}=(b_m-b_{j+1})+\frac{b_{j+1}-b_j}{2}\ge 10(m-j)+10j-5\ge 5.
\end{equation}
Similarly, for $m<j$, we have
\begin{equation}\label{eq:t-distance}
t-b_m\ge \frac{b_{j-1}+b_j}{2}-b_m=\frac{b_j-b_{j-1}}{2}+(b_{j-1}-b_m)\ge 10(j-m)+10j-15\ge5.
\end{equation}
For every $m\ge0$, we have
\begin{equation} \label{eq:t+bound}
t+b_m\ge\frac{b_{j-1}+b_j}{2}+b_m=10m^2+10j^2-10j+5\ge10m+10j-5\geq 5.
\end{equation}
Define 
\begin{equation}\label{eq:remainder-definition}
\Delta_{j}(t)=\psi(t)-\omega_{j}g(t-b_{j}). 
\end{equation}
Combining \eqref{eq:peak-decomposition}, \eqref{eq:g-tail}, $\omega_m=m/(m+1)<1$ and the estimates above, we can obtain
\begin{align}
\Delta_{j}(t) &=\,(1-v)g(t)+(v-r)e^{-t^2/2}+\sum_{\substack{m\ge1\\m\ne j}}\omega_mg(t-b_m)+\sum_{m\ge1}\omega_mg(t+b_m)\notag\\
&\leq \,g(t)+\sum_{\substack{m\ge1\\m\ne j}}g(t-b_m)+\sum_{m\ge1}g(t+b_m)\notag\\
&\le \sum_{\substack{m\ge0\\m\ne j}} e^{-(t-b_m)^2/4}+\sum_{m\ge1}e^{-(t+b_m)^2/4}\notag\\
&\le \sum_{\substack{m\ge0\\m\ne j}} e^{-(10|m-j|+10j-15)^2/4}+\sum_{m\ge1}e^{-(10m+10j-5)^2/4}\notag\\
&=2\sum_{k\ge1}e^{-(10k+10j-15)^2/4}+\sum_{m\ge1}e^{-(10m+10j-5)^2/4}\notag\\
&\le 3\sum_{k\ge1}e^{-(10k+10j-15)^2/4}\notag\\
&\le 6e^{-(10j-5)^2/4}. \label{eq:psi-remainder}
\end{align}
As shown in the previous case, $g(t)$ is decreasing on $[0, \infty)$. Therefore,  $g(t)\le g(0)=1$. 
This, together with $w_{j}=j/(j+1)<1$ and \eqref{eq:psi-remainder}, gives
$$
\psi(t)=\omega_jg(t-b_j)+\Delta_{j}(t)\leq \frac{j}{j+1}+6e^{-(10j-5)^2/4}<1.
$$
\end{proof}

\subsection{Curvature condition}

\begin{proof}	
The curvature condition in Theorem \ref{thm:obstruction} asserts that
$$
\lim_{A(t)\to 0}|A''(t)|=0.
$$
As pointed out in Remark \ref{rem:Th1.2} and proved in Remark \ref{rem:curvature}, under the sub-Gaussian condition (which has been verified in the previous subsection), the above curvature condition is equivalent to
\begin{equation}\label{eq:c1}
A(t)=0\quad\Longrightarrow\quad A''(t)=0,
\end{equation}
\begin{equation}\label{eq:c2}
A(t_k)\to0, ~~ |t_k|\to\infty	~\Longrightarrow\quad |A''(t_k)|\to 0.
\end{equation}

Recall that $A(t)=t^2/2-\log L(t)=-\log\psi(t)$. The strict sub-Gaussian inequality $0<\psi(t)<1$ for all $t\ne 0$ and $\psi(0)$=1 yield that $A(t)=0$ if and only if $t=0$. This, together with $A''(0)=1-\Var(X)=0$, yields \eqref{eq:c1}. 

We next verify \eqref{eq:c2}. Since $X$ is symmetric, $\psi(t)$, and hence $A(t)$, is even. Therefore, we can assume that $t_k\to\infty$ in \eqref{eq:c2}. It is clear that
$$
A''(t_k)=\frac{\psi'(t_k)^2}{\psi(t_k)^2}-\frac{\psi''(t_k)}{\psi(t_k)}.
$$
Since $\psi(t_k)=e^{-A(t_k)}\to 1$, it suffices to show that 
\begin{equation}\label{eq:psi-derivatives}
\psi'(t_k)\to 0, \qquad \psi''(t_k)\to 0.
\end{equation}
For each $t_k>5$, there exists $j_{k}$ such that $t_{k}\in V_{j_{k}}$. As $t_{k}\to \infty$, we have $j_k\to\infty$. Invoking \eqref{eq:remainder-definition},  we have
\begin{equation} \label{eq:psi-derivative}
\psi^{(\ell)}(t_k)=\omega_{j_k}g^{(\ell)}(t_{k}-b_{j_{k}})+\Delta_{j_k}^{(\ell)}(t_k), \quad \ell=0, 1, 2.
\end{equation}

We first show that 
\begin{equation} \label{eq:g-derivative}
g^{(\ell)}(t_{k}-b_{j_{k}})\to0 \quad\text{as} \quad k \to \infty, \quad\ell=1,2.
\end{equation}
By \eqref{eq:psi-remainder}, we have $0\leq\Delta_{j_k}(t_k)\le6e^{-(10j_k-5)^2/4}\to 0$ as $j_{k}\to \infty$. Together with $\psi(t_k)\to 1$ and $\omega_{j_{k}}=j_{k}/(j_{k}+1)\to 1$, we obtain 
$$
g(t_{k}-b_{j_{k}})=\frac{\psi(t_{k})-\Delta_{j_k}(t_k)}{\omega_{j_{k}}} \to 1 \quad \text{as} \quad k \to \infty.
$$
Note $g(0)=1$ and $g$ is even and strictly decreasing on $(0,\infty)$. Hence, $g(t_{k}-b_{j_{k}})\to1$ implies that $t_{k}-b_{j_{k}}\to 0$. Statement \eqref{eq:g-derivative} then follows from the continuity of $g'(t)$, $g''(t)$ and  $g'(0)=g''(0)=0$.

We next establish upper bounds for $\Delta_{j_{k}}^{(\ell)}(t_{k})$ for $\ell=1, 2$. By \eqref{eq:peak-decomposition} and \eqref{eq:remainder-definition}, 
$$
\Delta_{j_{k}}^{(\ell)}(t)=(1-v)g^{(\ell)}(t)+(v-r)\frac{d^\ell}{dt^\ell}e^{-t^2/2}+\sum_{\substack{m\ge1\\m\ne j_{k}}}\omega_m g^{(\ell)}(t-b_m)+\sum_{m\ge1}\omega_m g^{(\ell)}(t+b_m).
$$
We first record upper bounds for $g^{(\ell)}(t)$ and $\frac{d^\ell}{dt^\ell}e^{-t^2/2}$. Recall $g(t)=e^{-t^2/2}\cosh t$. One can check that 
$$
g'(t)=e^{-t^2/2}(\sinh t-t\cosh t),	 \qquad g''(t)=e^{-t^2/2}
	\bigl(t^2\cosh t-2t\sinh t\bigr).
$$
Using $|\sinh t|\leq e^{|t|}$ and $\cosh t\leq e^{|t|}$,
we obtain
\begin{align*}    
|g^{(\ell)}(t)| &\le (1+|t|)^2e^{-t^2/2+|t|}\leq e^{-t^2/4}, \qquad |t|\ge9,\\
\left|\frac{d^\ell}{dt^\ell}e^{-t^2/2}\right| &\le (1+t^2)e^{-t^2/2}\le e^{-t^2/4}, \hspace{1.565cm} |t|\ge9.
\end{align*}
For $j_{k}\ge3$, \eqref{eq:t+distance}, \eqref{eq:t-distance} and \eqref{eq:t+bound} give $t_{k}+b_{m}\geq 10m+10j_{k}-5\geq25$ for all $m\in \mathbb{Z}_{\geq 1}$ and $|t_{k}-b_{m}|\geq 10|j_{k}-m|+10j_{k}-15\geq 15$ for $m\neq j_{k}$. Therefore, we can apply the estimates for $g^{(\ell)}(t)$ and $\frac{d^\ell}{dt^\ell}e^{-t^2/2}$ to obtain
\begin{align}\
\big|\Delta_{j_{k}}^{(\ell)}(t_{k})\big|
&\leq(1-v)e^{-t_{k}^2/4}+(v-r)e^{-t_{k}^2/4}+\sum_{\substack{m\ge1\\m\ne j}}\omega_m e^{-(t_{k}-b_m)^2/4}+\sum_{m\ge1}\omega_m e^{-(t_{k}+b_m)^2/4} \notag\\
&\leq \sum_{\substack{m\ge0\\m\ne j_{k}}}e^{-(t_{k}-b_m)^2/4}+\sum_{m\ge1}e^{-(t_{k}+b_m)^2/4}\notag\\
&\le \sum_{\substack{m\ge0\\m\ne j_{k}}} e^{-(10|m-j_{k}|+10j_{k}-15)^2/4}+\sum_{m\ge1}e^{-(10m+10j_{k}-5)^2/4}\notag\\
&\leq3\sum_{n\ge1}e^{-(10n+10j_{k}-15)^2/4}\notag\\
&\leq6e^{-(10j_{k}-5)^2/4} \to 0 \qquad\text{as} \quad k\to \infty.\label{eq:nearest-derivative-remainder}
\end{align}
Then \eqref{eq:psi-derivative}, \eqref{eq:g-derivative}, \eqref{eq:nearest-derivative-remainder} and $\omega_{j_{k}}=j_{k}/(j_{k}+1)$ yield the statement \eqref{eq:psi-derivatives}. This completes the proof of the curvature condition.
\end{proof}

\subsection{Failure of uniform separation}\label{sec:failure US}

\begin{proof}	
We will show that $A(b_j)\to 0$ and $|L(b_j+i\pi)|/L(b_j)\to 1$ as $j\to\infty$. This violates the uniform separation condition. Every nonzero support point of $X$ is odd. Hence, for any $t\in \mathbb R$, we have
\begin{equation}\label{eq:ce-parity}
L(t+i\pi)=\mathbb E e^{(t+i\pi)X}=p_{0}-\sum_{x\ne0}p_xe^{tx}=2p_0-L(t).
\end{equation}
By \eqref{eq:psi-remainder}, we have $0\leq \Delta_{j}(b_{j})\leq 6e^{-(10j-5)^2/4}\to 0 \quad \text{as} \quad j\to \infty$. Combining with \eqref{eq:peak-decomposition},  $g(t)=e^{-t^2/2}\cosh t$ and $\omega_{j}=j/(j+1)$, we obtain
\begin{equation}\label{eq:psi-bj}
\psi(b_j)=\omega_jg(0)+\Delta_j(b_{j})\to1 \quad \text{as} \quad j\to \infty.
\end{equation}
Consequently, we have
\begin{equation*}
A(b_j)=-\log \psi(b_{j})\to0,\qquad
L(b_j)=e^{b_j^2/2}\psi(b_j)\to\infty\quad \text{as} \quad j\to \infty.
\end{equation*}
This, together with \eqref{eq:ce-parity}, gives
\begin{equation*}
\frac{|L(b_j+i\pi)|}{L(b_j)}=\left|-1+\frac{2p_0}{L(b_j)}\right|\longrightarrow1 \quad \text{as} \quad j\to\infty.
\end{equation*}
\end{proof}	

\begin{rmk}
We provide a heuristic for the failure of the uniform separation condition. Firstly, recall again that $L(t+iu)/L(t)$ is the characteristic function of the exponentially tilted random variable $X_t$ with the distribution
$$
\mathbb{P}(X_{t}=x)=\frac{e^{tx}\mathbb{P}(X=x)}{L(t)}.
$$
Particularly, for $t=b_j$, we use \eqref{eq:ce-remote-masses} and \eqref{eq:psi-bj} to obtain
\begin{align*}
\mathbb{P}(X_{b_{j}}=b_{j}-1) &=\frac{\frac{1}{2}w_je^{b_{j}^2/2}}{L(b_{j})}=\frac{\frac{1}{2}w_je^{b_{j}^2/2}}{\psi(b_{j})e^{b_{j}^2/2}}\to \frac{1}{2} \quad \text{as} \quad j\to \infty,\\
\mathbb{P}(X_{b_{j}}=b_{j}+1) &= \frac{\frac{1}{2}w_je^{b_{j}^2/2}}{L(b_{j})}=\frac{\frac{1}{2}w_je^{b_{j}^2/2}}{\psi(b_{j})e^{b_{j}^2/2}}\to \frac{1}{2} \quad \text{as} \quad j\to \infty.
\end{align*}
Thus, as $j\to \infty$, the random variable $X_{b_{j}}$ is asymptotically concentrated on $\{b_{j}-1,b_{j}+1\}$, a two-point lattice with maximal span 2. If $X_{b_j}$ were a lattice random variable with span exactly 2, then Lemma \ref{lem:periodicity of F_n} would imply that $|L(b_j+iu)|$ is $\pi$-periodic. In that case, one would have exactly $|L(b_j+i\pi)|/L(b_j)=|L(b_j+i\cdot 0)|/L(b_j)=1$. The stability of this conclusion then yields $|L(b_j+i\pi)|/L(b_j)\approx 1$. Furthermore, it follows from  Remark \ref{rem:explanation-concentration} that $\liminf_{n\to \infty}D_{\infty}(S_{n}\|Z_{n})\geq \log 2$. A direct proof of this lower bound is given in the following section.
\end{rmk}

\subsection{A direct lower bound for the divergence} \label{sec:lower bound}

\begin{proof}	
Let $X_{1},\cdots,X_{n}$ be independent copies of the random variable $X$ constructed above. Set $S_{n}=(X_{1}+\cdots+X_{n})/\sqrt{n}$. Clearly, $\mathbb{E}S_{n}=0$, $\Var(S_{n})=1$, and $S_{n}$ is supported on the lattice $\mathbb Z/\sqrt{n}$ with the maximal span $1/\sqrt{n}$. Let $Z_n$ denote the quantized Gaussian random variables on $\mathbb{Z}/\sqrt{n}$. For any $k\in\mathbb Z$, we have
\begin{equation*}
\mathbb{P}\left(Z_{n}=\frac{k}{\sqrt{n}}\right)=\frac{1}{Q_{n}\sqrt{2\pi n}}e^{-k^2/(2n)} , \quad 	Q_n=\frac1{\sqrt{2\pi n}}\sum_{m\in\mathbb Z}e^{-m^2/(2n)}\xrightarrow{n\to \infty} 1.
\end{equation*}
The infinite-order R\'enyi divergence $D_\infty(S_{n}\|Z_{n})$ of $S_{n}$ from $Z_{n}$ is defined by 
$$
D_\infty(S_{n}\|Z_{n})=\log \left[\sup_{k\in \mathbb{Z}} \frac{\mathbb{P}(S_{n}=k/\sqrt{n})}{\mathbb{P}(Z_{n}=k/\sqrt{n})}\right].
$$
Now, we establish a lower bound of $D_\infty(S_n\| Z_n)$. Set
\begin{equation*}
d_n=
\begin{cases}
0,& n\ \text{is even},\\
1,& n\ \text{is odd},
\end{cases}
\qquad\ell_n=\frac{n+d_n}{2}.
\end{equation*}
For any $n,j\ge1$, consider the event that all $X_1, \cdots, X_n$ take
values in $\{b_j-1,b_j+1\}$, with exactly $\ell_n$ of them equal to
$b_j+1$. Given this event, we have $S_n=(nb_j+d_n)/\sqrt n$. Since
$$
p_{b_j+1}=\frac{\omega_j}{2}e^{-b_j^2/2-b_j},
\qquad
p_{b_j-1}=\frac{\omega_j}{2}e^{-b_j^2/2+b_j},
$$
independence gives
\begin{equation*}
\mathbb P\left(S_n=\frac{nb_j+d_n}{\sqrt n}\right)\ge\binom n{\ell_n}2^{-n}\omega_j^ne^{-nb_j^2/2-b_jd_n}.
\end{equation*}
At the same lattice point,
$$
\mathbb P\left(Z_n=\frac{nb_j+d_n}{\sqrt n}\right)=\frac{1}{Q_n\sqrt{2\pi n}}e^{-nb_j^2/2-b_jd_n-d_n^2/(2n)}.
$$
Therefore, we have
$$
e^{D_\infty(S_n\Vert Z_n)}\ge Q_n\sqrt{2\pi n}\binom n{\ell_n}2^{-n}\omega_j^ne^{d_n^2/(2n)}.
$$
For each fixed $n$, this inequality holds for every $j$. Letting $j\to\infty$ and using $\omega_j=j/(j+1)\to1$, we obtain
\begin{equation}\label{eq:ce-ratio-lower-bound}
e^{D_\infty(S_n\Vert Z_n)}\ge Q_n\sqrt{2\pi n}\binom n{\ell_n}2^{-n}e^{d_n^2/(2n)}.
\end{equation}
Finally, Stirling's formula and $\ell_n=n/2+O(1)$ give
$$
2^{-n}\binom n{\ell_n}\sim\sqrt{\frac{2}{\pi n}}.
$$
Together with $Q_{n}\to 1$ and $d_n^2/(2n)\to0$ as $n \to \infty$, we obtain
\begin{equation}\label{eq:bound-lim}
\lim_{n\to \infty}Q_n\sqrt{2\pi n}\binom n{\ell_n}2^{-n}e^{d_n^2/(2n)}=2.
\end{equation}
Thus, \eqref{eq:ce-ratio-lower-bound} and \eqref{eq:bound-lim} yield
$$
\liminf_{n\to\infty}D_\infty(S_n\Vert Z_n)\ge\log2.
$$
\end{proof}


\textbf{AI Declaration}: GPT-6 pro (OpenAI) was used to polish portions of the manuscript, derive the explicit distribution in Section \ref{sec:construction}, and develop the proof of Theorem \ref{thm:obstruction}. The authors have carefully reviewed the reasoning line by line and revised the proofs to enhance their clarity and presentation. They believe that the arguments are mathematically sound and take full responsibility for the contents of the present manuscript.


{\bf Acknowledgement.} The work was supported by the National Natural Science Foundation of China (NSFC) grant 62201175.


\bibliographystyle{plain}

\end{document}